\documentclass{article}
\usepackage[a4paper, left=2.5cm, right=3.5cm, top=30mm, bottom=35mm]{geometry}
\usepackage{blindtext}
\usepackage{amssymb}
\usepackage{amsfonts}
\usepackage{eurosym}
\usepackage{color}

\newtheorem{theorem}{Theorem}

\newtheorem{proposition}[theorem]{Proposition}
\newtheorem{remark}[theorem]{Remark}

\newenvironment{proof}[1][Proof]{\noindent\textbf{#1.} }{\ \rule{0.5em}{0.5em}}
\newcommand{\dprod}{\displaystyle\prod}

\title{On the onset of correlations in Wave Turbulence close to
singularities.}
\begin{document}
\maketitle
%\begin{center}
%{\Large On the onset of correlations in Wave Turbulence close to
%singularities.}

\begin{center}{M. Escobedo , J. J. L. Vel\'azquez }

\bigskip
\end{center}

\textbf{Abstract} 
In this paper we describe in a formal way how the derivation of the turbulent wave equation for the Schr\"odinger equation breaks down for times close to the self similar blow up of the wave turbulence kinetic equation. To this end, we study how the  derivation of the cumulants  hierarchy can not be approximated using solutions of the wave turbulence  kinetic  equation near the blow up time. It tuns out that near the blow up time the kinetic equation has to be replaced by a hierarchy of equations which is equivalent to a random field, defined for times $t\in (-\infty, \infty)$ and satisfying a nonlinear
 non autonomous  Schr\"odinger equation.

\section{Introduction}
\setcounter{equation}{0}
\setcounter{theorem}{0}

The goal of this paper is to describe the mechanism in which the derivation
of the Wave Turbulence equation associated to the nonlinear Schr\"{o}dinger
equation (WT) breaks down for times near to the blow-up time of the
solutions of the WT equation.

It was established long time ago in the physical literature that the
solutions of a large class of wave equations with weak non-linearities and
random initial data can be approximated using some classes of kinetic
equations. The earliest example of the application of the ideas of Wave
Turbulence is due to Peierls (cf. \cite{Pe1}). A general formalism that
allows to derive kinetic equations for a general class of quantum systems
with random initial data was develop in \cite{BrPr}. A kinetic theory
describing Wave Turbulence for water waves was developed by Hasselmann (cf. 
\cite{Hass1}, \cite{Hass2}). Further developments of the theory of Wave
Turbulence for water waves can be found in \cite{BN}, \cite{BS}, \cite{Z2a}.
A large class of methods were developed by Zakharov and collaborators in
order to derive Wave Turbulence theories for several physical systems (cf. 
\cite{Z2}, \cite{Z2b}, \cite{Z3} as well as in the book \cite{Zbook} and
references therein). The ideas and methods of Wave Turbulence theory had
been extensively applied in plasma physics (\cite{Ved}, \cite{ZMR}, \cite{Zas}). A
detailed list of references concerning Wave Turbulence theory can be found
in the books \cite{Zbook} and \cite{N}.

\bigskip

In the first part of this paper we will revisit the derivation of the
kinetic equation associated to the theory of Wave Turbulence for the
nonlinear Schr\"{o}dinger equation. This particular wave equation has been
often used to illustrate the methods and ideas of Wave Turbulence in a
setting in which the PDE under consideration is relatively simple. A
detailed description of the problem under consideration will be given in
Section \ref{Correlations}. Roughly speaking we will consider the classical,
defocusing cubic Schr\"{o}dinger equation with random initial data $u_{0}.$
These initial data will be selected as some families of Gaussian random
variables which are uniquely characterized by the correlation function. We
will restrict our analysis to the case in which the probability distribution
describing the choice of initial data is invariant under spatial
translations. The corresponding solution of the nonlinear Schr\"{o}dinger
defines a time dependent random field $u\left( \cdot ,t\right) .$

Rigorous mathematical results proving that this approximation is valid for a
suitable scaling limit of the nonlinear Schr\"{o}dinger equation and
suitable time scales have been recently obtained in \cite{DeHa}. Earlier
results in \cite{LuSp} provided a rigorous derivation of a linearized Wave
Turbulence equation in the kinetic time scale near the equilibrium
distribution. For technical reasons the problem considered in \cite{LuSp}
replaces the Laplacian in the Nonlinear Schr\"{o}dinger equation by a
discretized version of it.

\bigskip

In this paper we will use only formal, non-rigorous arguments. As a first
step we will revisit the derivation of the WT equation for the nonlinear Schr%
\"{o}dinger equation using the method of cumulants. Cumulants have been
extensively used in the non-rigorous derivations of \ the WT equation. The
cumulants can be thought as some kind of generalization of the classical
hierarchies of equations that are often used in kinetic theory. It turns out
that it is possible to write an infinite hierarchy of equations for the
cumulants in which the evolution of each cumulant is linked to higher order
cumulants in a way analogous to the BBGKY hierarchy for, say, the Boltzmann
equation. Using the smallness of the nonlinear interactions in the Schr\"{o}%
dinger equation it is possible to derive a perturbative series for the
cumulants, that, in particular, provides a closure mechanism for the
hierarchy of cumulant equations and it allows to derive the standard WT
kinetic equations that describes the evolution of the Fourier transform of
the correlation function that characterizes the random field $u\left( \cdot
,t\right) .$ This approach has been used to derive the kinetic WT equation
in \cite{DNPZ}, \cite{DPR}.

\bigskip

It is worth to mention that the rigorous derivation of the kinetic WT
equation in \cite{DeHa} is not based in the use of cumulants. Instead, the
approach used in \cite{DeHa} is based in rewriting the nonlinear Schr\"{o}%
dinger equation as an integral equation by means of the Duhamel formula.
Iterating that equation it is possible to obtain a power series in terms of
the small parameter $\varepsilon $ that measures the strength of the
non-linearities. An extremely involved analysis that requires a detailed
study of the combinatorics of the terms in the resulting series allows to
prove that the correlation function associated to the random field $u\left(
\cdot ,t\right) $ solves the kinetic WT equation. Derivations of Wave
Turbulence theories (in general for wave equations different from the
nonlinear Schr\"{o}dinger equation) based in cumulants can be found in \cite{New1}, \cite{New3}, \cite%
{NR} while the Duhamel approach has been used in 
\cite{BrPr} and in the rigorous approaches developed in \cite{DeHa}, \cite%
{LuSp}.

\bigskip

From the mathematical point of view the kinetic equations arising in Wave
Turbulence theory, and in particular the one associated to the cubic nonlinear Schr\"{o}dinger equation have many analogies with some kinetic equations that
appear in the description of some classes of quantum gases. Specifically,
the kinetic equation which describes the behavior of the distribution of
velocities for a rarefied gas of bosons is a kinetic equation containing
some quadratic terms and some cubic terms that are the same that appear in
the theory of Wave Turbulence for the nonlinear Schr\"{o}dinger equation.
This equation, usually termed as the Nordheim equation was first derived in 
\cite{Nord} using physical arguments analogous to the ones used in the
derivation of the classical Boltzmann equation, but replacing the classical
statistical arguments used in Boltzmann by the statistics of a set of
bosons. As indicated before the cubic terms arising in the Nordheim equation
are similar to the ones contained in the WT kinetic equation, while the
quadratic terms in the Nordheim equation are identical to those appearing in
the classical Boltzmann equation. Due to this, Nordheim equation provides
some kind of interpolation between the classical Boltzmann equation and the
kinetic behaviour of a system of weakly interacting waves, which is
described by Wave Turbulence theory.

There are some partial results concerning concerning the derivation of
Nordheim equation taking as starting point the dynamics of a system of many
quantum particles (cf. \cite{CPulv}).

\bigskip

It is known that the solutions of the WT equation for the cubic Schr\"{o}%
dinger equation blow-up in finite time for a large class of bounded (and
smooth) initial data. The possible existence of blow-up phenomena for
kinetic equations related to WT was first addressed in \cite{LY1}, \cite{LY2}%
, \cite{LY3},  and then in \cite{KSS},  \cite{S}, \cite{SK1, SK2},   \cite{JPR, LLPR}). 
It was also argued that isotropic solutions of the WT equation develop a Dirac mass at the origin after the blow up by means of an additional self similar solution. A set of equations describing a self similar formation of a Dirac mass in finite time can be found in [32], [33].
Numerical simulations performed in [32], [33], [17], [19] and more recently in [34] strongly suggest the existence of a stable self similar  blow up mechanism. A rigorous proof of the onset of blow up for the WT and Nordheim equations, as well as the formation of Dirac masses in finite time for isotropic solutions of WR as well as for the Nordheim equation, for a large class of initial data was rigorously proved in \cite{EV1, EV3}. 

A natural question is to determine if the kinetic equation remains a valid approximation for the random field $u$ for times $t$ close  to the blow up time for the kinetic WT equation. The problem of the validity of the  wave turbulence theory was   present from the beginning in the literature of the field,  \cite{{Hass1}}, \cite{BS}, Benney-Newell, \cite{New2}. It was in particular considered in  detail in \cite{New3}  where  a general criteria for the loss of validity of wave turbulence deduction was given and checked (p. 242, towards the end of p. 261, p. 262). 

 An issue related with the above concerns the  Bose Einstein condensation phenomena in the context of quantum gases of bosons as described in great generality in the literature of physics (see the tutorial \cite{Pr}). It is seems now well accepted (cf. \cite{Stoof2, Stoof3}, \cite{SK1}) that the Bose condensation is a three stages process, where the first and third are 
kinetic regimes but not the second. Kinetic descriptions would then only valid for times sufficiently separated from the actual nucleation of the condensate. For a  discussion about the end of validity of the kinetic description of the gas and the description of the formation and growth of the condensate in \cite{KSS}, \cite{{Stoof1, Stoof2, Stoof3}}, \cite{GZ} and the review \cite{DWGGP}.
That question simplifies for  dilute, spatially  homogeneous and isotropic gases,
in the case of  large occupation  numbers, weak interactions as  considered in  \cite{S}, \cite{SK1, SK2},   \cite{JPR, LLPR}. In that context, the problem may be treated in terms of a blow up of solutions of the Nordheim and formation of Dirac's delta.
An approach closely related to the one developed in this paper has been used
in \cite{EV0} in order to study the breakdown of the derivation of the
Nordheim equation for bosons taking as starting point a hamiltonian system
of interacting quantum particles. In that case, the kinetic equation can be
obtained using a hierarchy of Wigner functions assuming that the
interactions between the particles are weak enough. 
The breakdown of the  kinetic regime  happens when the variations of the solutions
of the  Nordheim equation  are large, as it happens for solutions that blow up in finite time, whose existence is proved in \cite{EV1}.

It may then happen both in the WT theory and the  BE condensation (in the simplified situation)  that,  if we assume that the blow-up time
takes place at $t=0$ (something that it is always possible using a
translation in time), the kinetic regime would provide a good approximation
for the solutions of the gas only if $\left\vert t\right\vert $ is
sufficiently large (in suitable time units). The transition between both
regimes should be given by the hierarchy of  equations for cumulants (WT equation), or Wigner functions (Nordheim's equation)  mentioned above.
\\\\

\noindent
{\bf {\Large Main Results}}\\\\
In this paper we obtain two main results. The first one is   that the formal derivation of the
WT kinetic equation through the cumulants equations breaks down near the blow-up. More precisely, we will
show that, assuming that the blow-up for the WT kinetic equation takes place
in the self-similar manner numerically observed in \cite{SK1}, \cite{SK2}, 
\cite{JPR}, \cite{LLPR}, \cite{SGMN}, the main assumptions in which the
derivation of the kinetic equation is based taking as starting point the
hierarchy of cumulants, cease being valid. The two main assumptions that are
made in the derivation of the kinetic WT equation are the smallness of the
correlation functions as well as the fact that the variations of the
solutions of the kinetic equation are sufficiently slow to ensure the
validity of a Markovian approximation for the solutions of the hierarchy of
cumulant equations. It turns out that both approximations fail for times
sufficiently close to the blow-up time, and therefore the kinetic WT
equation cannot be used anymore in order to describe the form of the random
field $u\left( \cdot ,t\right) $ for times sufficiently close to the blow-up
time. 

In a second result we will see that for the range of times near the blow up indicated above, the description of the
random field $u\left( \cdot ,t\right) $ must be made by means of hierarchy
of cumulants that, differently from the original set of cumulant equations,
does not contain any small parameter. The hierarchy  that we have obtained must be solved with an
initial condition for very negative times, which allows to match the
self-similar behavior near the blow-up time for the kinetic WT equation. We
notice that the hierarchy obtained in this paper is equivalent to a class of
random fields that satisfy the cubic nonlinear Schr\"{o}dinger   equation for
times $t\in \left( -\infty ,\infty \right) $ and become uncorrelated as $%
t\rightarrow -\infty .$\\

The plan of this paper is the following. In Section 2 we formulate precisely
the problem under consideration and we derive the full hierarchy of
cumulants associated to the problem. In Section 3 we describe how to obtain
a closure of the hierarchy using the fact that the equations contain a small
parameter $\varepsilon $ that measures the strength of the nonlinear terms.
In Section 4 we describe in detail the self-similar solutions yielding
blow-up for the isotropic version of the WT kinetic equation. These
solutions have been obtained numerically in \cite{SK1}, \cite{SK2}, \cite%
{JPR}, \cite{LLPR}. Section 5 describes how the breakdown of the closure
scheme that allows to obtain the kinetic equation taking as starting point
the hierarchy of cumulants takes place near the blow-up time. Finally in
Section 6 we briefly discuss a non-Markovian kinetic equation that strictly
speaking does not arise from an analysis of the cumulant equations, but
yields in a suitable limit the Markovian equation. It has some independent
mathematical interest and it might provide some light understanding the
limit from non-Markovian to Markovian.

\bigskip 

\section{Hierarchy of correlation functions.}
\label{Correlations}
\setcounter{equation}{0}
\setcounter{theorem}{0}
\bigskip

Our goal is to derive an effective description for the solutions of the
following problem\ 
\begin{equation}
i\partial _{t}u=-\frac{1}{2}\Delta u+\varepsilon \left\vert u\right\vert
^{2}u\ \ ,\ \ x\in \mathbb{R}^{3}\ \ ,\ \ t\in \mathbb{R}\ ,\ \ \varepsilon
>0  \label{S1E1}
\end{equation}%
\begin{equation}
u\left( x,0\right) =u_{0}\left( x,\omega \right)  \label{S1E2}
\end{equation}%
where the initial data is a random variable $u_{0}:\mathbb{R}^{3}\times
\Omega \rightarrow \mathbb{C},$ with $u_{0}\left( \cdot ,\omega \right) \in
C\left( \mathbb{R}^{3}\right) $ for $a.e.$ $\omega \in \Omega .$ Notice that
the sign of $\varepsilon $ corresponding to the so-called defocusing case
corresponds to the case of absence of blow-up in finite time (\cite{ThC}).

The initial value $u_{0}$ will be assumed to be a Gaussian variable, and
therefore, it is uniquely characterized by means of its average and
correlation function, i.e.%
\begin{equation}
\mathbb{E}\left[ u_{0}\left( x\right) \right] =0\ \ ,\ \ \mathbb{E}\left[
u_{0}^{\ast }\left( x\right) u_{0}\left( y\right) \right] =N_{0}\left(
x-y\right) \ \ ,\ \ x,y\in \mathbb{R}^{3}  \label{S1E3}
\end{equation}%
where $N_{0}=N_{0}\left( x\right) $ is a nonnegative function from $\mathbb{R%
}^{3}$ to $\mathbb{R}$ that decreases sufficiently fast as $\left\vert
x\right\vert \rightarrow \infty .$

We remark that in (\ref{S1E3}) we assume that the probability distribution
assigning the values of $u_{0}$ is invariant under translations, i.e. the
probability distribution is invariant under the change $x\rightarrow x+a$
for any $a\in \mathbb{R}^{3}.$

\bigskip 

Most likely, the solutions of (\ref{S1E1}), (\ref{S1E2}), (\ref{S1E3}) do
not exist with the degree of generality indicated in those equations. As a
matter of fact, in the rigorous derivation in \cite{DeHa} the nonlinear Schr%
\"{o}dinger equation is solved in a torus $\mathbb{T}_{L}$ with size $L$ and 
$L$ is sent to infinity as $\varepsilon \rightarrow 0.$ This allows to work
with functions $u\left( \cdot ,t\right) $ which are well defined in the
torus $\mathbb{T}_{L}.$ In any case, we will write the formal arguments of
this paper assuming that the nonlinear Schr\"{o}dinger equation is solved in
the whole $\mathbb{R}^{3}$ in order to simplify the arguments. We remark
also that in some of the rigorous works about WT, specifically in \cite{LuSp}%
, the Laplacian in the Schr\"{o}dinger equation is replaced by a discretized
Laplacian. This allows to work in a bounded set of Fourier frequencies and
it avoids difficulties associated to the fine scale of the solutions for
small wavelengths.

\bigskip 

The theory of Wave Turbulence allows to describe the dynamics of a system of
weakly non-linear waves by means of kinetic equations. In particular, in the
case of the problem (\ref{S1E1})-(\ref{S1E3}), it is possible to see
formally that $u\left( x,t\right) $ can be approximated as $\varepsilon
\rightarrow 0$ by means of a Gaussian variable with zero average, whose
correlation function evolves in times of order $\frac{1}{\varepsilon ^{2}}.$
Moreover, the evolution of the Fourier transform of the function $N\left(
x,t\right) $ which characterizes the random variable $u\left( x,t\right) $
is given by a kinetic equation. More precisely, we will denote the Fourier
transform of $N\left( x,t\right) $ in the $x$ variable as $n\left(
k,t\right) ,$ i.e.%
\[
n\left( k,t\right) =\frac{1}{\left( 2\pi \right) ^{\frac{3}{2}}}\int_{%
\mathbb{R}^{3}}N\left( x,t\right) e^{-ik\cdot x}dx 
\]

Then, we have 
\begin{eqnarray}
&&\partial _{t}n_{1}=8\pi \varepsilon ^{2}\int_{\mathbb{R}^{3}}dk_{2}\int_{%
\mathbb{R}^{3}}dk_{3}\int_{\mathbb{R}^{3}}dk_{4}
\delta \left(k_{1}+k_{2}-k_{3}-k_{4}\right)\times \nonumber\\
&&\times  \delta \left( \left\vert k_{1}\right\vert
^{2}+\left\vert k_{2}\right\vert ^{2}-\left\vert k_{3}\right\vert
^{2}-\left\vert k_{4}\right\vert ^{2}\right)\left[ \left( n_{1}+n_{2}\right) n_{3}n_{4}-\left(
n_{3}+n_{4}\right) n_{1}n_{2}\right]    \label{S1E4}
\end{eqnarray}%
where we use the standard kinetic notation%
\[
n_{j}=n\left( k_{j},t\right) \ \ ,\ \ j=1,2,3,4.
\]

As indicated in the introduction, the derivation of (\ref{S1E4}) taking as
starting point (\ref{S1E1})-(\ref{S1E3} has been obtained by many authors
with different levels of mathematical rigour. We recall here a formal
derivation of (\ref{S1E4}),  that follows the arguments of  \cite{BN}, \cite{ZMR}, \cite{DNPZ}   which will be particularly suited to describe the
onset of correlations near the blow-up of the solutions of (\ref{S1E4}). To
this end, we introduce the following family of correlation functions%
\begin{eqnarray}
&&F_{L,M}\left( x_{1},x_{2},...,x_{L};y_{1},y_{2},...,y_{M}\right) =\nonumber\\
&&=\mathbb{E}\left[ u_{0}^{\ast }\left( x_{1}\right) u_{0}^{\ast }\left( x_{2}\right)
...u_{0}^{\ast }\left( x_{L}\right); u_{0}\left( y_{1}\right) u_{0}\left(
y_{2}\right) ...u_{0}\left( y_{M}\right) \right].   \label{S1E5}
\end{eqnarray}

Notice that the functions $F_{L,M}$ are invariant under permutations of any
set of variables $x_{k}$ or $y_{k}.$
\subsection{The Cauchy problem for the hierarchy of  correlation functions}
The functions $F_{L,M}$ depend also of the $t$ variable, although this
dependence will not be made explicit for the sake of simplicity. We can
derive a set of evolution equations for the functions $F_{L,M}$ using (\ref%
{S1E1}). This will result in a hierarchy of equations in which the evolution
of $F_{L,M}$ is written in terms of $F_{L,M}$ itself as well as the
functions $F_{L+1,M+1}.$ More precisely, differentiating (\ref{S1E5}) and
using (\ref{S1E1}) we obtain
\begin{eqnarray}
&&i\partial _{t}F_{L,M}\left(
x_{1},x_{2},...,x_{L};y_{1},y_{2},...,y_{M}\right)=  \nonumber \\
&&=\frac{1}{2}\left( \sum_{j=1}^{L}\Delta _{x_{j}}-\sum_{j=1}^{M}\Delta
_{y_{j}}\right) F_{L,M}\left(
x_{1},x_{2},...,x_{L};y_{1},y_{2},...,y_{M}\right) -  \nonumber \\
&&-\varepsilon \sum_{j=1}^{L}F_{L+1,M+1}\left(
x_{1},x_{2},...,x_{L},x_{j};y_{1},y_{2},...,y_{M},x_{j}\right) +  \nonumber
\\
&&+\varepsilon \sum_{j=1}^{M}F_{L+1,M+1}\left(
x_{1},x_{2},...,x_{L},y_{j};y_{1},y_{2},...,y_{M},y_{j}\right),\,\,  t\in  \mathbb{R},\, x_i\in  \mathbb{R}^3, y_i\in \mathbb{R}^3. \label{S1E6}
\end{eqnarray}

Equations (\ref{S1E6}) are reminiscent from the set of hierarchies that can be
found in \cite{EV0}. We need to complement (\ref{S1E6}) with a set of
initial values. These can be obtained using the fact that the variables $%
u_{0}$ are Gaussian variables, combined with (\ref{S1E3}). Due to the fact
that the variables $u_{0}$ are Gaussian, we can obtain all the correlation
functions $F_{L,M}\left( x_{1},x_{2},...,x_{L};y_{1},y_{2},...,y_{M}\right) $
at time $t=0,$ in terms of the correlation functions that can be found in (%
\ref{S1E3}). To this end we use Isserlis Theorem (\cite{I}), that states that%
\[
\mathbb{E}\left[ Z_{1}Z_{2}...Z_{n}\right] =\sum_{p\in
P_{n}^{2}}\dprod\limits_{\left\{ j,k\right\} \in p}\mathbb{E}\left[
Z_{j}Z_{k}\right] 
\]%
where $P_{n}^{2}$ are all the possible ways of partitioning $\left\{
1,2,....,n\right\} $ into pairs $\left\{ j,k\right\} .$ The variables $Z_{j}$
are complex random variables distributed according to a multivariate random
normal vector. Using (\ref{S1E3}) we then obtain that%
\begin{equation}
\mathbb{E}\left[ u_{0}^{\ast }\left( x_{1}\right) u_{0}^{\ast }\left(
x_{2}\right) ...u_{0}^{\ast }\left( x_{L}\right) u_{0}\left( y_{1}\right)
u_{0}\left( y_{2}\right) ...u_{0}\left( y_{M}\right) \right] =0,\,\ \  if\,\,\,\,
L\neq M  \label{S1E7}
\end{equation}%
and%
\begin{eqnarray}
&&\mathbb{E}\left[ u_{0}^{\ast }\left( x_{1}\right) u_{0}^{\ast }\left(
x_{2}\right) ...u_{0}^{\ast }\left( x_{L}\right) u_{0}\left( y_{1}\right)
u_{0}\left( y_{2}\right) ...u_{0}\left( y_{L}\right) \right]=  \nonumber \\
&&=\sum_{p\in P_{L,L}^{2}}\dprod\limits_{\left\{ j,k\right\} \in p}\mathbb{E}%
\left[ u_{0}^{\ast }\left( x_{j}\right) u_{0}\left( y_{k}\right) \right]
=\sum_{p\in P_{L,L}^{2}}\dprod\limits_{\left\{ j,k\right\} \in p}N_{0}\left(
x_{j}-y_{k}\right)\nonumber\\
&& =\sum_{\sigma \in
S^{L}}\dprod\limits_{j=1}^{L}N_{0}\left( x_{j}-y_{\sigma \left( j\right)
}\right)  \label{S1E8}
\end{eqnarray}%
where $P_{L,L}^{2}$ are all the possible ways of pairing the elements of $%
\left\{ 1,2,...,L\right\} $ with the elements of $\left\{ 1,2,...,L\right\}
, $ i.e. the group of permutations $S_{L}$ of the elements $\left\{
1,2,...,L\right\} .$ We then obtain the following initial values for the
functions $F_{L,M}$ that must be used to solve (\ref{S1E6})%
\begin{eqnarray}
F_{L,M}\left( x_{1},x_{2},...,x_{L};y_{1},y_{2},...,y_{L};0\right) &=&0\ \ 
if \ \ L\neq M  \label{S1E9a} \\
F_{L,L}\left( x_{1},x_{2},...,x_{L};y_{1},y_{2},...,y_{L};0\right)
&=&\sum_{\sigma \in S^{L}}\dprod\limits_{j=1}^{L}N_{0}\left( x_{j}-y_{\sigma
\left( j\right) }\right)  \label{S1E9b}
\end{eqnarray}

\subsection{The Cauchy problem  in Fourier variables}
We can reformulate (\ref{S1E6}), (\ref{S1E9a}), (\ref{S1E9b}) using Fourier
variables. To this end, we define the Fourier transform of the functions $%
F_{L,M}$ by means of%
\begin{eqnarray}
&&\widehat{F_{L,M}}\left( k_{1},k_{2},...,k_{L};\xi _{1},\xi _{2},...,\xi
_{M};t\right)=\frac{1}{\left( 2\pi \right) ^{\frac{3}{2}\left( L+M\right) }}\times   \nonumber \\
&&\times \int_{\left( \mathbb{R}^{3}\right) ^{L}}dx_{1}dx_{2}...dx_{L}\int_{\left( 
\mathbb{R}^{3}\right) ^{M}}dy_{1}dy_{2}...dy_{M}\,e^{ \left( -i\left(
\sum_{j=1}^{L}k_{j}x_{j}-\sum_{\ell =1}^{M}\xi _{\ell }y_{\ell }\right)
\right)} \times \nonumber \\ 
&&\times  F_{L,M}\left( x_{1},x_{2},...,x_{L};y_{1},y_{2},...,y_{M};t\right), \,\,t\in  \mathbb{R},\, k_i\in  \mathbb{R}^3, \xi _i\in \mathbb{R}^3.
\label{S1E10}
\end{eqnarray}

We then have, using (\ref{S1E9a}), (\ref{S1E9b}), the following initial
values for these correlation functions%
 \begin{equation}
\widehat{F_{L,M}}\left( k_{1},k_{2},...,k_{L};\xi _{1},\xi _{2},...,\xi
_{M};0\right) =0,\ \ \ \  if \,\,L\neq M  \label{S2E2a}
\end{equation}%

\begin{eqnarray*}
&&\widehat{F_{L,L}}\left( k_{1},k_{2},...,k_{L};\xi _{1},\xi _{2},...,\xi
_{L};0\right)= \\
&&=\frac{1}{\left( 2\pi \right) ^{\frac{3}{2}\left( L+M\right) }}%
\int_{\left( \mathbb{R}^{3}\right) ^{L}}dx_{1}dx_{2}...dx_{L}\int_{\left( 
\mathbb{R}^{3}\right) ^{L}}dy_{1}dy_{2}...dy_{L}\times \nonumber \\
&& \times \exp \left( -i\left(
\sum_{j=1}^{L}k_{j}x_{j}-\sum_{\ell =1}^{L}\xi _{\ell }y_{\ell }\right)
\right) \sum_{\sigma \in S^{L}}\dprod\limits_{j=1}^{L}N_{0}\left(
x_{j}-y_{\sigma \left( j\right) }\right).
\end{eqnarray*}

Then%
\begin{eqnarray*}
&&\widehat{F_{L,L}}\left( k_{1},k_{2},...,k_{L};\xi _{1},\xi _{2},...,\xi
_{L};0\right)=  \\
&&=\frac{1}{\left( 2\pi \right) ^{\frac{3}{2}\left( L+M\right) }}%
\sum_{\sigma \in S^{L}}\int_{\left( \mathbb{R}^{3}\right) ^{L}}dy_{\sigma
\left( 1\right) }dy_{\sigma \left( 2\right) }...dy_{\sigma \left( L\right)
}\int_{\left( \mathbb{R}^{3}\right) ^{L}}dx_{1}dx_{2}...dx_{L}\times \\
&&\times \exp \left(
-i\left( \sum_{j=1}^{L}\left( k_{j}x_{j}-\xi _{\sigma \left( j\right)
}y_{\sigma \left( j\right) }\right) \right) \right) \dprod\limits_{j=1}^{L}N_{0}\left( x_{j}-y_{\sigma \left( j\right)
}\right).
\end{eqnarray*}

Therefore%
\begin{eqnarray*}
&&\widehat{F_{L,L}}\left( k_{1},k_{2},...,k_{L};\xi _{1},\xi _{2},...,\xi
_{L};0\right)= \frac{1}{\left( 2\pi \right) ^{\frac{3}{2}\left( L+L\right) }}\times \\
&&\times 
\sum_{\sigma \in S^{L}}\prod_{j=1}^{L}\left[ \int_{\mathbb{R}^{3}}dy_{\sigma
\left( j\right) }\int_{\mathbb{R}^{3}}dx_{j}\exp \left( -i\left( \left(
k_{j}x_{j}-\xi _{\sigma \left( j\right) }y_{\sigma \left( j\right) }\right)
\right) \right) N_{0}\left( x_{j}-y_{\sigma \left( j\right) }\right) \right]
\end{eqnarray*}

Thus%
\begin{eqnarray}
&&\widehat{F_{L,L}}\left( k_{1},k_{2},...,k_{L};\xi _{1},\xi _{2},...,\xi
_{L};0\right)= \frac{1}{\left( 2\pi \right) ^{\frac{3}{2}\left( L+L\right) }}\times \nonumber \\
&&\times 
\sum_{\sigma \in S^{L}}\prod_{j=1}^{L}\Bigg[ \int_{\mathbb{R}^{3}}dy_{\sigma
\left( j\right) }\int_{\mathbb{R}^{3}}dx_{j}\exp \left( -i\left( k_{j}\left(
x_{j}-y_{\sigma \left( j\right) }\right) \right) \right)\times \nonumber\\
&&\times  \exp \left(
-i\left( k_{j}-\xi _{\sigma \left( j\right) }\right) y_{\sigma \left(
j\right) }\right) N_{0}\left( x_{j}-y_{\sigma \left( j\right) }\right) %
\Bigg]  \nonumber \\
&&=\frac{1}{\left( 2\pi \right) ^{\frac{3}{2}\left( L+L\right) }}%
\sum_{\sigma \in S^{L}}\prod_{j=1}^{L}\Bigg[ \int_{\mathbb{R}^{3}}dy_{\sigma
\left( j\right) }\times \nonumber\\
&&\times \exp \left( -i\left( k_{j}-\xi _{\sigma \left( j\right)
}\right) y_{\sigma \left( j\right) }\right) \int_{\mathbb{R}^{3}}dz_{j}\exp
\left( -ik_{j}z_{j}\right) N_{0}\left( z_{j}\right) \Bigg]  \label{S2E1}
\end{eqnarray}
where we use the change of variables $z_{j}=x_{j}-y_{\sigma \left( j\right)
},\ dz_{j}=dx_{j}.$ We now write the Fourier transform for the function $%
N_{0},$ that depends on a single variable, as%
\[
n_{0}\left( k\right) =\frac{1}{\left( 2\pi \right) ^{\frac{3}{2}}}\int_{%
\mathbb{R}^{3}}dz\exp \left( -ikz\right) N_{0}\left( z\right) 
\]

We can then write (\ref{S2E1}) as%
\begin{eqnarray*}
&&\widehat{F_{L,L}}\left( k_{1},k_{2},...,k_{L};\xi _{1},\xi _{2},...,\xi
_{L};0\right) =\frac{1}{\left( 2\pi \right) ^{\frac{3}{2}L}}\times \\
&&\times \sum_{\sigma \in
S^{L}}\prod_{j=1}^{L}\left[ \int_{\mathbb{R}^{3}}dy_{\sigma \left( j\right)
}\exp \left( -i\left( k_{j}-\xi _{\sigma \left( j\right) }\right) y_{\sigma
\left( j\right) }\right) n_{0}\left( k_{j}\right) \right] .
\end{eqnarray*}

We now use that%
\begin{equation}
\frac{1}{\left( 2\pi \right) ^{3}}\int_{\mathbb{R}^{3}}dz\exp \left(
-ikz\right) =\delta \left( k\right).  \label{T1E6}
\end{equation}

Therefore%
\begin{equation}
\widehat{F_{L,L}}\left( k_{1},k_{2},...,k_{L};\xi _{1},\xi _{2},...,\xi
_{L};0\right) =\left( 2\pi \right) ^{\frac{3}{2}L}\sum_{\sigma \in
S^{L}}\prod_{j=1}^{L}\left[ \delta \left( k_{j}-\xi _{\sigma \left( j\right)
}\right) n_{0}\left( k_{j}\right) \right]  \label{S2E2}
\end{equation}

This formula, combined with (\ref{S2E2a}) gives the form of the Fourier
transforms of the correlation functions $F_{L,M}$ at time $t=0.$

\bigskip

We now rewrite the equations (\ref{S1E6}) using the Fourier variables.
In order to simplify some of the formulas below, let us introduce the following notation. Given some  $z\in \mathbb{R}^{3}$ and the function $F _{ M, L }$, depending on the generic variables $\left( x_{1},x_{2},...,x_{L};y_{1},y_{2},...,y_{M}\right) $, we denote $F _{ L, M }^{\{z\}}$ the function defined as
\begin{eqnarray}
&&F _{ L, M }^{\{z\}}\left( x_{1},x_{2},...,x_{L};y_{1},y_{2},...,y_{M}\right) \nonumber \\
&&=F _{ L+1, M+1 }\left( x_{1},x_{2},...,x_{L}, z;y_{1},y_{2},...,y_{M}, z\right)\label{FLM}
\end{eqnarray} 
for all $\{ x_{1},x_{2},...,x_{L};y_{1},y_{2},...,y_{M} \} \in (\mathbb{R}^{3})^{L+M}.$

Taking the Fourier transform, defined in (\ref{S1E10}), of (\ref{S1E6}) we
obtain%
 
\begin{eqnarray}
&&i\partial _{t}\widehat{F_{L,M}}\left( k_{1},k_{2},...,k_{L};\xi _{1},\xi
_{2},...,\xi _{M};t\right)=  \nonumber \\
&&=\frac{1}{2}\left( -\sum_{j=1}^{L}\left\vert k_{j}\right\vert
^{2}+\sum_{j=1}^{M}\left\vert \xi _{j}\right\vert ^{2}\right) \widehat{%
F_{L,M}}\left( k_{1},k_{2},...,k_{L};\xi _{1},\xi _{2},...,\xi _{M};t\right)
-  \nonumber \\
&&-\varepsilon \sum_{j=1}^{L}\mathcal{F}\left[ F_{L+1,M+1}\left(
x_{1},x_{2},...,x_{L},x_{j};y_{1},y_{2},...,y_{M},x_{j};t\right) \right]\nonumber\\
&&\hskip 5.5cm \left( k_{1},k_{2},...,k_{L};\xi _{1},\xi _{2},...,\xi _{M};t\right) + 
\nonumber \\
&&+\varepsilon \sum_{j=1}^{M}\mathcal{F}\left[ F_{L+1,M+1}\left(
x_{1},x_{2},...,x_{L},y_{j};y_{1},y_{2},...,y_{M},y_{j};t\right) \right] \nonumber\\
&&\hskip 5.5 cm \left( k_{1},k_{2},...,k_{L};\xi _{1},\xi _{2},...,\xi _{M};t\right)
\label{T1E1}
\end{eqnarray}%

\begin{eqnarray}
&&i\partial _{t}\widehat{F_{L,M}}\left( k_{1},k_{2},...,k_{L};\xi _{1},\xi
_{2},...,\xi _{M};t\right)  \nonumber \\
&=&\frac{1}{2}\left( -\sum_{j=1}^{L}\left\vert k_{j}\right\vert
^{2}+\sum_{j=1}^{M}\left\vert \xi _{j}\right\vert ^{2}\right) \widehat{%
F_{L,M}}\left( k_{1},k_{2},...,k_{L};\xi _{1},\xi _{2},...,\xi _{M};t\right)
-  \nonumber \\
&&-\varepsilon \sum_{j=1}^{L}\mathcal{F}\left[ F^{\{x_j\}}_{L,M} \right]\left( k_{1},k_{2},...,k_{L};\xi _{1},\xi _{2},...,\xi _{M};t\right) + 
\nonumber \\
&&+\varepsilon \sum_{j=1}^{M}\mathcal{F}\left[ F^{\{x_j\}}_{L,M}\right]
\left( k_{1},k_{2},...,k_{L};\xi _{1},\xi _{2},...,\xi _{M};t\right)
\label{T1E1BB}
\end{eqnarray}

where $\mathcal{F}$ denoted the Fourier transform, defined by means of (\ref%
{S1E10}). Notice that the definition depends on the number and type of
variables in which $\mathcal{F}$ is acting. We then use that inverting the
Fourier transform defined in (\ref{S1E10}) we have%
\begin{eqnarray}
&&F_{L+1,M+1}\left(
x_{1},x_{2},...,x_{L},x_{L+1};y_{1},y_{2},...,y_{M},y_{M+1};t\right)= 
\nonumber \\
&&= \frac{1}{\left( 2\pi \right) ^{\frac{3}{2}\left( L+M+2\right) }}\times \int_{\left( \mathbb{R}^{3}\right) ^{L+1}}d\bar{k}_{1}d\bar{k}_{2}...d\bar{k}%
_{L}d\bar{k}_{L+1}\int_{\left( \mathbb{R}^{3}\right) ^{M}+1}d\bar{\xi}_{1}d%
\bar{\xi}_{2}...d\bar{\xi}_{M}d\bar{\xi}_{M+1}\nonumber \\
&&\times \exp \left( i\left(
\sum_{j=1}^{L+1}\bar{k}_{j}x_{j}-\sum_{\ell =1}^{M+1}\bar{\xi}_{\ell
}y_{\ell }\right) \right)\times \nonumber \\
&&\times \widehat{F_{L+1,M+1}}\left( \bar{k}_{1},\bar{k}_{2},...,\bar{k}_{L},%
\bar{k}_{L+1};\bar{\xi}_{1},\bar{\xi}_{2},...,\bar{\xi}_{M},\bar{\xi}%
_{M+1};t\right).  \label{T1E2}
\end{eqnarray}

Then
\begin{eqnarray}
&&F_{L+1,M+1}\left(
x_{1},x_{2},...,x_{L},x_{j};y_{1},y_{2},...,y_{M},x_{j};t\right)=  \nonumber
\\
&&=\frac{1}{\left( 2\pi \right) ^{\frac{3}{2}\left( L+M+2\right) }}%
\int_{\left( \mathbb{R}^{3}\right) ^{L+1}}d\bar{k}_{1}d\bar{k}_{2}...d\bar{k}%
_{L}d\bar{k}_{L+1}\int_{\left( \mathbb{R}^{3}\right) ^{M+1}}d\bar{\xi}_{1}d%
\bar{\xi}_{2}...d\bar{\xi}_{M}d\bar{\xi}_{M+1}\times \nonumber \\
&&\times \exp \left( i\left(
\sum_{s=1}^{L}\bar{k}_{s}x_{s}-\sum_{\ell =1}^{M}\bar{\xi}_{\ell }y_{\ell
}\right) \right)\exp \left( i\left( \left( \bar{k}_{L+1}-\bar{\xi}_{M+1}\right)
x_{j}\right) \right)\times \nonumber\\ 
&&\times \widehat{F_{L+1,M+1}}\left( \bar{k}_{1},\bar{k}%
_{2},...,\bar{k}_{L},\bar{k}_{L+1};\bar{\xi}_{1},\bar{\xi}_{2},...,\bar{\xi}%
_{M},\bar{\xi}_{M+1};t\right)  \label{T1E3}
\end{eqnarray}%
for any $j=1,2,...,L.$  Using now the notation in (\ref{FLM}), for each $j=1,...L$,
\begin{eqnarray*}
&&F _{ L, M }^{\{x_j\}}\left( x_{1},x_{2},...,x_{L};y_{1},y_{2},...,y_{M}\right)  \\
&&=F _{ L+1, M+1 }\left( x_{1},x_{2},...,x_{L}, x_j;y_{1},y_{2},...,y_{M}, x_j\right) 
\end{eqnarray*} 

we now compute the Fourier transform of this
function, that is
given by%
\begin{eqnarray*}
&&\mathcal{F}\left[ F^{\{x_j\}}_{L,M} \right]
\left( k_{1},k_{2},...,k_{L};\xi _{1},\xi _{2},...,\xi _{M};t\right)= \\
&&=\frac{1}{\left( 2\pi \right) ^{\frac{3}{2}\left( L+M\right) }}\frac{1}{%
\left( 2\pi \right) ^{\frac{3}{2}\left( L+M+2\right) }}\int_{\left( \mathbb{R%
}^{3}\right) ^{L}}dx_{1}dx_{2}...dx_{L}\int_{\left( \mathbb{R}^{3}\right)
^{M}}dy_{1}dy_{2}...dy_{M}\times \\
&&\exp \left( -i\left(
\sum_{s=1}^{L}k_{s}x_{s}-\sum_{\ell =1}^{M}\xi _{\ell }y_{\ell }\right)
\right) \times  \\
&&\times  \int_{\left( \mathbb{R}^{3}\right) ^{L+1}}d\bar{k}_{1}d\bar{k}%
_{2}...d\bar{k}_{L}d\bar{k}_{L+1}\int_{\left( \mathbb{R}^{3}\right) ^{M+1}}d%
\bar{\xi}_{1}d\bar{\xi}_{2}...d\bar{\xi}_{M}d\bar{\xi}_{M+1}\times \\
&&\times \exp \left(
i\left( \sum_{s=1}^{L}\bar{k}_{s}x_{s}-\sum_{\ell =1}^{M}\bar{\xi}_{\ell
}y_{\ell }\right) \right)\exp \left( i\left( \left( \bar{k}_{L+1}-\bar{\xi}_{M+1}\right)
x_{j}\right) \right) \times \\
&&\times \widehat{F_{L+1,M+1}}\left( \bar{k}_{1},\bar{k}%
_{2},...,\bar{k}_{L},\bar{k}_{L+1};\bar{\xi}_{1},\bar{\xi}_{2},...,\bar{\xi}%
_{M},\bar{\xi}_{M+1};t\right)
\end{eqnarray*}

Using now the identity (in the sense of distributions) $\frac{1}{\left( 2\pi
\right) ^{3}}\int_{\mathbb{R}^{3}}dx\exp \left( ikx\right) =\delta \left(
k\right) $ (cf. (\ref{T1E6})) we obtain%
\begin{eqnarray*}
&&\mathcal{F}\left[ F^{\{x_j\}}_{L,M}\right]
\left( k_{1},k_{2},...,k_{L};\xi _{1},\xi _{2},...,\xi _{M};t\right) =\\
&&=\frac{1}{\left( 2\pi \right) ^{3}}\int_{\left( \mathbb{R}^{3}\right)
^{L+1}}d\bar{k}_{1}d\bar{k}_{2}...d\bar{k}_{L}d\bar{k}_{L+1}\int_{\left( 
\mathbb{R}^{3}\right) ^{M+1}}d\bar{\xi}_{1}d\bar{\xi}_{2}...d\bar{\xi}_{M}d%
\bar{\xi}_{M+1}\times  \\
&&\times  \prod_{\ell =1}^{M}\left[ \delta \left( \xi _{\ell }-\bar{\xi}_{\ell
}\right) \right]  \prod_{s=1;s\neq j}^{L}\left[ \delta \left( k_{s}-%
\bar{k}_{s}\right) \right] \delta \left( \bar{k}_{j}-k_{j}+\bar{k}_{L+1}-%
\bar{\xi}_{M+1}\right) \times  \\
&&\times  \widehat{F_{L+1,M+1}}\left( \bar{k}_{1},\bar{k}_{2},...,\bar{k}_{L},%
\bar{k}_{L+1};\bar{\xi}_{1},\bar{\xi}_{2},...,\bar{\xi}_{M},\bar{\xi}%
_{M+1};t\right)
\end{eqnarray*}

Therefore%
\begin{eqnarray*}
&&\mathcal{F}\left[ F^{\{x_j\}}_{L,M}\right]
\left( k_{1},k_{2},...,k_{L};\xi _{1},\xi _{2},...,\xi _{M};t\right) = \\
&&=\frac{1}{\left( 2\pi \right) ^{3}}\int_{\left( \mathbb{R}^{3}\right)
^{2}}d\bar{k}_{j}d\bar{k}_{L+1}\int_{\mathbb{R}^{3}}d\bar{\xi}_{M+1}\cdot
\delta \left( \bar{k}_{j}-k_{j}+\bar{k}_{L+1}-\bar{\xi}_{M+1}\right) \times  \\
&&\times  \widehat{F_{L+1,M+1}}\left( k_{1},k_{2},...k_{j-1},\bar{k}%
_{j},k_{j+1},...,k_{L},\bar{k}_{L+1};\xi _{1},\xi _{2},...,\xi _{M},\bar{\xi}%
_{M+1};t\right)
\end{eqnarray*}

Henceforth%
\begin{eqnarray}
&&\mathcal{F}\left[ F^{\{x_j\}}_{L,M}\right]
\left( k_{1},k_{2},...,k_{L};\xi _{1},\xi _{2},...,\xi _{M};t\right) = 
\frac{1}{\left( 2\pi \right) ^{3}}\int_{\mathbb{R}^{3}}d\bar{k}%
_{L+1}\int_{\mathbb{R}^{3}}d\bar{\xi}_{M+1}\times \nonumber\\
&&\widehat{F_{L+1,M+1}}\Big(
k_{1},k_{2},...k_{j-1},k_{j}-\bar{k}_{L+1}+\bar{\xi}_{M+1},k_{j+1},...,k_{L},%
\bar{k}_{L+1};\nonumber \\
&&\hskip 6.6cm \xi _{1},\xi _{2},...,\xi _{M},\bar{\xi}_{M+1};t\Big)
\label{T1E4}
\end{eqnarray}
We can compute arguing similarly the term 
$
\mathcal{F}\left[ F^{\{y_j\}}_{L,M}\right]
\left( k_{1},k_{2},...,k_{L};\xi _{1},\xi _{2},...,\xi _{M};t\right) 
$
that yields
 
\begin{eqnarray}
&&\mathcal{F}\left[ F^{\{y_j\}}_{L,M}\right]
\left( k_{1},k_{2},...,k_{L};\xi _{1},\xi _{2},...,\xi _{M};t\right) =
\frac{1}{\left( 2\pi \right) ^{3}}\int_{\mathbb{R}^{3}}d\bar{k}%
_{L+1}\int_{\mathbb{R}^{3}}d\bar{\xi}_{M+1}\nonumber \\
&&\widehat{F_{L+1,M+1}}
\Big(
k_{1},k_{2},...,k_{L},\bar{k}_{L+1};\\\nonumber
&&\hskip 2cm;\xi _{1},\xi _{2},...,\xi _{j-1},\xi
_{j}+\bar{k}_{L+1}-\bar{\xi}_{M+1},\xi _{j+1},...,\xi _{M},\bar{\xi}%
_{M+1};t\Big).  \label{T1E5}
\end{eqnarray}
(Notice the change of sign in the term $\bar{\xi}_{j}$).

Therefore, combining (\ref{T1E1}), (\ref{T1E4}) and (\ref{T1E5}) we obtain
the following evolution equation for $\widehat{F_{L,M}}$

\begin{eqnarray}
&&i\partial _{t}\widehat{F_{L,M}}\left( k_{1},k_{2},...,k_{L};\xi _{1},\xi
_{2},...,\xi _{M};t\right)=\nonumber\\
&&=\frac{1}{2}\left( -\sum_{j=1}^{L}\left\vert k_{j}\right\vert
^{2}+\sum_{j=1}^{M}\left\vert \xi _{j}\right\vert ^{2}\right) \widehat{%
F_{L,M}}\left( k_{1},k_{2},...,k_{L};\xi _{1},\xi _{2},...,\xi _{M};t\right)
-  \nonumber \\
&&-\frac{\varepsilon }{\left( 2\pi \right) ^{3}}\sum_{j=1}^{L}\int_{\mathbb{R%
}^{3}}d\bar{k}_{L+1}\int_{\mathbb{R}^{3}}d\bar{\xi}_{M+1}\widehat{F_{L+1,M+1}%
}\Big( k_{1},k_{2},...k_{j-1},\nonumber \\
&&\hskip 0.5cm ,k_{j}-\bar{k}_{L+1}+\bar{\xi}%
_{M+1},k_{j+1},...,k_{L},\bar{k}_{L+1};\xi _{1},\xi _{2},...,\xi _{M},\bar{%
\xi}_{M+1};t\Big) +  \nonumber \\
&&+\frac{\varepsilon }{\left( 2\pi \right) ^{3}}\sum_{j=1}^{M}\int_{\mathbb{R%
}^{3}}d\bar{k}_{L+1}\int_{\mathbb{R}^{3}}d\bar{\xi}_{M+1}\widehat{F_{L+1,M+1}%
}\Big( k_{1},k_{2},...,k_{L},\bar{k}_{L+1};;\xi _{1},\nonumber \\
&&,\xi _{2},...,\xi
_{j-1},\xi _{j}+\bar{k}_{L+1}-\bar{\xi}_{M+1},\xi _{j+1},...,\xi _{M},\bar{%
\xi}_{M+1};t\Big),\,\,\,t\in  \mathbb{R},\, k_i\in  \mathbb{R}^3, \xi _i\in \mathbb{R}^3. \label{S2E3} 
\end{eqnarray}

The functions $F_{L,L}\left(
x_{1},x_{2},...,x_{L};y_{1},y_{2},...,y_{M}\right) $ are invariant under translation and under permutations of the first or second group of variables taken separately. Then%
\begin{eqnarray*}
&&F_{L,L}\left( x_{\sigma \left( 1\right) },x_{\sigma \left( 2\right)
},...,x_{\sigma \left( L\right) };y_{1},y_{2},...,y_{L}\right)=\\
&& =F_{L,L}\left( x_{1},x_{2},...,x_{L};y_{\sigma \left( 1\right) },y_{\sigma
\left( 2\right) },...,y_{\sigma \left( L\right) }\right)\\ 
&& =F_{L,L}\left(
x_{1},x_{2},...,x_{L};y_{1},y_{2},...,y_{M}\right) 
\end{eqnarray*}
for any $\sigma \in S^{L}$ and for all $R$,
\begin{eqnarray*}
&&F_{L,L}\left(
x_{1},x_{2},...,x_{L};y_{1},y_{2},...,y_{M}\right)=\\
&&F _{ L, L } \left(
x_{1}+R,x_{2}+R,...,x_{L}+R;y_{1}+R,y_{2}+R,...,y_{M}+R\right).
\end{eqnarray*}
\\
The translation invariance of  $F_{L,L}$ implies that 
$\widehat{F_{L,L}}$ has  a particular functional form. Such a functional form will be valid even if the correlations
are not small.  Due to this invariance,  in the case $L=1$,
\begin{equation}
\widehat{F_{1,1}}\left( k_{1};\xi _{1};t\right) =\left( 2\pi \right) ^{\frac{%
3}{2}}\delta \left( k_{1}-\xi _{1}\right) n\left( k_{1},t\right)
\label{S3E1} 
\end{equation}
We perform this computation in $\widehat{F_{2,2}}$,
\begin{eqnarray*}
&&\widehat{F_{2,2}}\left( k_{1},k_{2};\xi _{1},\xi _{2};t\right) = \\
&&=\frac{1}{\left( 2\pi \right) ^{6}}\int_{\left( \mathbb{R}^{3}\right)
^{2}}dx_{1}dx_{2}\int_{\left( \mathbb{R}^{3}\right) ^{2}}dy_{1}dy_{2}\exp
\left( -i\left( \sum_{j=1}^{2}k_{j}x_{j}-\sum_{\ell =1}^{2}\xi _{\ell
}y_{\ell }\right) \right)\times \\
&&\hskip 7.5cm  \times F_{2,2}\left( x_{1},x_{2};y_{1},y_{2};t\right)  \\
&&=\frac{1}{\left( 2\pi \right) ^{6}}\int_{\left( \mathbb{R}^{3}\right)
^{2}}dx_{1}dx_{2}\int_{\left( \mathbb{R}^{3}\right) ^{2}}dy_{1}dy_{2}\exp
\left( -i\left( \sum_{j=1}^{2}k_{j}x_{j}-\sum_{\ell =1}^{2}\xi _{\ell
}y_{\ell }\right) \right)\times \\
&&\hskip 5.5cm  \times  F_{2,2}\left(
x_{1}-y_{2},x_{2}-y_{2};y_{1}-y_{2},0;t\right)  \\
&&=\frac{1}{\left( 2\pi \right) ^{6}}\int_{\left( \mathbb{R}^{3}\right)
^{2}}dx_{1}dx_{2}\int_{\left( \mathbb{R}^{3}\right) ^{2}}dy_{1}dy_{2}\exp
\left( -i\left( k_{1}x_{1}+k_{2}x_{2}-\left( \xi _{1}y_{1}+\xi
_{2}y_{2}\right) \right) \right)\times \\
&&\hskip 5.5cm\times  F_{2,2}\left(
x_{1}-y_{2},x_{2}-y_{2};y_{1}-y_{2},0;t\right)  \\
&&=\frac{1}{\left( 2\pi \right) ^{6}}\int_{\left( \mathbb{R}^{3}\right)
^{2}}dx_{1}dx_{2}\int_{\left( \mathbb{R}^{3}\right) ^{2}}dy_{1}dy_{2}\times \\
&&\times \exp
\left( -i\left( k_{1}\left( x_{1}-y_{2}\right) +k_{2}\left(
x_{2}-y_{2}\right) -\xi _{1}\left( y_{1}-y_{2}\right) \right) \right) \times  
\\
&&\times  \exp \left( -i\left( -k_{1}y_{2}-k_{2}y_{2}+\xi _{1}y_{2}+\xi
_{2}y_{2}\right) \right) \!F_{2,2}\left(
x_{1}-y_{2},x_{2}-y_{2};y_{1}-y_{2},0;t\right)  \\
&&=\frac{1}{\left( 2\pi \right) ^{6}}\left( 2\pi \right) ^{3}\int_{\left( 
\mathbb{R}^{3}\right) ^{2}}dx_{1}dx_{2}\int_{\mathbb{R}^{3}}dy_{1}\exp
\left( -i\left( k_{1}x_{1}+k_{2}x_{2}-\xi _{1}y_{1}\right) \right)\times \\
&&\hskip 4cm \times  \delta
\left( k_{1}+k_{2}-\xi _{1}-\xi _{2}\right) F_{2,2}\left(
x_{1},x_{2};y_{1},0;t\right)  \\
&&=\frac{1}{\left( 2\pi \right) ^{3}}\delta \left( k_{1}+k_{2}-\xi _{1}-\xi
_{2}\right) \int_{\left( \mathbb{R}^{3}\right) ^{2}}dx_{1}dx_{2}\int_{%
\mathbb{R}^{3}}dy_{1}F_{2,2}\left( x_{1},x_{2};y_{1},0;t\right)\times \\
&&\hskip 5.7cm\times  \exp \left(
-i\left( k_{1}x_{1}+k_{2}x_{2}-\xi _{1}y_{1}\right) \right) 
\end{eqnarray*}
In absence of correlations, the invariance under permutations of the variables, as well as the
invariance under translations yields the functional dependence%
\begin{equation}
\widehat{F_{L,L}}\left( k_{1},k_{2},...,k_{L};\xi _{1},\xi _{2},...,\xi
_{L};t\right) =\left( 2\pi \right) ^{\frac{3}{2}L}\sum_{\sigma \in
S^{L}}\prod_{j=1}^{L}\left[ \delta \left( k_{j}-\xi _{\sigma \left( j\right)
}\right) n\left( k_{j},t\right) \right]   \label{S3E1a}
\end{equation}
However, for $t>0$ this formula
can only be expected to hold in an approximated way.  
\\\\
We define  the following function $F_{2,2}^{\left( \tau \right) },$ that simplifies the form of $F_{2,2}$ if the problem is
invariant under translations.
\[
F_{2,2}^{\left( \tau \right) }\left( x_{1},x_{2};y_{1};t\right)
=F_{2,2}\left( x_{1},x_{2};y_{1},0;t\right). 
\]
Indeed, we have%
\[
F_{2,2}\left( x_{1},x_{2};y_{1},y_{2};t\right) =F_{2,2}^{\left( \tau \right)
}\left( x_{1}-y_{2},x_{2}-y_{2};y_{1}-y_{2};t\right). 
\]
Then, the Fourier transform of $F_{2,2}^{\left( \tau \right) }$ is given by%
\begin{eqnarray*}
\widehat{F_{2,2}^{\left( \tau \right) }}\left( k_{1},k_{2};\xi _{1};t\right)
=\frac{1}{\left( 2\pi \right) ^{\frac{9}{2}}}\int_{\left( \mathbb{R}%
^{3}\right) ^{2}}dx_{1}dx_{2}\int_{\mathbb{R}^{3}}dy_{1}F_{2,2}^{\left(
t\right) }\left( x_{1},x_{2};y_{1};t\right)\times \\ 
\times \exp \left( -i\left(
k_{1}x_{1}+k_{2}x_{2}-\xi _{1}y_{1}\right) \right) 
\end{eqnarray*}
The function $\widehat{F_{2,2}^{\left( \tau \right) }}$ can be expected to
be smooth. In particular, it does not contain Dirac masses or other measures
supported in low-dimensional sets. Then%
\begin{equation}
\widehat{F_{2,2}}\left( k_{1},k_{2};\xi _{1},\xi _{2};t\right) =\left( 2\pi
\right) ^{\frac{3}{2}}\delta \left( k_{1}+k_{2}-\xi _{1}-\xi _{2}\right) 
\widehat{F_{2,2}^{\left( \tau \right) }}\left( k_{1},k_{2};\xi _{1};t\right)
\label{S3E3}
\end{equation}%
where the function $\widehat{F_{2,2}^{\left( \tau \right) }}$ is smooth.
Notice that, since we have not assumed that the variables are uncorrelated,
we have only one Dirac mass involving the four variables $k_{1},\ k_{2},\
\xi _{1},\ \xi _{2},$ instead of the symmetrized product of Dirac masses in (%
\ref{S2E2}). In the case $L=2,$ this formula becomes%
\begin{eqnarray*}
\widehat{F_{2,2}}\left( k_{1},k_{2};\xi _{1},\xi _{2};0\right) =\left( 2\pi
\right) ^{3}\left[ \delta \left( k_{1}-\xi _{1}\right) \delta \left(
k_{2}-\xi _{2}\right) +\delta \left( k_{1}-\xi _{2}\right) \delta \left(
k_{2}-\xi _{1}\right) \right]\times \\
\times  n_{0}\left( k_{1}\right) n_{0}\left(
k_{2}\right) 
\end{eqnarray*}
Notice that the support of $\widehat{F_{2,2}}\left( k_{1},k_{2};\xi _{1},\xi
_{2};0\right) $ is contained in the union of hyperplanes 
$$\left[ \left\{
k_{1}=\xi _{1}\right\} \times \left\{ k_{2}=\xi _{2}\right\} \right] \cup %
\left[ \left\{ k_{1}=\xi _{2}\right\} \times \left\{ k_{2}=\xi _{1}\right\} %
\right] $$ that is contained in the hyperplane $\left\{ k_{1}+k_{2}=\xi
_{1}+\xi _{2}\right\} $ where $\widehat{F_{2,2}}\left( k_{1},k_{2};\xi
_{1},\xi _{2};t\right) $ is supported for $t>0.$

\section{Closure of the hierarchy. Derivation of a kinetic equation for $
n_{1}.$}
\label{Derivation}
\setcounter{equation}{0}
\setcounter{theorem}{0}

We now notice that we can obtain a solution of the equations (\ref{S2E3})
with initial conditions (\ref{S2E2a}) for the whole set of values $L,\ M$
with $L\neq M$, namely  
\begin{equation}
\label{S2E4}
\widehat{F_{L,M}}\left( k_{1},k_{2},...,k_{L};\xi _{1},\xi _{2},...,\xi
_{M};t\right) =0,\,\,\,t\in  \mathbb{R},\, k_i\in  \mathbb{R}^3, \xi _i\in \mathbb{R}^3,\,\,\,if \ \ \ L\neq M
\end{equation}
Notice that (\ref{S2E4}) would hold if the solutions of the
hierarchy (\ref{S2E3}) are unique. Therefore, we will assume (\ref{S2E4}) in the
following.

We now examine the approximation of the solutions of (\ref{S2E3}) with $L=M$
and initial value (\ref{S2E1}). We consider first the evolution of the
functions $\widehat{F_{L,L}}$ with lowest values of $L.$ Specifically, in
order to compute the evolution of $\widehat{F_{1,1}}$ we use (\ref{S2E3}) 
and (\ref{S2E2}) to
obtain%
\begin{eqnarray}
&&i\partial _{t}\widehat{F_{1,1}}\left( k_{1};\xi _{1};t\right) =
\nonumber \\
&=&\frac{1}{2}\left( -\left\vert k_{1}\right\vert ^{2}+\left\vert \xi
_{1}\right\vert ^{2}\right) \widehat{F_{1,1}}\left( k_{1};\xi _{1};t\right) -
\nonumber \\
&&-\frac{\varepsilon }{\left( 2\pi \right) ^{3}}\int_{\mathbb{R}^{3}}d\bar{k}%
_{2}\int_{\mathbb{R}^{3}}d\bar{\xi}_{2}\widehat{F_{2,2}}\left( k_{1}-\bar{k}%
_{2}+\bar{\xi}_{2},\bar{k}_{2};\xi _{1},\bar{\xi}_{2};t\right) +  \nonumber
\\
&&+\frac{\varepsilon }{\left( 2\pi \right) ^{3}}\int_{\mathbb{R}^{3}}d\bar{k}%
_{2}\int_{\mathbb{R}^{3}}d\bar{\xi}_{2}\widehat{F_{2,2}}\left( k_{1},\bar{k}%
_{2};\xi _{1}+\bar{k}_{2}-\bar{\xi}_{2},\bar{\xi}_{2};t\right)   \label{S2E7} 
\end{eqnarray}
 
\begin{eqnarray}
\widehat{F_{1,1}}\left( k_{1};\xi _{1};0\right)&=&\left( 2\pi \right) ^{\frac{3}{2}%
}\delta \left( k_{j}-\xi _{\sigma \left( j\right) }\right) n_{0}\left(
k_{j}\right).  \label{S2E79}
\end{eqnarray}

As a matter of fact, in order to obtain the evolution equation for $\widehat{%
F_{1,1}}$ we need only the function $G_{2,2}$ (or $\widehat{G_{2,2}}$), that
it is usually termed as the cumulant of second order. More precisely, we
define $G_{2,2}$ by means of%
\begin{eqnarray}
F_{2,2}\left( x_{1},x_{2};y_{1},y_{2}; t\right) =F_{1,1}\left(
x_{1};y_{1}; t\right) F_{1,1}\left( x_{2};y_{2}; t\right) +F_{1,1}\left(
x_{1};y_{2}; t\right) F_{1,1}\left( x_{2};y_{1}; t\right)+\nonumber\\
 +G_{2,2}\left(
x_{1},x_{2};y_{1},y_{2}; t\right).  \label{S2E8}
\end{eqnarray}%
In the expression (\ref{S2E8}) it is implicitly  understood that  $G _{ 2,2 }$ is a lower order term with respect to the two other terms in the right hand side of (\ref{S2E8}). By (\ref{S3E1}), these terms are of order $n^2$.

\begin{proposition}
It follows from equation (\ref{S2E79}) that the function $\widehat F _{ 1, 1 }$, satisfies
\begin{eqnarray}
i\partial _{t}\widehat{F_{1,1}}\left( k_{1};\xi _{1};t\right) =\frac{%
\varepsilon }{\left( 2\pi \right) ^{3}}\int_{\mathbb{R}^{3}}d\bar{k}%
_{2}\int_{\mathbb{R}^{3}}d\bar{\xi}_{2}\Big[ \widehat{G_{2,2}}\left( k_{1},%
\bar{k}_{2};\xi _{1}+\bar{k}_{2}-\bar{\xi}_{2},\bar{\xi}_{2};t\right)- \nonumber \\ 
-\widehat{G_{2,2}}\left( k_{1}-\bar{k}_{2}+\bar{\xi}_{2},\bar{k}_{2};\xi _{1},%
\bar{\xi}_{2};t\right) \Big]  \label{S3E2a}
\end{eqnarray}%
and the function $n$ given in (\ref{S3E1}),
\begin{eqnarray}
i\left( 2\pi \right) ^{\frac{3}{2}}\delta \left( k_{1}-\xi _{1}\right)
\partial _{t}n\left( k_{1},t\right) =\frac{\varepsilon }{\left( 2\pi \right)
^{3}}\int_{\mathbb{R}^{3}}d\bar{k}_{2}\int_{\mathbb{R}^{3}}d\bar{\xi}_{2}%
\Big[ \widehat{G_{2,2}}\left( k_{1},\bar{k}_{2};\xi _{1}+\bar{k}_{2}-\bar{%
\xi}_{2},\bar{\xi}_{2};t\right)-\nonumber \\
 -\widehat{G_{2,2}}\left( k_{1}-\bar{k}_{2}+%
\bar{\xi}_{2},\bar{k}_{2};\xi _{1},\bar{\xi}_{2};t\right) \Big]
\label{S3E2}
\end{eqnarray}
\end{proposition}

\begin{proof}
Taking into
account (\ref{S2E7}) we can see, that in order to obtain the evolution
equation for $\widehat{F_{1,1}}$ we need to compute%
\begin{equation}
\widehat{F_{2,2}}\left( k_{1},\bar{k}_{2};\xi _{1}+\bar{k}_{2}-\bar{\xi}_{2},%
\bar{\xi}_{2};t\right) -\widehat{F_{2,2}}\left( k_{1}-\bar{k}_{2}+\bar{\xi}%
_{2},\bar{k}_{2};\xi _{1},\bar{\xi}_{2};t\right).  \label{S2E9}
\end{equation}

Using (\ref{S2E8}) we obtain

\begin{eqnarray}
\widehat{F_{2,2}}\left( k_{1},k_{2};\xi _{1},\xi _{1}\right) =\widehat{%
F_{1,1}}\left( k_{1};\xi _{1}\right) \widehat{F_{1,1}}\left( k_{2};\xi
_{2}\right) +\widehat{F_{1,1}}\left( k_{1};\xi _{2}\right) \widehat{F_{1,1}}%
\left( k_{2};\xi _{1}\right)+\nonumber\\
 +\widehat{G_{2,2}}\left( k_{1},k_{2};\xi
_{1},\xi _{1}\right)  \label{S2E8a}
\end{eqnarray}

Plugging this formula into (\ref{S2E9}) we obtain

\begin{eqnarray*}
&&\widehat{F_{2,2}}\left( k_{1},\bar{k}_{2};\xi _{1}+\bar{k}_{2}-\bar{\xi}%
_{2},\bar{\xi}_{2};t\right) -\widehat{F_{2,2}}\left( k_{1}-\bar{k}_{2}+\bar{%
\xi}_{2},\bar{k}_{2};\xi _{1},\bar{\xi}_{2};t\right)= \\
&&=\widehat{F_{1,1}}\left( k_{1};\xi _{1}+\bar{k}_{2}-\bar{\xi}_{2}\right) 
\widehat{F_{1,1}}\left( \bar{k}_{2};\bar{\xi}_{2}\right) +\widehat{F_{1,1}}%
\left( k_{1};\bar{\xi}_{2}\right) \widehat{F_{1,1}}\left( \bar{k}_{2};\xi
_{1}+\bar{k}_{2}-\bar{\xi}_{2}\right) - \\
&&-\widehat{F_{1,1}}\left( k_{1}-\bar{k}_{2}+\bar{\xi}_{2};\xi _{1}\right) 
\widehat{F_{1,1}}\left( \bar{k}_{2};\bar{\xi}_{2}\right) -\widehat{F_{1,1}}%
\left( k_{1}-\bar{k}_{2}+\bar{\xi}_{2};\bar{\xi}_{2}\right) \widehat{F_{1,1}}%
\left( \bar{k}_{2};\xi _{1}\right) + \\
&&+\widehat{G_{2,2}}\left( k_{1},\bar{k}_{2};\xi _{1}+\bar{k}_{2}-\bar{\xi}%
_{2},\bar{\xi}_{2};t\right) -\widehat{G_{2,2}}\left( k_{1}-\bar{k}_{2}+\bar{%
\xi}_{2},\bar{k}_{2};\xi _{1},\bar{\xi}_{2};t\right)
\end{eqnarray*}

We now recall that the invariance under translations implies (\ref{S3E1}).
Thus,  dropping the dependence on $t$ for the sake of simplicity.
\begin{eqnarray*}
&&\widehat{F_{1,1}}\left( k_{1};\xi _{1}+\bar{k}_{2}-\bar{\xi}_{2}\right) 
\widehat{F_{1,1}}\left( \bar{k}_{2};\bar{\xi}_{2}\right) +\widehat{F_{1,1}}%
\left( k_{1};\bar{\xi}_{2}\right) \widehat{F_{1,1}}\left( \bar{k}_{2};\xi
_{1}+\bar{k}_{2}-\bar{\xi}_{2}\right) - \\
&&-\widehat{F_{1,1}}\left( k_{1}-\bar{k}_{2}+\bar{\xi}_{2};\xi _{1}\right) 
\widehat{F_{1,1}}\left( \bar{k}_{2};\bar{\xi}_{2}\right) -\widehat{F_{1,1}}%
\left( k_{1}-\bar{k}_{2}+\bar{\xi}_{2};\bar{\xi}_{2}\right) \widehat{F_{1,1}}%
\left( \bar{k}_{2};\xi _{1}\right)= \\
&&=\left( 2\pi \right) ^{3}\Big[ n\left( k_{1}\right) n\left( \bar{k}%
_{2}\right) \delta \left( k_{1}-\xi _{1}-\bar{k}_{2}+\bar{\xi}_{2}\right)
\delta \left( \bar{k}_{2}-\bar{\xi}_{2}\right)+\\
&&\hskip 5.5cm  +n\left( k_{1}\right) n\left( 
\bar{k}_{2}\right) \delta \left( k_{1}-\bar{\xi}_{2}\right) \delta \left( 
\bar{\xi}_{2}-\xi _{1}\right) -  \\
&&-n\left( k_{1}-\bar{k}_{2}+\bar{\xi}_{2}\right) n\left( \bar{k}%
_{2}\right) \delta \left( k_{1}-\bar{k}_{2}+\bar{\xi}_{2}-\xi _{1}\right)
\delta \left( \bar{k}_{2}-\bar{\xi}_{2}\right)-\\
&&\hskip 2.5cm  -n\left( k_{1}-\bar{k}_{2}+%
\bar{\xi}_{2}\right) n\left( \bar{k}_{2}\right) \delta \left( k_{1}-\bar{k}%
_{2}+\bar{\xi}_{2}-\bar{\xi}_{2}\right) \delta \left( \bar{k}_{2}-\xi
_{1}\right) \Big] \\
&&=\left( 2\pi \right) ^{3}\Big[ n\left( k_{1}\right) n\left( \bar{k}%
_{2}\right) \delta \left( k_{1}-\xi _{1}\right) \delta \left( \bar{k}_{2}-%
\bar{\xi}_{2}\right) +n\left( k_{1}\right) n\left( \bar{k}_{2}\right) \delta
\left( k_{1}-\xi _{1}\right) \delta \left( \bar{\xi}_{2}-\xi _{1}\right)
-  \\
&&  -n\left( k_{1}\right) n\left( \bar{k}_{2}\right) \delta \left(
k_{1}-\xi _{1}\right) \delta \left( \bar{k}_{2}-\bar{\xi}_{2}\right)
-n\left( k_{1}-\bar{k}_{2}+\bar{\xi}_{2}\right) n\left( \bar{k}_{2}\right)
\delta \left( k_{1}-\xi _{1}\right) \delta \left( \bar{k}_{2}-\xi
_{1}\right) \Big] \\
&&=\left( 2\pi \right) ^{3}\Big[ n\left( k_{1}\right) n\left( \bar{k}%
_{2}\right) \delta \left( k_{1}-\xi _{1}\right) \delta \left( \bar{\xi}%
_{2}-\xi _{1}\right)-\\
&& -n\left( k_{1}-\bar{k}_{2}+\bar{\xi}_{2}\right) n\left( 
\bar{k}_{2}\right) \delta \left( k_{1}-\xi _{1}\right) \delta \left( \bar{k}%
_{2}-\xi _{1}\right) \Big] \\
&&=\left( 2\pi \right) ^{3}\left[ n\left( k_{1}\right) n\left( \bar{k}%
_{2}\right) \delta \left( k_{1}-\xi _{1}\right) \delta \left( \bar{\xi}%
_{2}-\xi _{1}\right) -n\left( \bar{\xi}_{2}\right) n\left( \bar{k}%
_{2}\right) \delta \left( k_{1}-\xi _{1}\right) \delta \left( \bar{k}%
_{2}-k_{1}\right) \right] \\
&&=\left( 2\pi \right) ^{3}\left[ n\left( k_{1}\right) n\left( \bar{k}%
_{2}\right) \delta \left( k_{1}-\xi _{1}\right) \delta \left( \bar{\xi}%
_{2}-\xi _{1}\right) -n\left( \bar{\xi}_{2}\right) n\left( k_{1}\right)
\delta \left( k_{1}-\xi _{1}\right) \delta \left( \bar{k}_{2}-k_{1}\right) %
\right]
\end{eqnarray*}

We now recall that we need to compute%
\[
\int_{\mathbb{R}^{3}}d\bar{k}_{2}\int_{\mathbb{R}^{3}}d\bar{\xi}_{2}\left[ 
\widehat{F_{2,2}}\left( k_{1},\bar{k}_{2};\xi _{1}+\bar{k}_{2}-\bar{\xi}_{2},%
\bar{\xi}_{2};t\right) -\widehat{F_{2,2}}\left( k_{1}-\bar{k}_{2}+\bar{\xi}%
_{2},\bar{k}_{2};\xi _{1},\bar{\xi}_{2};t\right) \right] 
\]

Then, the contribution of the terms containing $F_{1,1}$ to this integral
reduces to%
\begin{eqnarray*}
&&\left( 2\pi \right) ^{3}\int_{\mathbb{R}^{3}}d\bar{k}_{2}\int_{\mathbb{R}%
^{3}}d\bar{\xi}_{2}\Big[ n\left( k_{1}\right) n\left( \bar{k}_{2}\right)
\delta \left( k_{1}-\xi _{1}\right) \delta \left( \bar{\xi}_{2}-\xi
_{1}\right)-\\
&&\hskip 6cm  -n\left( \bar{\xi}_{2}\right) n\left( k_{1}\right) \delta \left(
k_{1}-\xi _{1}\right) \delta \left( \bar{k}_{2}-k_{1}\right) \Big] \\
&&=\left( 2\pi \right) ^{3}n\left( k_{1}\right) \delta \left( k_{1}-\xi
_{1}\right) \Big[ \int_{\mathbb{R}^{3}}n\left( \bar{k}_{2}\right) d\bar{k}%
_{2}\int_{\mathbb{R}^{3}}\delta \left( \bar{\xi}_{2}-\xi _{1}\right) d\bar{%
\xi}_{2}-\\
&&\hskip 6cm -\int_{\mathbb{R}^{3}}\delta \left( \bar{k}_{2}-k_{1}\right) d\bar{k}%
_{2}\int_{\mathbb{R}^{3}}n\left( \bar{\xi}_{2}\right) d\bar{\xi}_{2}\Big]
\\
&&=\left( 2\pi \right) ^{3}n\left( k_{1}\right) \delta \left( k_{1}-\xi
_{1}\right) \left[ \int_{\mathbb{R}^{3}}n\left( \bar{k}_{2}\right) d\bar{k}%
_{2}-\int_{\mathbb{R}^{3}}n\left( \bar{\xi}_{2}\right) d\bar{\xi}_{2}\right]
=0
\end{eqnarray*}

Therefore, the terms containing the functions $\widehat{F_{1,1}}$ in the
equation of $\widehat{F_{2,2}}$ vanish. It then follows that we can rewrite
the equation for $\widehat{F_{1,1}}$ as%
\begin{eqnarray}
&&i\partial _{t}\widehat{F_{1,1}}\left( k_{1};\xi _{1};t\right)=  \nonumber \\
&&=\frac{1}{2}\left( -\left\vert k_{1}\right\vert ^{2}+\left\vert \xi
_{1}\right\vert ^{2}\right) \widehat{F_{1,1}}\left( k_{1};\xi _{1};t\right) +\frac{\varepsilon }{\left( 2\pi \right) ^{3}}\times \nonumber \\
&&\times \int_{\mathbb{R}^{3}}d\bar{k}%
_{2}\int_{\mathbb{R}^{3}}d\bar{\xi}_{2}\left[ \widehat{G_{2,2}}\left( k_{1},%
\bar{k}_{2};\xi _{1}+\bar{k}_{2}-\bar{\xi}_{2},\bar{\xi}_{2};t\right) -%
\widehat{G_{2,2}}\left( k_{1}-\bar{k}_{2}+\bar{\xi}_{2},\bar{k}_{2};\xi _{1},%
\bar{\xi}_{2};t\right) \right]  \label{T1E7}
\end{eqnarray}

We now use (\ref{S3E1}) to prove%
\[
\left( -\left\vert k_{1}\right\vert ^{2}+\left\vert \xi _{1}\right\vert
^{2}\right) \widehat{F_{1,1}}\left( k_{1};\xi _{1};t\right) =\left( 2\pi
\right) ^{\frac{3}{2}}\delta \left( k_{1}-\xi _{1}\right) n\left(
k_{1},t\right) \left( -\left\vert k_{1}\right\vert ^{2}+\left\vert \xi
_{1}\right\vert ^{2}\right) =0 
\]

Then (\ref{T1E7}) reduces to%
\begin{eqnarray}
i\partial _{t}\widehat{F_{1,1}}\left( k_{1};\xi _{1};t\right) =\frac{%
\varepsilon }{\left( 2\pi \right) ^{3}}\int_{\mathbb{R}^{3}}d\bar{k}%
_{2}\int_{\mathbb{R}^{3}}d\bar{\xi}_{2}\Big[ \widehat{G_{2,2}}\left( k_{1},%
\bar{k}_{2};\xi _{1}+\bar{k}_{2}-\bar{\xi}_{2},\bar{\xi}_{2};t\right)- \nonumber \\ 
-\widehat{G_{2,2}}\left( k_{1}-\bar{k}_{2}+\bar{\xi}_{2},\bar{k}_{2};\xi _{1},%
\bar{\xi}_{2};t\right) \Big]  \label{S3E2a}
\end{eqnarray}%
and, also by (\ref{S3E1} ),
\begin{eqnarray}
i\left( 2\pi \right) ^{\frac{3}{2}}\delta \left( k_{1}-\xi _{1}\right)
\partial _{t}n\left( k_{1},t\right) =\frac{\varepsilon }{\left( 2\pi \right)
^{3}}\int_{\mathbb{R}^{3}}d\bar{k}_{2}\int_{\mathbb{R}^{3}}d\bar{\xi}_{2}%
\Big[ \widehat{G_{2,2}}\left( k_{1},\bar{k}_{2};\xi _{1}+\bar{k}_{2}-\bar{%
\xi}_{2},\bar{\xi}_{2};t\right)-\nonumber \\
 -\widehat{G_{2,2}}\left( k_{1}-\bar{k}_{2}+%
\bar{\xi}_{2},\bar{k}_{2};\xi _{1},\bar{\xi}_{2};t\right) \Big]
\label{S3E2}
\end{eqnarray}
\end{proof}

Notice that (\ref{S3E3}) implies that the right hand side contains a Dirac
mass $\delta \left( k_{1}-\xi _{1}\right) $ that will cancel out the
corresponding Dirac on the left hand side. We will examine this in detail later.
\subsection{An approximated equation for $\widehat G _{ 2,2 }$}

First, we derive an approximation for $\widehat{G_{2,2}}$.To 
this end we  consider the evolution equation for $\widehat{F_{2,2}}$ that it is given by (cf. (\ref{S2E3}))%
\begin{eqnarray}
&&i\partial _{t}\widehat{F_{2,2}}\left( k_{1},k_{2};\xi _{1},\xi
_{2};t\right)=  \nonumber \\
&&=\frac{1}{2}\left( -\sum_{j=1}^{2}\left\vert k_{j}\right\vert
^{2}+\sum_{j=1}^{2}\left\vert \xi _{j}\right\vert ^{2}\right) \widehat{%
F_{2,2}}\left( k_{1},k_{2} ;\xi _{1},\xi _{2};t\right)
+ \Xi\left( k_{1},k_{2};\xi _{1},\xi
_{2};t\right) \label{S3E4} 
\end{eqnarray}

\begin{eqnarray}
&&\Xi\left( k_{1},k_{2};\xi _{1},\xi
_{2};t\right)=-\frac{\varepsilon }{\left( 2\pi \right) ^{3}}\sum_{j=1}^{2}\int_{\mathbb{R%
}^{3}}d\bar{k}_{3}\int_{\mathbb{R}^{3}}d\bar{\xi}_{3}\widehat{F_{3,3}}\Big(
k_{1},k_{2},...k_{j-1},\nonumber \\
&&\hskip 3cm ,k_{j}-\bar{k}_{L+1}+\bar{\xi}_{M+1},k_{j+1},...,k_{3},%
\bar{k}_{L+1};\xi _{1},\xi _{2},\bar{\xi}_{3};t\Big) +  \nonumber \\
&&+\frac{\varepsilon }{\left( 2\pi \right) ^{3}}\sum_{j=1}^{2}\int_{\mathbb{R%
}^{3}}d\bar{k}_{3}\int_{\mathbb{R}^{3}}d\bar{\xi}_{3}\widehat{F_{3,3}}\Big(
k_{1},k_{2},\bar{k}_{3};\xi _{1},\xi _{2},...,\xi _{j-1},\nonumber \\
&&\hskip 4.9cm,\xi _{j}+\bar{k}%
_{L+1}-\bar{\xi}_{M+1},\xi _{j+1},...,\xi _{2},\bar{\xi}_{3};t\Big). 
\label{S3E4B}
\end{eqnarray}
We claim that if only the terms of higher order of magnitude in (\ref{S3E4}) are kept we obtain as an approximation of equation (\ref{S3E4}):
\begin{eqnarray}
&&i\partial _{t}\widehat{G_{2,2}}\left( k_{1},k_{2};\xi _{1},\xi
_{2};t\right)=  \nonumber  \\
&&=\frac{1}{2}\left( -\sum_{j=1}^{2}\left\vert k_{j}\right\vert
^{2}+\sum_{j=1}^{2}\left\vert \xi _{j}\right\vert ^{2}\right) \widehat{%
G_{2,2}}\left( k_{1},k_{2};\xi _{1},\xi _{2};t\right)+
 \Xi\left( k_{1},k_{2};\xi _{1},\xi
_{2};t\right) 
\label{S3E6}
\end{eqnarray}
Indeed, plugging first (
\ref{S2E8a}) into (\ref{S3E4}) we obtain%
\begin{eqnarray*}
&&i\left( \partial _{t}\widehat{F_{1,1}}\left( k_{1};\xi _{1}\right) \right) 
\widehat{F_{1,1}}\left( k_{2};\xi _{2}\right) +i\widehat{F_{1,1}}\left(
k_{1};\xi _{1}\right) \left( \partial _{t}\widehat{F_{1,1}}\left( k_{2};\xi
_{2}\right) \right) + \\
&&i\left( \partial _{t}\widehat{F_{1,1}}\left( k_{1};\xi _{2}\right) \right) 
\widehat{F_{1,1}}\left( k_{2};\xi _{1}\right) +i\widehat{F_{1,1}}\left(
k_{1};\xi _{2}\right) \left( \partial _{t}\widehat{F_{1,1}}\left( k_{2};\xi
_{1}\right) \right) + \\
&&+i\partial _{t}\widehat{G_{2,2}}\left( k_{1},k_{2};\xi _{1},\xi
_{2};t\right)= \\
&&=\frac{1}{2}\left( -\sum_{j=1}^{2}\left\vert k_{j}\right\vert
^{2}+\sum_{j=1}^{2}\left\vert \xi _{j}\right\vert ^{2}\right) \left[ 
\widehat{F_{1,1}}\left( k_{1};\xi _{1}\right) \widehat{F_{1,1}}\left(
k_{2};\xi _{2}\right) +\widehat{F_{1,1}}\left( k_{1};\xi _{2}\right) 
\widehat{F_{1,1}}\left( k_{2};\xi _{1}\right) \right] + \\
&&+\frac{1}{2}\left( -\sum_{j=1}^{2}\left\vert k_{j}\right\vert
^{2}+\sum_{j=1}^{2}\left\vert \xi _{j}\right\vert ^{2}\right) \widehat{%
G_{2,2}}\left( k_{1},k_{2};\xi _{1},\xi _{2};t\right)+
 \Xi\left( k_{1},k_{2};\xi _{1},\xi
_{2};t\right).
\end{eqnarray*}

Using (\ref{S3E1}) we obtain%
\begin{eqnarray*}
&&\frac{1}{2}\left( -\sum_{j=1}^{2}\left\vert k_{j}\right\vert
^{2}+\sum_{j=1}^{2}\left\vert \xi _{j}\right\vert ^{2}\right) \left[ 
\widehat{F_{1,1}}\left( k_{1};\xi _{1}\right) \widehat{F_{1,1}}\left(
k_{2};\xi _{2}\right) +\widehat{F_{1,1}}\left( k_{1};\xi _{2}\right) 
\widehat{F_{1,1}}\left( k_{2};\xi _{1}\right) \right]= \\
&&=\frac{\left( 2\pi \right) ^{3}}{2}\left( -\sum_{j=1}^{2}\left\vert
k_{j}\right\vert ^{2}+\sum_{j=1}^{2}\left\vert \xi _{j}\right\vert
^{2}\right) n\left( k_{1},t\right) n\left( k_{2},t\right) \left(
-\sum_{j=1}^{2}\left\vert k_{j}\right\vert ^{2}+\sum_{j=1}^{2}\left\vert \xi
_{j}\right\vert ^{2}\right) \cdot \\
&&\cdot \left[ \delta \left( k_{1}-\xi _{1}\right) \delta \left( k_{2}-\xi
_{2}\right) +\delta \left( k_{1}-\xi _{2}\right) \delta \left( k_{2}-\xi
_{1}\right) \right]=0
\end{eqnarray*}

Then%
\begin{eqnarray}
&&i\left( \partial _{t}\widehat{F_{1,1}}\left( k_{1};\xi _{1}\right) \right) 
\widehat{F_{1,1}}\left( k_{2};\xi _{2}\right) +i\widehat{F_{1,1}}\left(
k_{1};\xi _{1}\right) \left( \partial _{t}\widehat{F_{1,1}}\left( k_{2};\xi
_{2}\right) \right) +  \nonumber\\
&&i\left( \partial _{t}\widehat{F_{1,1}}\left( k_{1};\xi _{2}\right) \right) 
\widehat{F_{1,1}}\left( k_{2};\xi _{1}\right) +i\widehat{F_{1,1}}\left(
k_{1};\xi _{2}\right) \left( \partial _{t}\widehat{F_{1,1}}\left( k_{2};\xi
_{1}\right) \right) +  \nonumber \\
&&+i\partial _{t}\widehat{G_{2,2}}\left( k_{1},k_{2};\xi _{1},\xi
_{2};t\right)=  \nonumber \\
&&=\frac{1}{2}\left( -\sum_{j=1}^{2}\left\vert k_{j}\right\vert
^{2}+\sum_{j=1}^{2}\left\vert \xi _{j}\right\vert ^{2}\right) \widehat{%
G_{2,2}}\left( k_{1},k_{2} ;\xi _{1},\xi _{2};t\right)
+ \Xi\left( k_{1},k_{2};\xi _{1},\xi
_{2};t\right)\label{S3E5} 
\end{eqnarray}

We now remark that the terms $\partial _{t}\widehat{F_{1,1}}\left( k_{j};\xi
_{\ell }\right) $ are expected to be of the same  order of magnitude  than $\varepsilon  \left|
G _{ 2,2 }\right|$ due to (\ref{S3E2}). We will
check that the terms containing $\widehat{F_{3,3}}$ will give contributions
to $G _{ 2, 2 }$ of order $\varepsilon $ in (\ref{S3E5}). Therefore, the contribution
of the terms $\partial _{t}\widehat{F_{1,1}}\left( k_{j};\xi _{\ell }\right) 
$ can be expected to be negligible. We will then approximate (\ref{S3E5}) as%
\begin{eqnarray}
&&+i\partial _{t}\widehat{G_{2,2}}\left( k_{1},k_{2};\xi _{1},\xi
_{2};t\right)=  \nonumber \\
&&=\frac{1}{2}\left( -\sum_{j=1}^{2}\left\vert k_{j}\right\vert
^{2}+\sum_{j=1}^{2}\left\vert \xi _{j}\right\vert ^{2}\right) \widehat{%
G_{2,2}}\left( k_{1},k_{2};\xi _{1},\xi _{2};t\right)
+ \Xi\left( k_{1},k_{2};\xi _{1},\xi
_{2};t\right).\label{S3E6}
\end{eqnarray}
\subsection{Approximation of $\widehat{F_{3,3}}.$}

We now approximate $\widehat{F_{3,3}}.$ To this end, we will assume, that to
the leading order, this function is uncorrelated for all the relevant range
of times. Therefore, we assume that (\ref{S2E2}) is valid for times $t>0$
(using also the invariance under translations). We then assume the following
approximation%
\begin{equation}
\widehat{F_{3,3}}\left( k_{1},k_{2},k_{3};\xi _{1},\xi _{2},\xi
_{3};t\right) =\left( 2\pi \right) ^{\frac{9}{2}}\sum_{\sigma \in
S^{3}}\prod_{j=1}^{3}\left[ \delta \left( k_{j}-\xi _{\sigma \left( j\right)
}\right) n\left( k_{j},t\right) \right] \label{F3fact}
\end{equation}
or, in a more detailed form, dropping the dependence on $t,$ for the sake of
simplicity%
\begin{eqnarray}
&&\widehat{F_{3,3}}\left( k_{1},k_{2},k_{3};\xi _{1},\xi _{2},\xi
_{3};t\right)=  \label{S3E7} \\
&&=\left( 2\pi \right) ^{\frac{9}{2}}n\left( k_{1}\right) n\left(
k_{2}\right) n\left( k_{3}\right) \Bigg[ \delta \left( k_{1}-\xi _{1}\right)
\delta \left( k_{2}-\xi _{2}\right) \delta \left( k_{3}-\xi _{3}\right)+\nonumber\\
&&
+\delta \left( k_{1}-\xi _{1}\right) \delta \left( k_{2}-\xi _{3}\right)
\delta \left( k_{3}-\xi _{2}\right) +  \nonumber \\
&&+  \delta \left( k_{1}-\xi _{2}\right) \delta \left( k_{2}-\xi
_{1}\right) \delta \left( k_{3}-\xi _{3}\right) +\delta \left( k_{1}-\xi
_{2}\right) \delta \left( k_{2}-\xi _{3}\right) \delta \left( k_{3}-\xi
_{1}\right)   +  \nonumber \\
&&+  \delta \left( k_{1}-\xi _{3}\right) \delta \left( k_{2}-\xi
_{1}\right) \delta \left( k_{3}-\xi _{2}\right) +\delta \left( k_{1}-\xi
_{3}\right) \delta \left( k_{2}-\xi _{2}\right) \delta \left( k_{3}-\xi
_{1}\right) \Bigg]  \nonumber
\end{eqnarray}
Let us now denote by  $\widetilde {\Xi}\left( k_{1},k_{2};\xi _{1},\xi_{2};t\right)$ the approximation of $\Xi\left( k_{1},k_{2};\xi _{1},\xi_{2};t\right)$ approximating  $\widehat F _{ 3, 3 }$ by (\ref{F3fact}).
It is now necessary to compute in (\ref{S3E6}) the term
$\widetilde {\Xi}\left( k_{1},k_{2};\xi _{1},\xi_{2};t\right)$ defined in (\ref{S3E4B}).
\begin{proposition}
\begin{eqnarray}
&&\hskip -2cm \widetilde {\Xi}\left( k_{1},k_{2};\xi _{1},\xi_{2};t\right)=
2\cdot \left( 2\pi \right) ^{\frac{9}{2}}n\left( k_{1}\right) n\left(
k_{2}\right) \left( n\left( \xi _{2}\right) +n\left( \xi _{1}\right) \right)
\delta \left( k_{1}+k_{2}-\xi _{1}-\xi _{2}\right) -  \nonumber \\
&&\hskip 0.75cm-2\cdot \left( 2\pi \right) ^{\frac{9}{2}}\left[ \left( n\left(
k_{1}\right) +n\left( k_{2}\right) \right) n\left( \xi _{1}\right) n\left(
\xi _{2}\right) \delta \left( k_{1}+k_{2}-\xi _{2}-\xi _{1}\right) \right].
\end{eqnarray}
\end{proposition}
\begin{proof}
 We first remark,
\begin{eqnarray}
&&\sum_{j=1}^{2}\left[ \widehat{F_{3,3}}\left( k_{1},k_{2},\bar{k}_{3};\xi
_{1},\xi _{2},...,\xi _{j-1},\xi _{j}+\bar{k}_{3}-\bar{\xi}_{3},\xi
_{j+1},...,\xi _{2},\bar{\xi}_{3};t\right) \right.\nonumber \\
&&\left. -\widehat{F_{3,3}}\left( k_{1},k_{2},...k_{j-1},k_{j}-\bar{k}_{3}+%
\bar{\xi}_{3},k_{j+1},...,k_{3},\bar{k}_{3};\xi _{1},\xi _{2},\bar{\xi}%
_{3};t\right) \right]\nonumber\\
&&=\widehat{F_{3,3}}\left( k_{1},k_{2},\bar{k}_{3};\xi _{1}+\bar{k}_{3}-\bar{%
\xi}_{3},\xi _{2},\bar{\xi}_{3};t\right) -\widehat{F_{3,3}}\left( k_{1}-\bar{%
k}_{3}+\bar{\xi}_{3},k_{2},\bar{k}_{3};\xi _{1},\xi _{2},\bar{\xi}%
_{3};t\right) +  \nonumber \\
&&+\widehat{F_{3,3}}\left( k_{1},k_{2},\bar{k}_{3};\xi _{1},\xi _{2}+\bar{k}%
_{3}-\bar{\xi}_{3},\bar{\xi}_{3};t\right) -\widehat{F_{3,3}}\left(
k_{1},k_{2}-\bar{k}_{3}+\bar{\xi}_{3},\bar{k}_{3};\xi _{1},\xi _{2},\bar{\xi}%
_{3};t\right) \label{T1E8}
\end{eqnarray}

We now plug (\ref{S3E7}) into this formula. We first have

\begin{eqnarray*}
&&\widehat{F_{3,3}}\left( k_{1},k_{2},\bar{k}_{3};\xi _{1}+\bar{k}_{3}-\bar{%
\xi}_{3},\xi _{2},\bar{\xi}_{3};t\right) =\\
&&=\left( 2\pi \right) ^{\frac{9}{2}}n\left( k_{1}\right) n\left(
k_{2}\right) n\left( \bar{k}_{3}\right) \Big[ \delta \left( k_{1}-\left(
\xi _{1}+\bar{k}_{3}-\bar{\xi}_{3}\right) \right) \delta \left( k_{2}-\xi
_{2}\right) \delta \left( \bar{k}_{3}-\bar{\xi}_{3}\right)+\\
&&\hskip 4cm  +\delta \left(
k_{1}-\left( \xi _{1}+\bar{k}_{3}-\bar{\xi}_{3}\right) \right) \delta \left(
k_{2}-\bar{\xi}_{3}\right) \delta \left( \bar{k}_{3}-\xi _{2}\right) + 
\\
&&\hskip 4cm   +\delta \left( k_{1}-\xi _{2}\right) \delta \left( k_{2}-\left( \xi
_{1}+\bar{k}_{3}-\bar{\xi}_{3}\right) \right) \delta \left( \bar{k}_{3}-\bar{%
\xi}_{3}\right) +\\
&&\hskip 4cm +\delta \left( k_{1}-\xi _{2}\right) \delta \left( k_{2}-%
\bar{\xi}_{3}\right) \delta \left( \bar{k}_{3}-\left( \xi _{1}+\bar{k}_{3}-%
\bar{\xi}_{3}\right) \right)   + \\
&&\hskip 4cm +  \delta \left( k_{1}-\bar{\xi}_{3}\right) \delta \left(
k_{2}-\left( \xi _{1}+\bar{k}_{3}-\bar{\xi}_{3}\right) \right) \delta \left( 
\bar{k}_{3}-\xi _{2}\right)+\\
&&\hskip 4cm  +\delta \left( k_{1}-\bar{\xi}_{3}\right) \delta
\left( k_{2}-\xi _{2}\right) \delta \left( \bar{k}_{3}-\left( \xi _{1}+\bar{k%
}_{3}-\bar{\xi}_{3}\right) \right) \Big]
\end{eqnarray*}

Then%
\begin{eqnarray*}
&&\widehat{F_{3,3}}\left( k_{1},k_{2},\bar{k}_{3};\xi _{1}+\bar{k}_{3}-\bar{%
\xi}_{3},\xi _{2},\bar{\xi}_{3};t\right) \\
&&=\left( 2\pi \right) ^{\frac{9}{2}}n\left( k_{1}\right) n\left(
k_{2}\right) n\left( \bar{k}_{3}\right) \Big[ \delta \left( k_{1}+\bar{\xi}%
_{3}-\xi _{1}-\bar{k}_{3}\right) \delta \left( k_{2}-\xi _{2}\right) \delta
\left( \bar{k}_{3}-\bar{\xi}_{3}\right)+\\
&& \hskip 4cm  +\delta \left( k_{1}-\xi _{1}-\bar{k}%
_{3}+\bar{\xi}_{3}\right) \delta \left( k_{2}-\bar{\xi}_{3}\right) \delta
\left( \bar{k}_{3}-\xi _{2}\right) +  \\
&&\hskip 4cm + \delta \left( k_{1}-\xi _{2}\right) \delta \left( k_{2}-\xi _{1}-%
\bar{k}_{3}+\bar{\xi}_{3}\right) \delta \left( \bar{k}_{3}-\bar{\xi}%
_{3}\right)+\\
&& \hskip 4cm  +\delta \left( k_{1}-\xi _{2}\right) \delta \left( k_{2}-\bar{\xi%
}_{3}\right) \delta \left( \bar{k}_{3}-\xi _{1}-\bar{k}_{3}+\bar{\xi}%
_{3}\right)   + \\
&&\hskip 4cm  + \delta \left( k_{1}-\bar{\xi}_{3}\right) \delta \left( k_{2}-\xi
_{1}-\bar{k}_{3}+\bar{\xi}_{3}\right) \delta \left( \bar{k}_{3}-\xi
_{2}\right) +\\
&&\hskip 4cm  +\delta \left( k_{1}-\bar{\xi}_{3}\right) \delta \left(
k_{2}-\xi _{2}\right) \delta \left( \bar{k}_{3}-\xi _{1}-\bar{k}_{3}+\bar{\xi%
}_{3}\right) \Big]
\end{eqnarray*}

Thus,
\begin{eqnarray}
&&\widehat{F_{3,3}}\left( k_{1},k_{2},\bar{k}_{3};\xi _{1}+\bar{k}_{3}-\bar{%
\xi}_{3},\xi _{2},\bar{\xi}_{3};t\right)= \nonumber \\
&&=\left( 2\pi \right) ^{\frac{9}{2}}n\left( k_{1}\right) n\left(
k_{2}\right) n\left( \bar{k}_{3}\right) \Big[ \delta \left( k_{1}-\xi
_{1}\right) \delta \left( k_{2}-\xi _{2}\right) \delta \left( \bar{k}_{3}-%
\bar{\xi}_{3}\right)+\nonumber\\
&&\hskip 5.2cm +\delta \left( k_{1}-\xi _{1}-\xi _{2}+k_{2}\right)
\delta \left( k_{2}-\bar{\xi}_{3}\right) \delta \left( \bar{k}_{3}-\xi
_{2}\right) +  \nonumber \\
&&\hskip 1.4cm+  \delta \left( k_{1}-\xi _{2}\right) \delta \left( k_{2}-\xi
_{1}\right) \delta \left( \bar{k}_{3}-\bar{\xi}_{3}\right) +\delta \left(
k_{1}-\xi _{2}\right) \delta \left( k_{2}-\bar{\xi}_{3}\right) \delta \left(
\xi _{1}-k_{2}\right) +  \nonumber \\
&&+  \delta \left( k_{1}-\bar{\xi}_{3}\right) \delta \left( k_{2}-\xi
_{1}-\xi _{2}+k_{1}\right) \delta \left( \bar{k}_{3}-\xi _{2}\right) +\delta
\left( k_{1}-\bar{\xi}_{3}\right) \delta \left( k_{2}-\xi _{2}\right) \delta
\left( k_{1}-\xi _{1}\right) \Big]   \label{S3E8}
\end{eqnarray}

We now have that the third term in (\ref{T1E8}) can be approximated as%
\begin{eqnarray*}
&&\widehat{F_{3,3}}\left( k_{1},k_{2},\bar{k}_{3};\xi _{1},\xi _{2}+\bar{k}%
_{3}-\bar{\xi}_{3},\bar{\xi}_{3};t\right)=\nonumber \\
&&=\left( 2\pi \right) ^{\frac{9}{2}}n\left( k_{1}\right) n\left(
k_{2}\right) n\left( \bar{k}_{3}\right) \Big[ \delta \left( k_{1}-\xi
_{1}\right) \delta \left( k_{2}-\left( \xi _{2}+\bar{k}_{3}-\bar{\xi}%
_{3}\right) \right) \delta \left( \bar{k}_{3}-\bar{\xi}_{3}\right)+\nonumber\\
&&\hskip 4.1cm +\delta
\left( k_{1}-\xi _{1}\right) \delta \left( k_{2}-\bar{\xi}_{3}\right) \delta
\left( \bar{k}_{3}-\left( \xi _{2}+\bar{k}_{3}-\bar{\xi}_{3}\right) \right)
+  \nonumber \\
&&\hskip 4.1cm+  \delta \left( k_{1}-\left( \xi _{2}+\bar{k}_{3}-\bar{\xi}%
_{3}\right) \right) \delta \left( k_{2}-\xi _{1}\right) \delta \left( \bar{k}%
_{3}-\bar{\xi}_{3}\right)+ \nonumber\\
&&\hskip 4.1cm +\delta \left( k_{1}-\left( \xi _{2}+\bar{k}_{3}-%
\bar{\xi}_{3}\right) \right) \delta \left( k_{2}-\bar{\xi}_{3}\right) \delta
\left( \bar{k}_{3}-\xi _{1}\right)   +  \nonumber \\
&&\hskip 4.1cm+  \delta \left( k_{1}-\bar{\xi}_{3}\right) \delta \left( k_{2}-\xi
_{1}\right) \delta \left( \bar{k}_{3}-\left( \xi _{2}+\bar{k}_{3}-\bar{\xi}%
_{3}\right) \right) + \nonumber\\
&&\hskip 4.1cm+\delta \left( k_{1}-\bar{\xi}_{3}\right) \delta \left(
k_{2}-\left( \xi _{2}+\bar{k}_{3}-\bar{\xi}_{3}\right) \right) \delta \left( 
\bar{k}_{3}-\xi _{1}\right) \Big]  \nonumber
\end{eqnarray*}

Then%
\begin{eqnarray}
&&\widehat{F_{3,3}}\left( k_{1},k_{2},\bar{k}_{3};\xi _{1},\xi _{2}+\bar{k}%
_{3}-\bar{\xi}_{3},\bar{\xi}_{3};t\right)= \nonumber \\ 
&&=\left( 2\pi \right) ^{\frac{9}{2}}n\left( k_{1}\right) n\left(
k_{2}\right) n\left( \bar{k}_{3}\right) \Big[ \delta \left( k_{1}-\xi
_{1}\right) \delta \left( k_{2}-\xi _{2}\right) \delta \left( \bar{k}_{3}-%
\bar{\xi}_{3}\right)+\nonumber\\
&& +\delta \left( k_{1}-\xi _{1}\right) \delta \left(
k_{2}-\xi _{2}\right) \delta \left( \bar{\xi}_{3}-\xi _{2}\right) +  
\nonumber \\
&&+  \delta \left( k_{1}-\xi _{2}\right) \delta \left( k_{2}-\xi
_{1}\right) \delta \left( \bar{k}_{3}-\bar{\xi}_{3}\right) +\delta \left(
k_{1}+k_{2}-\xi _{2}-\xi _{1}\right) \delta \left( k_{2}-\bar{\xi}%
_{3}\right) \delta \left( \bar{k}_{3}-\xi _{1}\right)  +  \nonumber \\
&&+  \delta \left( k_{1}-\xi _{2}\right) \delta \left( k_{2}-\xi
_{1}\right) \delta \left( \xi _{2}-\bar{\xi}_{3}\right) +\delta \left( k_{1}-%
\bar{\xi}_{3}\right) \delta \left( k_{2}+k_{1}-\xi _{2}-\xi _{1}\right)
\delta \left( \bar{k}_{3}-\xi _{1}\right) \Big]  \label{S3E9}
\end{eqnarray}

The second term in (\ref{T1E8}) can be approximated as

\begin{eqnarray*}
&&\widehat{F_{3,3}}\left( k_{1}-\bar{k}_{3}+\bar{\xi}_{3},k_{2},\bar{k}%
_{3};\xi _{1},\xi _{2},\bar{\xi}_{3};t\right)= \\
&&=\left( 2\pi \right) ^{\frac{9}{2}}n\left( k_{1}-\bar{k}_{3}+\bar{\xi}%
_{3}\right) n\left( k_{2}\right) n\left( \bar{k}_{3}\right) \Big[ \delta
\left( k_{1}-\bar{k}_{3}+\bar{\xi}_{3}-\xi _{1}\right) \delta \left(
k_{2}-\xi _{2}\right) \delta \left( \bar{k}_{3}-\bar{\xi}_{3}\right)+\\
&&\hskip 5.6cm +\delta
\left( k_{1}-\bar{k}_{3}+\bar{\xi}_{3}-\xi _{1}\right) \delta \left( k_{2}-%
\bar{\xi}_{3}\right) \delta \left( \bar{k}_{3}-\xi _{2}\right) +  \\
&&\hskip 5.6cm+  \delta \left( k_{1}-\bar{k}_{3}+\bar{\xi}_{3}-\xi _{2}\right)
\delta \left( k_{2}-\xi _{1}\right) \delta \left( \bar{k}_{3}-\bar{\xi}%
_{3}\right)+\\
&& \hskip 5.6cm+\delta \left( k_{1}-\bar{k}_{3}+\bar{\xi}_{3}-\xi _{2}\right)
\delta \left( k_{2}-\bar{\xi}_{3}\right) \delta \left( \bar{k}_{3}-\xi
_{1}\right)  + \\
&&\hskip 5.6cm+  \delta \left( k_{1}-\bar{k}_{3}+\bar{\xi}_{3}-\bar{\xi}_{3}\right)
\delta \left( k_{2}-\xi _{1}\right) \delta \left( \bar{k}_{3}-\xi
_{2}\right)+\\
&&\hskip 5.6cm +\delta \left( k_{1}-\bar{k}_{3}+\bar{\xi}_{3}-\bar{\xi}%
_{3}\right) \delta \left( k_{2}-\xi _{2}\right) \delta \left( \bar{k}%
_{3}-\xi _{1}\right) \Big]
\end{eqnarray*}

Then, after some simplifications we obtain%
\begin{eqnarray}
&&\widehat{F_{3,3}}\left( k_{1}-\bar{k}_{3}+\bar{\xi}_{3},k_{2},\bar{k}%
_{3};\xi _{1},\xi _{2},\bar{\xi}_{3};t\right)= \nonumber \\
&&=\left( 2\pi \right) ^{\frac{9}{2}}n\left( k_{1}-\bar{k}_{3}+\bar{\xi}%
_{3}\right) n\left( k_{2}\right) n\left( \bar{k}_{3}\right) \Big[ \delta
\left( k_{1}-\xi _{1}\right) \delta \left( k_{2}-\xi _{2}\right) \delta
\left( \bar{k}_{3}-\bar{\xi}_{3}\right)+\nonumber\\
&&\hskip 5.3cm  +\delta \left( k_{1}-\xi
_{2}+k_{2}-\xi _{1}\right) \delta \left( k_{2}-\bar{\xi}_{3}\right) \delta
\left( \bar{k}_{3}-\xi _{2}\right) +  \nonumber \\
&&+ \delta \left( k_{1}-\xi _{2}\right) \delta \left( k_{2}-\xi
_{1}\right) \delta \left( \bar{k}_{3}-\bar{\xi}_{3}\right) +\delta \left(
k_{1}-\xi _{1}+k_{2}-\xi _{2}\right) \delta \left( k_{2}-\bar{\xi}%
_{3}\right) \delta \left( \bar{k}_{3}-\xi _{1}\right)  +  \nonumber \\
&&\hskip 1.7cm  +  \delta \left( k_{1}-\xi _{2}\right) \delta \left( k_{2}-\xi
_{1}\right) \delta \left( \bar{k}_{3}-\xi _{2}\right) +\delta \left(
k_{1}-\xi _{1}\right) \delta \left( k_{2}-\xi _{2}\right) \delta \left( \bar{%
k}_{3}-\xi _{1}\right) \Big].   \label{S3E10}
\end{eqnarray}

Arguing in a similar manner, we can approximate the last term in (\ref{T1E8}%
) as%
\begin{eqnarray*}
&&\widehat{F_{3,3}}\left( k_{1},k_{2}-\bar{k}_{3}+\bar{\xi}_{3},\bar{k}%
_{3};\xi _{1},\xi _{2},\bar{\xi}_{3};t\right)= \\
&&=\left( 2\pi \right) ^{\frac{9}{2}}n\left( k_{1}\right) n\left( k_{2}-\bar{%
k}_{3}+\bar{\xi}_{3}\right) n\left( \bar{k}_{3}\right) \Big[ \delta \left(
k_{1}-\xi _{1}\right) \delta \left( k_{2}-\bar{k}_{3}+\bar{\xi}_{3}-\xi
_{2}\right) \delta \left( \bar{k}_{3}-\bar{\xi}_{3}\right) +\\
&&\hskip 6.9cm+\delta \left(
k_{1}-\xi _{1}\right) \delta \left( k_{2}-\bar{k}_{3}+\bar{\xi}_{3}-\bar{\xi}%
_{3}\right) \delta \left( \bar{k}_{3}-\xi _{2}\right) +  \\
&&+  \delta \left( k_{1}-\xi _{2}\right) \delta \left( k_{2}-\bar{k}%
_{3}+\bar{\xi}_{3}-\xi _{1}\right) \delta \left( \bar{k}_{3}-\bar{\xi}%
_{3}\right) +\delta \left( k_{1}-\xi _{2}\right) \delta \left( k_{2}-\bar{k}%
_{3}+\bar{\xi}_{3}-\bar{\xi}_{3}\right) \delta \left( \bar{k}_{3}-\xi
_{1}\right)  + \\
&&+ \delta \left( k_{1}-\bar{\xi}_{3}\right) \delta \left( k_{2}-\bar{k%
}_{3}+\bar{\xi}_{3}-\xi _{1}\right) \delta \left( \bar{k}_{3}-\xi
_{2}\right) +\delta \left( k_{1}-\bar{\xi}_{3}\right) \delta \left( k_{2}-%
\bar{k}_{3}+\bar{\xi}_{3}-\xi _{2}\right) \delta \left( \bar{k}_{3}-\xi
_{1}\right) \Big].
\end{eqnarray*}

Then%
\begin{eqnarray}
&&\widehat{F_{3,3}}\left( k_{1},k_{2}-\bar{k}_{3}+\bar{\xi}_{3},\bar{k}%
_{3};\xi _{1},\xi _{2},\bar{\xi}_{3};t\right) =\nonumber  \label{S3E11} \\
&&=\left( 2\pi \right) ^{\frac{9}{2}}n\left( k_{1}\right) n\left( k_{2}-\bar{%
k}_{3}+\bar{\xi}_{3}\right) n\left( \bar{k}_{3}\right) \Big[ \delta \left(
k_{1}-\xi _{1}\right) \delta \left( k_{2}-\xi _{2}\right) \delta \left( \bar{%
k}_{3}-\bar{\xi}_{3}\right)+\nonumber \\
&& \hskip 5.6cm +\delta \left( k_{1}-\xi _{1}\right) \delta
\left( k_{2}-\xi _{2}\right) \delta \left( \bar{k}_{3}-\xi _{2}\right)
+  \nonumber \\
&&\hskip 0.3cm+ \delta \left( k_{1}-\xi _{2}\right) \delta \left( k_{2}-\xi
_{1}\right) \delta \left( \bar{k}_{3}-\bar{\xi}_{3}\right) +\delta \left(
k_{1}-\xi _{2}\right) \delta \left( k_{2}-\xi _{1}\right) \delta \left( \bar{%
k}_{3}-\xi _{1}\right)  +  \nonumber \\
&&\hskip 0.3cm+  \delta \left( k_{1}-\bar{\xi}_{3}\right) \delta \left(
k_{2}+k_{1}-\xi _{2}-\xi _{1}\right) \delta \left( \bar{k}_{3}-\xi
_{2}\right)+\nonumber\\ 
&&\hskip 0.3cm+\delta \left( k_{1}-\bar{\xi}_{3}\right) \delta \left(
k_{2}+k_{1}-\xi _{1}-\xi _{2}\right) \delta \left( \bar{k}_{3}-\xi
_{1}\right) \Big].    \label{S3E11}  
\end{eqnarray}

We now compute%
\begin{eqnarray*}
&&\int_{\mathbb{R}^{3}}d\bar{k}_{3}\int_{\mathbb{R}^{3}}d\bar{\xi}_{3}\Big[ 
\widehat{F_{3,3}}\left( k_{1},k_{2},\bar{k}_{3};\xi _{1}+\bar{k}_{3}-\bar{\xi%
}_{3},\xi _{2},\bar{\xi}_{3};t\right) -\\
&&-\widehat{F_{3,3}}\left( k_{1}-\bar{k}%
_{3}+\bar{\xi}_{3},k_{2},\bar{k}_{3};\xi _{1},\xi _{2},\bar{\xi}%
_{3};t\right)  +  \widehat{F_{3,3}}\left( k_{1},k_{2},\bar{k}_{3};\xi _{1},\xi _{2}+%
\bar{k}_{3}-\bar{\xi}_{3},\bar{\xi}_{3};t\right)- \\
&&\hskip 6.2cm  -\widehat{F_{3,3}}\left(
k_{1},k_{2}-\bar{k}_{3}+\bar{\xi}_{3},\bar{k}_{3};\xi _{1},\xi _{2},\bar{\xi}%
_{3};t\right) \Big].
\end{eqnarray*}

Then, using (\ref{S3E8})%
\begin{eqnarray*}
&&\int_{\mathbb{R}^{3}}d\bar{k}_{3}\int_{\mathbb{R}^{3}}d\bar{\xi}_{3}%
\widehat{F_{3,3}}\left( k_{1},k_{2},\bar{k}_{3};\xi _{1}+\bar{k}_{3}-\bar{\xi%
}_{3},\xi _{2},\bar{\xi}_{3};t\right)= \\
&&=\left( 2\pi \right) ^{\frac{9}{2}}n\left( k_{1}\right) n\left(
k_{2}\right) \int_{\mathbb{R}^{3}}n\left( \bar{k}_{3}\right) d\bar{k}%
_{3}\int_{\mathbb{R}^{3}}d\bar{\xi}_{3}\Big[ \delta \left( k_{1}-\xi
_{1}\right) \delta \left( k_{2}-\xi _{2}\right) \delta \left( \bar{k}_{3}-%
\bar{\xi}_{3}\right) +\\
&&\hskip 5.2cm +\delta \left( k_{1}-\xi _{1}-\xi _{2}+k_{2}\right)
\delta \left( k_{2}-\bar{\xi}_{3}\right) \delta \left( \bar{k}_{3}-\xi
_{2}\right) + \\
&&\hskip 1.4cm+  \delta \left( k_{1}-\xi _{2}\right) \delta \left( k_{2}-\xi
_{1}\right) \delta \left( \bar{k}_{3}-\bar{\xi}_{3}\right) +\delta \left(
k_{1}-\xi _{2}\right) \delta \left( k_{2}-\bar{\xi}_{3}\right) \delta \left(
\xi _{1}-k_{2}\right) + \\
&&+  \delta \left( k_{1}-\bar{\xi}_{3}\right) \delta \left( k_{2}-\xi
_{1}-\xi _{2}+k_{1}\right) \delta \left( \bar{k}_{3}-\xi _{2}\right) +\delta
\left( k_{1}-\bar{\xi}_{3}\right) \delta \left( k_{2}-\xi _{2}\right) \delta
\left( k_{1}-\xi _{1}\right) \Big].
\end{eqnarray*}

Therefore, eliminating all the integrals that contain a Dirac mass in the
variable $\bar{\xi}_{3},$ or, more precisely, using $\int_{\mathbb{R}^{3}}d%
\bar{\xi}_{3}\delta \left( a-\bar{\xi}_{3}\right) =1,$ we obtain%
\begin{eqnarray*}
&&\int_{\mathbb{R}^{3}}d\bar{k}_{3}\int_{\mathbb{R}^{3}}d\bar{\xi}_{3}%
\widehat{F_{3,3}}\left( k_{1},k_{2},\bar{k}_{3};\xi _{1}+\bar{k}_{3}-\bar{\xi%
}_{3},\xi _{2},\bar{\xi}_{3};t\right)= \\
&&=\left( 2\pi \right) ^{\frac{9}{2}}n\left( k_{1}\right) n\left(
k_{2}\right) \int_{\mathbb{R}^{3}}n\left( \bar{k}_{3}\right) d\bar{k}_{3}%
\Big[ \delta \left( k_{1}-\xi _{1}\right) \delta \left( k_{2}-\xi
_{2}\right) +\delta \left( k_{1}+k_{2}-\xi _{1}-\xi _{2}\right) \delta
\left( \bar{k}_{3}-\xi _{2}\right) +  \\
&&\hskip 1.5cm +  \delta \left( k_{1}-\xi _{2}\right) \delta \left( k_{2}-\xi
_{1}\right) +\delta \left( k_{1}-\xi _{2}\right) \delta \left( \xi
_{1}-k_{2}\right) +\delta \left( k_{1}+k_{2}-\xi _{1}-\xi _{2}\right) \delta
\left( \bar{k}_{3}-\xi _{2}\right)+\\
&&\hskip 10.5cm +\delta \left( k_{2}-\xi _{2}\right)
\delta \left( k_{1}-\xi _{1}\right) \Big]
\end{eqnarray*}%
or%
\begin{eqnarray}
&&\int_{\mathbb{R}^{3}}d\bar{k}_{3}\int_{\mathbb{R}^{3}}d\bar{\xi}_{3}%
\widehat{F_{3,3}}\left( k_{1},k_{2},\bar{k}_{3};\xi _{1}+\bar{k}_{3}-\bar{\xi%
}_{3},\xi _{2},\bar{\xi}_{3};t\right)= \nonumber\\
&&=2\cdot \left( 2\pi \right) ^{\frac{9}{2}}n\left( k_{1}\right) n\left(
k_{2}\right) \int_{\mathbb{R}^{3}}n\left( \bar{k}_{3}\right) d\bar{k}_{3}%
\Big[ \delta \left( k_{1}-\xi _{1}\right) \delta \left( k_{2}-\xi
_{2}\right) +\nonumber\\
&&\hskip 0.5cm +\delta \left( k_{1}+k_{2}-\xi _{1}-\xi _{2}\right) \delta
\left( \bar{k}_{3}-\xi _{2}\right) +\delta \left( k_{1}-\xi _{2}\right)
\delta \left( k_{2}-\xi _{1}\right) \Big].   \label{S4E1} 
\end{eqnarray}

As a matter of fact, the terms containing $\delta \left( k_{1}+k_{2}-\xi
_{1}-\xi _{2}\right) $ can be further simplified. This will be seen later.

We now compute%
\[
\int_{\mathbb{R}^{3}}d\bar{k}_{3}\int_{\mathbb{R}^{3}}d\bar{\xi}_{3}\widehat{%
F_{3,3}}\left( k_{1},k_{2},\bar{k}_{3};\xi _{1},\xi _{2}+\bar{k}_{3}-\bar{\xi%
}_{3},\bar{\xi}_{3};t\right) 
\]

Using (\ref{S3E9})%
\begin{eqnarray*}
&&\int_{\mathbb{R}^{3}}d\bar{k}_{3}\int_{\mathbb{R}^{3}}d\bar{\xi}_{3}%
\widehat{F_{3,3}}\left( k_{1},k_{2},\bar{k}_{3};\xi _{1},\xi _{2}+\bar{k}%
_{3}-\bar{\xi}_{3},\bar{\xi}_{3};t\right) =\\
&&=\left( 2\pi \right) ^{\frac{9}{2}}n\left( k_{1}\right) n\left(
k_{2}\right) \int_{\mathbb{R}^{3}}n\left( \bar{k}_{3}\right) d\bar{k}%
_{3}\int_{\mathbb{R}^{3}}d\bar{\xi}_{3}\Big[ \delta \left( k_{1}-\xi
_{1}\right) \delta \left( k_{2}-\xi _{2}\right) \delta \left( \bar{k}_{3}-%
\bar{\xi}_{3}\right) +\\
&&\hskip 6.9cm+\delta \left( k_{1}-\xi _{1}\right) \delta \left(
k_{2}-\xi _{2}\right) \delta \left( \bar{\xi}_{3}-\xi _{2}\right) + \\
&&+ \delta \left( k_{1}-\xi _{2}\right) \delta \left( k_{2}-\xi
_{1}\right) \delta \left( \bar{k}_{3}-\bar{\xi}_{3}\right) +\delta \left(
k_{1}+k_{2}-\xi _{2}-\xi _{1}\right) \delta \left( k_{2}-\bar{\xi}%
_{3}\right) \delta \left( \bar{k}_{3}-\xi _{1}\right)   + \\
&&\hskip 0.2cm+ \delta \left( k_{1}-\xi _{2}\right) \delta \left( k_{2}-\xi
_{1}\right) \delta \left( \xi _{2}-\bar{\xi}_{3}\right) +\delta \left( k_{1}-%
\bar{\xi}_{3}\right) \delta \left( k_{2}+k_{1}-\xi _{2}-\xi _{1}\right)
\delta \left( \bar{k}_{3}-\xi _{1}\right) \Big].
\end{eqnarray*}

Then%
\begin{eqnarray}
&&\int_{\mathbb{R}^{3}}d\bar{k}_{3}\int_{\mathbb{R}^{3}}d\bar{\xi}_{3}%
\widehat{F_{3,3}}\left( k_{1},k_{2},\bar{k}_{3};\xi _{1},\xi _{2}+\bar{k}%
_{3}-\bar{\xi}_{3},\bar{\xi}_{3};t\right)=  \nonumber \\
&&=\left( 2\pi \right) ^{\frac{9}{2}}n\left( k_{1}\right) n\left(
k_{2}\right) \int_{\mathbb{R}^{3}}n\left( \bar{k}_{3}\right) d\bar{k}_{3}%
\Big[ \delta \left( k_{1}-\xi _{1}\right) \delta \left( k_{2}-\xi
_{2}\right) +\delta \left( k_{1}-\xi _{1}\right) \delta \left( k_{2}-\xi
_{2}\right) +   \nonumber \\
&&+  \delta \left( k_{1}-\xi _{2}\right) \delta \left( k_{2}-\xi
_{1}\right) +\delta \left( k_{1}+k_{2}-\xi _{2}-\xi _{1}\right) \delta
\left( \bar{k}_{3}-\xi _{1}\right) +\delta \left( k_{1}-\xi _{2}\right)
\delta \left( k_{2}-\xi _{1}\right)+\nonumber\\
&&\hskip 7.5cm  +\delta \left( k_{2}+k_{1}-\xi _{2}-\xi
_{1}\right) \delta \left( \bar{k}_{3}-\xi _{1}\right) \Big]  \nonumber \\
&=&2\cdot \left( 2\pi \right) ^{\frac{9}{2}}n\left( k_{1}\right) n\left(
k_{2}\right) \int_{\mathbb{R}^{3}}n\left( \bar{k}_{3}\right) d\bar{k}_{3}%
\Big[ \delta \left( k_{1}-\xi _{1}\right) \delta \left( k_{2}-\xi
_{2}\right) +\delta \left( k_{1}-\xi _{2}\right) \delta \left( k_{2}-\xi
_{1}\right)+\nonumber\\
&&\hskip 7cm  +\delta \left( k_{2}+k_{1}-\xi _{2}-\xi _{1}\right) \delta
\left( \bar{k}_{3}-\xi _{1}\right) \Big]    \label{S4E2}
\end{eqnarray}

We now compute%
\begin{eqnarray*}
&&\int_{\mathbb{R}^{3}}d\bar{k}_{3}\int_{\mathbb{R}^{3}}d\bar{\xi}_{3}%
\widehat{F_{3,3}}\left( k_{1}-\bar{k}_{3}+\bar{\xi}_{3},k_{2},\bar{k}%
_{3};\xi _{1},\xi _{2},\bar{\xi}_{3};t\right) =\\
&&=\left( 2\pi \right) ^{\frac{9}{2}}n\left( k_{2}\right) \int_{\mathbb{R}%
^{3}}n\left( \bar{k}_{3}\right) d\bar{k}_{3}\int_{\mathbb{R}^{3}}d\bar{\xi}%
_{3}n\left( k_{1}-\bar{k}_{3}+\bar{\xi}_{3}\right) \cdot \\
&&\cdot \left[ \delta \left( k_{1}-\xi _{1}\right) \delta \left( k_{2}-\xi
_{2}\right) \delta \left( \bar{k}_{3}-\bar{\xi}_{3}\right) +\delta \left(
k_{1}-\xi _{2}+k_{2}-\xi _{1}\right) \delta \left( k_{2}-\bar{\xi}%
_{3}\right) \delta \left( \bar{k}_{3}-\xi _{2}\right) +\right. \\
&&+\left. \delta \left( k_{1}-\xi _{2}\right) \delta \left( k_{2}-\xi
_{1}\right) \delta \left( \bar{k}_{3}-\bar{\xi}_{3}\right) +\delta \left(
k_{1}-\xi _{1}+k_{2}-\xi _{2}\right) \delta \left( k_{2}-\bar{\xi}%
_{3}\right) \delta \left( \bar{k}_{3}-\xi _{1}\right) \right. + \\
&&+\left. \delta \left( k_{1}-\xi _{2}\right) \delta \left( k_{2}-\xi
_{1}\right) \delta \left( \bar{k}_{3}-\xi _{2}\right) +\delta \left(
k_{1}-\xi _{1}\right) \delta \left( k_{2}-\xi _{2}\right) \delta \left( \bar{%
k}_{3}-\xi _{1}\right) \right]
\end{eqnarray*}

We can now simplify the integrals with the form $\int_{\mathbb{R}^{3}}\left(
\cdot \cdot \cdot \right) d\bar{k}_{3}$ using the Dirac masses containing
terms like $\delta \left( \bar{k}_{3}-\cdot \right) .$ We then obtain%
\begin{eqnarray}
&&\int_{\mathbb{R}^{3}}d\bar{k}_{3}\int_{\mathbb{R}^{3}}d\bar{\xi}_{3}%
\widehat{F_{3,3}}\left( k_{1}-\bar{k}_{3}+\bar{\xi}_{3},k_{2},\bar{k}%
_{3};\xi _{1},\xi _{2},\bar{\xi}_{3};t\right)=  \nonumber \\
&&=\left( 2\pi \right) ^{\frac{9}{2}}n\left( k_{2}\right) \int_{\mathbb{R}%
^{3}}d\bar{\xi}_{3}\Big[ n\left( \bar{\xi}_{3}\right) n\left( k_{1}\right)
\delta \left( k_{1}-\xi _{1}\right) \delta \left( k_{2}-\xi _{2}\right)+\nonumber \\
&&+n\left( \xi _{2}\right) n\left( k_{1}+k_{2}-\xi _{2}\right) \delta \left(
k_{1}+k_{2}-\xi _{2}-\xi _{1}\right) \delta \left( k_{2}-\bar{\xi}%
_{3}\right) +   \nonumber \\
&&+  n\left( \bar{\xi}_{3}\right) n\left( k_{1}\right) \delta \left(
k_{1}-\xi _{2}\right) \delta \left( k_{2}-\xi _{1}\right) +\nonumber \\
&&+n\left( \xi
_{1}\right) n\left( k_{1}+k_{2}-\xi _{1}\right) \delta \left(
k_{1}+k_{2}-\xi _{1}-\xi _{2}\right) \delta \left( k_{2}-\bar{\xi}%
_{3}\right)   +  \nonumber \\
&&+  n\left( \xi _{2}\right) n\left( k_{1}-\xi _{2}+\bar{\xi}%
_{3}\right) \delta \left( k_{1}-\xi _{2}\right) \delta \left( k_{2}-\xi
_{1}\right)+\nonumber \\
&& +n\left( \xi _{1}\right) n\left( k_{1}-\xi _{1}+\bar{\xi}%
_{3}\right) \delta \left( k_{1}-\xi _{1}\right) \delta \left( k_{2}-\xi
_{2}\right) \Big].  \label{S4E3}
\end{eqnarray}

Then%
\begin{eqnarray}
&&\int_{\mathbb{R}^{3}}d\bar{k}_{3}\int_{\mathbb{R}^{3}}d\bar{\xi}_{3}%
\widehat{F_{3,3}}\left( k_{1}-\bar{k}_{3}+\bar{\xi}_{3},k_{2},\bar{k}%
_{3};\xi _{1},\xi _{2},\bar{\xi}_{3};t\right)= \\
&&=\left( 2\pi \right) ^{\frac{9}{2}}n\left( k_{1}\right) n\left(
k_{2}\right) \left( \int_{\mathbb{R}^{3}}d\bar{\xi}_{3}n\left( \bar{\xi}%
_{3}\right) \right) \delta \left( k_{1}-\xi _{1}\right) \delta \left(
k_{2}-\xi _{2}\right)+\nonumber \\
&& +\left( 2\pi \right) ^{\frac{9}{2}}n\left(
k_{2}\right) n\left( \xi _{2}\right) n\left( k_{1}+k_{2}-\xi _{2}\right)
\delta \left( k_{1}+k_{2}-\xi _{2}-\xi _{1}\right) +  \nonumber \\
&&+\left( 2\pi \right) ^{\frac{9}{2}}\left( \int_{\mathbb{R}^{3}}d\bar{\xi}%
_{3}n\left( \bar{\xi}_{3}\right) \right) n\left( k_{1}\right) n\left(
k_{2}\right) \delta \left( k_{1}-\xi _{2}\right) \delta \left( k_{2}-\xi
_{1}\right)+\nonumber \\
&& +\left( 2\pi \right) ^{\frac{9}{2}}n\left( \xi _{1}\right)
n\left( k_{2}\right) n\left( k_{1}+k_{2}-\xi _{1}\right) \delta \left(
k_{1}+k_{2}-\xi _{1}-\xi _{2}\right) +  \nonumber \\
&&+\left( 2\pi \right) ^{\frac{9}{2}}n\left( k_{2}\right) n\left( \xi
_{2}\right) \left( \int_{\mathbb{R}^{3}}d\bar{\xi}_{3}n\left( k_{1}-\xi _{2}+%
\bar{\xi}_{3}\right) \right) \delta \left( k_{1}-\xi _{2}\right) \delta
\left( k_{2}-\xi _{1}\right)+\nonumber\\
&& +\left( 2\pi \right) ^{\frac{9}{2}}n\left(
k_{2}\right) n\left( \xi _{1}\right) \left( \int_{\mathbb{R}^{3}}d\bar{\xi}%
_{3}n\left( k_{1}-\xi _{1}+\bar{\xi}_{3}\right) \right) \delta \left(
k_{1}-\xi _{1}\right) \delta \left( k_{2}-\xi _{2}\right)  \nonumber
\end{eqnarray}%
or, simplifying%
\begin{eqnarray}
&&\int_{\mathbb{R}^{3}}d\bar{k}_{3}\int_{\mathbb{R}^{3}}d\bar{\xi}_{3}%
\widehat{F_{3,3}}\left( k_{1}-\bar{k}_{3}+\bar{\xi}_{3},k_{2},\bar{k}%
_{3};\xi _{1},\xi _{2},\bar{\xi}_{3};t\right)=  \nonumber \\
&&=\left( 2\pi \right) ^{\frac{9}{2}}n\left( k_{1}\right) n\left(
k_{2}\right) \left( \int_{\mathbb{R}^{3}}d\bar{\xi}_{3}n\left( \bar{\xi}%
_{3}\right) \right) \delta \left( k_{1}-\xi _{1}\right) \delta \left(
k_{2}-\xi _{2}\right) +\nonumber\\
&&+\left( 2\pi \right) ^{\frac{9}{2}}n\left(
k_{2}\right) n\left( \xi _{2}\right) n\left( \xi _{1}\right) \delta \left(
k_{1}+k_{2}-\xi _{2}-\xi _{1}\right) +  \nonumber \\
&&+\left( 2\pi \right) ^{\frac{9}{2}}\left( \int_{\mathbb{R}^{3}}d\bar{\xi}%
_{3}n\left( \bar{\xi}_{3}\right) \right) n\left( k_{1}\right) n\left(
k_{2}\right) \delta \left( k_{1}-\xi _{2}\right) \delta \left( k_{2}-\xi
_{1}\right)+\nonumber \\
&& +\left( 2\pi \right) ^{\frac{9}{2}}n\left( \xi _{1}\right)
n\left( k_{2}\right) n\left( \xi _{2}\right) \delta \left( k_{1}+k_{2}-\xi
_{1}-\xi _{2}\right) +  \nonumber \\
&&+\left( 2\pi \right) ^{\frac{9}{2}}n\left( k_{1}\right) n\left(
k_{2}\right) \left( \int_{\mathbb{R}^{3}}d\bar{\xi}_{3}n\left( \bar{\xi}%
_{3}\right) \right) \delta \left( k_{1}-\xi _{2}\right) \delta \left(
k_{2}-\xi _{1}\right)+\nonumber \\
&& +\left( 2\pi \right) ^{\frac{9}{2}}n\left(
k_{1}\right) n\left( k_{2}\right) \left( \int_{\mathbb{R}^{3}}d\bar{\xi}%
_{3}n\left( \bar{\xi}_{3}\right) \right) \delta \left( k_{1}-\xi _{1}\right)
\delta \left( k_{2}-\xi _{2}\right)  \nonumber \\
&=&2\cdot \left( 2\pi \right) ^{\frac{9}{2}}n\left( k_{2}\right) \Bigg[
n\left( k_{1}\right) \int_{\mathbb{R}^{3}}d\bar{\xi}_{3}n\left( \bar{\xi}%
_{3}\right) \delta \left( k_{1}-\xi _{1}\right) \delta \left( k_{2}-\xi
_{2}\right) +\nonumber\\
&&+n\left( \xi _{1}\right) n\left( \xi _{2}\right) \delta \left(
k_{1}+k_{2}-\xi _{2}-\xi _{1}\right)  + n\left( k_{1}\right) \int_{\mathbb{R}^{3}}d\bar{\xi}_{3}n\left( 
\bar{\xi}_{3}\right) \delta \left( k_{1}-\xi _{2}\right) \delta \left(
k_{2}-\xi _{1}\right) \Bigg].  \label{S4E4}
\end{eqnarray}

We now compute%
\begin{eqnarray*}
&&\int_{\mathbb{R}^{3}}d\bar{k}_{3}\int_{\mathbb{R}^{3}}d\bar{\xi}_{3}%
\widehat{F_{3,3}}\left( k_{1},k_{2}-\bar{k}_{3}+\bar{\xi}_{3},\bar{k}%
_{3};\xi _{1},\xi _{2},\bar{\xi}_{3};t\right) =\\
&&=\left( 2\pi \right) ^{\frac{9}{2}}\int_{\mathbb{R}^{3}}d\bar{k}_{3}\int_{%
\mathbb{R}^{3}}d\bar{\xi}_{3}n\left( k_{1}\right) n\left( k_{2}-\bar{k}_{3}+%
\bar{\xi}_{3}\right) n\left( \bar{k}_{3}\right)\times \nonumber \\
&& \Big[ \delta \left(
k_{1}-\xi _{1}\right) \delta \left( k_{2}-\xi _{2}\right) \delta \left( \bar{%
k}_{3}-\bar{\xi}_{3}\right) +\delta \left( k_{1}-\xi _{1}\right) \delta
\left( k_{2}-\xi _{2}\right) \delta \left( \bar{k}_{3}-\xi _{2}\right)
+ \\
&&+  \delta \left( k_{1}-\xi _{2}\right) \delta \left( k_{2}-\xi
_{1}\right) \delta \left( \bar{k}_{3}-\bar{\xi}_{3}\right) +\delta \left(
k_{1}-\xi _{2}\right) \delta \left( k_{2}-\xi _{1}\right) \delta \left( \bar{%
k}_{3}-\xi _{1}\right)   + \\
&&+  \delta \left( k_{1}-\bar{\xi}_{3}\right) \delta \left(
k_{2}+k_{1}-\xi _{2}-\xi _{1}\right) \delta \left( \bar{k}_{3}-\xi
_{2}\right)+\\
&& +\delta \left( k_{1}-\bar{\xi}_{3}\right) \delta \left(
k_{2}+k_{1}-\xi _{1}-\xi _{2}\right) \delta \left( \bar{k}_{3}-\xi
_{1}\right) \Big]
\end{eqnarray*}

Then, computing the integral $\int_{\mathbb{R}^{3}}d\bar{k}_{3}$ we obtain%
\begin{eqnarray*}
&&\int_{\mathbb{R}^{3}}d\bar{k}_{3}\int_{\mathbb{R}^{3}}d\bar{\xi}_{3}%
\widehat{F_{3,3}}\left( k_{1},k_{2}-\bar{k}_{3}+\bar{\xi}_{3},\bar{k}%
_{3};\xi _{1},\xi _{2},\bar{\xi}_{3};t\right)= \\
&&=\left( 2\pi \right) ^{\frac{9}{2}}n\left( k_{1}\right) \int_{\mathbb{R}%
^{3}}d\bar{\xi}_{3}\Big[ n\left( k_{2}\right) n\left( \bar{\xi}_{3}\right)
\delta \left( k_{1}-\xi _{1}\right) \delta \left( k_{2}-\xi _{2}\right)+ \\
&&+n\left( k_{2}-\xi _{2}+\bar{\xi}_{3}\right) n\left( \xi _{2}\right) \delta
\left( k_{1}-\xi _{1}\right) \delta \left( k_{2}-\xi _{2}\right) + \\
&&+  n\left( k_{2}\right) n\left( \bar{\xi}_{3}\right) \delta \left(
k_{1}-\xi _{2}\right) \delta \left( k_{2}-\xi _{1}\right) +\\
&&+n\left( k_{2}-\xi
_{1}+\bar{\xi}_{3}\right) n\left( \xi _{1}\right) \delta \left( k_{1}-\xi
_{2}\right) \delta \left( k_{2}-\xi _{1}\right)  + \\
&&+  n\left( k_{2}-\xi _{2}+\bar{\xi}_{3}\right) n\left( \xi
_{2}\right) \delta \left( k_{1}-\bar{\xi}_{3}\right) \delta \left(
k_{2}+k_{1}-\xi _{2}-\xi _{1}\right) \\
&&+n\left( k_{2}-\xi _{1}+\bar{\xi}%
_{3}\right) n\left( \xi _{1}\right) \delta \left( k_{1}-\bar{\xi}_{3}\right)
\delta \left( k_{2}+k_{1}-\xi _{1}-\xi _{2}\right) \Big]
\end{eqnarray*}%
Therefore%
\begin{eqnarray*}
&&\int_{\mathbb{R}^{3}}d\bar{k}_{3}\int_{\mathbb{R}^{3}}d\bar{\xi}_{3}%
\widehat{F_{3,3}}\left( k_{1},k_{2}-\bar{k}_{3}+\bar{\xi}_{3},\bar{k}%
_{3};\xi _{1},\xi _{2},\bar{\xi}_{3};t\right) =\\
&&=\left( 2\pi \right) ^{\frac{9}{2}}n\left( k_{1}\right) \int_{\mathbb{R}%
^{3}}d\bar{\xi}_{3}\Big[ n\left( k_{2}\right) n\left( \bar{\xi}_{3}\right)
\delta \left( k_{1}-\xi _{1}\right) \delta \left( k_{2}-\xi _{2}\right)+\\
&&+n\left( k_{2}-\xi _{2}+\bar{\xi}_{3}\right) n\left( \xi _{2}\right) \delta
\left( k_{1}-\xi _{1}\right) \delta \left( k_{2}-\xi _{2}\right) +  \\
&&+  n\left( k_{2}\right) n\left( \bar{\xi}_{3}\right) \delta \left(
k_{1}-\xi _{2}\right) \delta \left( k_{2}-\xi _{1}\right) +n\left( k_{2}-\xi
_{1}+\bar{\xi}_{3}\right) n\left( \xi _{1}\right) \delta \left( k_{1}-\xi
_{2}\right) \delta \left( k_{2}-\xi _{1}\right)  + \\
&&+  n\left( k_{2}-\xi _{2}+k_{1}\right) n\left( \xi _{2}\right) \delta
\left( k_{1}-\bar{\xi}_{3}\right) \delta \left( k_{2}+k_{1}-\xi _{2}-\xi
_{1}\right)+\\
&& +n\left( k_{2}-\xi _{1}+k_{1}\right) n\left( \xi _{1}\right)
\delta \left( k_{1}-\bar{\xi}_{3}\right) \delta \left( k_{2}+k_{1}-\xi
_{1}-\xi _{2}\right) \Big]
\end{eqnarray*}%
or%
\begin{eqnarray*}
&&\int_{\mathbb{R}^{3}}d\bar{k}_{3}\int_{\mathbb{R}^{3}}d\bar{\xi}_{3}%
\widehat{F_{3,3}}\left( k_{1},k_{2}-\bar{k}_{3}+\bar{\xi}_{3},\bar{k}%
_{3};\xi _{1},\xi _{2},\bar{\xi}_{3};t\right)= \\
&&=\left( 2\pi \right) ^{\frac{9}{2}}n\left( k_{1}\right) \int_{\mathbb{R}%
^{3}}d\bar{\xi}_{3}\Big[ n\left( k_{2}\right) n\left( \bar{\xi}_{3}\right)
\delta \left( k_{1}-\xi _{1}\right) \delta \left( k_{2}-\xi _{2}\right)+\\
&&+n\left( \bar{\xi}_{3}\right) n\left( \xi _{2}\right) \delta \left(
k_{1}-\xi _{1}\right) \delta \left( k_{2}-\xi _{2}\right) +n\left( k_{2}\right) n\left( \bar{\xi}_{3}\right) \delta \left(
k_{1}-\xi _{2}\right) \delta \left( k_{2}-\xi _{1}\right) +  \\
&&+ n\left( \bar{\xi}%
_{3}\right) n\left( \xi _{1}\right) \delta \left( k_{1}-\xi _{2}\right)
\delta \left( k_{2}-\xi _{1}\right)  +n\left( \xi _{1}\right) n\left( \xi _{2}\right) \delta \left(
k_{1}-\bar{\xi}_{3}\right) \delta \left( k_{2}+k_{1}-\xi _{2}-\xi
_{1}\right) + \\
&&+n\left( \xi _{2}\right) n\left( \xi _{1}\right) \delta \left(
k_{1}-\bar{\xi}_{3}\right) \delta \left( k_{2}+k_{1}-\xi _{1}-\xi
_{2}\right) \Big].
\end{eqnarray*}

Computing the integrals in $\bar{\xi}_{3}$ we obtain%
\begin{eqnarray}
&&\int_{\mathbb{R}^{3}}d\bar{k}_{3}\int_{\mathbb{R}^{3}}d\bar{\xi}_{3}%
\widehat{F_{3,3}}\left( k_{1},k_{2}-\bar{k}_{3}+\bar{\xi}_{3},\bar{k}%
_{3};\xi _{1},\xi _{2},\bar{\xi}_{3};t\right)= \nonumber \\
&&=\left( 2\pi \right) ^{\frac{9}{2}}n\left( k_{1}\right) \Bigg[ n\left(
k_{2}\right) \left( \int_{\mathbb{R}^{3}}d\bar{\xi}_{3}n\left( \bar{\xi}%
_{3}\right) \right) \delta \left( k_{1}-\xi _{1}\right) \delta \left(
k_{2}-\xi _{2}\right)+\nonumber \\ 
&&+\left( \int_{\mathbb{R}^{3}}d\bar{\xi}_{3}n\left( 
\bar{\xi}_{3}\right) \right) n\left( k_{2}\right) \delta \left( k_{1}-\xi
_{1}\right) \delta \left( k_{2}-\xi _{2}\right) +  \nonumber \\
&&+  n\left( k_{2}\right) \left( \int_{\mathbb{R}^{3}}d\bar{\xi}%
_{3}n\left( \bar{\xi}_{3}\right) \right) \delta \left( k_{1}-\xi _{2}\right)
\delta \left( k_{2}-\xi _{1}\right)+\nonumber \\
&& +\left( \int_{\mathbb{R}^{3}}d\bar{\xi}%
_{3}n\left( \bar{\xi}_{3}\right) \right) n\left( k_{2}\right) \delta \left(
k_{1}-\xi _{2}\right) \delta \left( k_{2}-\xi _{1}\right)  + 
\nonumber \\
&&+  n\left( \xi _{1}\right) n\left( \xi _{2}\right) \delta \left(
k_{2}+k_{1}-\xi _{2}-\xi _{1}\right) +n\left( \xi _{2}\right) n\left( \xi
_{1}\right) \delta \left( k_{2}+k_{1}-\xi _{1}-\xi _{2}\right) \Bigg] 
\nonumber \\
&=&2\cdot \left( 2\pi \right) ^{\frac{9}{2}}n\left( k_{1}\right) \Bigg[
n\left( k_{2}\right) \left( \int_{\mathbb{R}^{3}}d\bar{\xi}_{3}n\left( \bar{%
\xi}_{3}\right) \right) \delta \left( k_{1}-\xi _{1}\right) \delta \left(
k_{2}-\xi _{2}\right)+\nonumber \\
&& +n\left( k_{2}\right) \left( \int_{\mathbb{R}^{3}}d%
\bar{\xi}_{3}n\left( \bar{\xi}_{3}\right) \right) \delta \left( k_{1}-\xi
_{2}\right) \delta \left( k_{2}-\xi _{1}\right) +   \nonumber \\
&&+  n\left( \xi _{1}\right) n\left( \xi _{2}\right) \delta \left(
k_{1}+k_{2}-\xi _{1}-\xi _{2}\right) \Bigg].   \label{S4E5}
\end{eqnarray}

Summarizing, we can collect the results obtained in (\ref{S4E1}), (\ref{S4E2}%
), (\ref{S4E4}), (\ref{S4E5}) as%
\begin{eqnarray}
&&\int_{\mathbb{R}^{3}}d\bar{k}_{3}\int_{\mathbb{R}^{3}}d\bar{\xi}_{3}%
\widehat{F_{3,3}}\left( k_{1},k_{2},\bar{k}_{3};\xi _{1}+\bar{k}_{3}-\bar{\xi%
}_{3},\xi _{2},\bar{\xi}_{3};t\right)=  \nonumber \\
&&=2\cdot \left( 2\pi \right) ^{\frac{9}{2}}n\left( k_{1}\right) n\left(
k_{2}\right) \int_{\mathbb{R}^{3}}n\left( \bar{k}_{3}\right) d\bar{k}_{3}%
\Big[ \delta \left( k_{1}-\xi _{1}\right) \delta \left( k_{2}-\xi
_{2}\right) +\nonumber\\
&&+\delta \left( k_{1}+k_{2}-\xi _{1}-\xi _{2}\right) \delta
\left( \bar{k}_{3}-\xi _{2}\right) +\delta \left( k_{1}-\xi _{2}\right)
\delta \left( k_{2}-\xi _{1}\right) \Big]  \label{T2E1}
\end{eqnarray}

\begin{eqnarray}
&&\int_{\mathbb{R}^{3}}d\bar{k}_{3}\int_{\mathbb{R}^{3}}d\bar{\xi}_{3}%
\widehat{F_{3,3}}\left( k_{1},k_{2},\bar{k}_{3};\xi _{1},\xi _{2}+\bar{k}%
_{3}-\bar{\xi}_{3},\bar{\xi}_{3};t\right) = \nonumber \\
&&=2\cdot \left( 2\pi \right) ^{\frac{9}{2}}n\left( k_{1}\right) n\left(
k_{2}\right) \int_{\mathbb{R}^{3}}n\left( \bar{k}_{3}\right) d\bar{k}_{3}%
\Big[ \delta \left( k_{1}-\xi _{1}\right) \delta \left( k_{2}-\xi
_{2}\right)+\nonumber\\
&& +\delta \left( k_{1}-\xi _{2}\right) \delta \left( k_{2}-\xi
_{1}\right) +\delta \left( k_{2}+k_{1}-\xi _{2}-\xi _{1}\right) \delta
\left( \bar{k}_{3}-\xi _{1}\right) \Big]  \label{T2E2}
\end{eqnarray}

\begin{eqnarray}
&&\int_{\mathbb{R}^{3}}d\bar{k}_{3}\int_{\mathbb{R}^{3}}d\bar{\xi}_{3}%
\widehat{F_{3,3}}\left( k_{1}-\bar{k}_{3}+\bar{\xi}_{3},k_{2},\bar{k}%
_{3};\xi _{1},\xi _{2},\bar{\xi}_{3};t\right) = \nonumber \\
&&=2\cdot \left( 2\pi \right) ^{\frac{9}{2}}n\left( k_{2}\right) \Bigg[
n\left( k_{1}\right) \int_{\mathbb{R}^{3}}d\bar{\xi}_{3}n\left( \bar{\xi}%
_{3}\right) \delta \left( k_{1}-\xi _{1}\right) \delta \left( k_{2}-\xi
_{2}\right)+\nonumber\\
&& +n\left( \xi _{1}\right) n\left( \xi _{2}\right) \delta \left(
k_{1}+k_{2}-\xi _{2}-\xi _{1}\right)  +  n\left( k_{1}\right) \int_{\mathbb{R}^{3}}d\bar{\xi}_{3}n\left( 
\bar{\xi}_{3}\right) \delta \left( k_{1}-\xi _{2}\right) \delta \left(
k_{2}-\xi _{1}\right) \Bigg]  \label{T2E3}
\end{eqnarray}%
\begin{eqnarray}
&&\int_{\mathbb{R}^{3}}d\bar{k}_{3}\int_{\mathbb{R}^{3}}d\bar{\xi}_{3}%
\widehat{F_{3,3}}\left( k_{1},k_{2}-\bar{k}_{3}+\bar{\xi}_{3},\bar{k}%
_{3};\xi _{1},\xi _{2},\bar{\xi}_{3};t\right)=  \nonumber \\
&&=2\cdot \left( 2\pi \right) ^{\frac{9}{2}}n\left( k_{1}\right) \Bigg[
n\left( k_{2}\right) \left( \int_{\mathbb{R}^{3}}d\bar{\xi}_{3}n\left( \bar{%
\xi}_{3}\right) \right) \delta \left( k_{1}-\xi _{1}\right) \delta \left(
k_{2}-\xi _{2}\right)+\nonumber\\
&& +n\left( k_{2}\right) \left( \int_{\mathbb{R}^{3}}d%
\bar{\xi}_{3}n\left( \bar{\xi}_{3}\right) \right) \delta \left( k_{1}-\xi
_{2}\right) \delta \left( k_{2}-\xi _{1}\right) +  n\left( \xi _{1}\right) n\left( \xi _{2}\right) \delta \left(
k_{1}+k_{2}-\xi _{1}-\xi _{2}\right) \Bigg].  \label{T2E4}
\end{eqnarray}

\bigskip

We now collect the terms contained in (\ref{T1E8}). More precisely, we
recall that we need to compute the integrals $\int_{\mathbb{R}^{3}}d\bar{k}%
_{3}\int_{\mathbb{R}^{3}}d\bar{\xi}_{3}$ of%
\begin{eqnarray*}
&&\widehat{F_{3,3}}\left( k_{1},k_{2},\bar{k}_{3};\xi _{1}+\bar{k}_{3}-\bar{%
\xi}_{3},\xi _{2},\bar{\xi}_{3};t\right) -\widehat{F_{3,3}}\left( k_{1}-\bar{%
k}_{3}+\bar{\xi}_{3},k_{2},\bar{k}_{3};\xi _{1},\xi _{2},\bar{\xi}%
_{3};t\right) + \\
&&+\widehat{F_{3,3}}\left( k_{1},k_{2},\bar{k}_{3};\xi _{1},\xi _{2}+\bar{k}%
_{3}-\bar{\xi}_{3},\bar{\xi}_{3};t\right) -\widehat{F_{3,3}}\left(
k_{1},k_{2}-\bar{k}_{3}+\bar{\xi}_{3},\bar{k}_{3};\xi _{1},\xi _{2},\bar{\xi}%
_{3};t\right)
\end{eqnarray*}%
Therefore, using (\ref{T2E1})-(\ref{T2E4}) we obtain that the integral of%
\begin{eqnarray*}
&&\int_{\mathbb{R}^{3}}d\bar{k}_{3}\int_{\mathbb{R}^{3}}d\bar{\xi}_{3}\Big[ 
\widehat{F_{3,3}}\left( k_{1},k_{2},\bar{k}_{3};\xi _{1}+\bar{k}_{3}-\bar{\xi%
}_{3},\xi _{2},\bar{\xi}_{3};t\right)-\nonumber\\
&&-\widehat{F_{3,3}}\left( k_{1}-\bar{k}%
_{3}+\bar{\xi}_{3},k_{2},\bar{k}_{3};\xi _{1},\xi _{2},\bar{\xi}%
_{3};t\right)+  \widehat{F_{3,3}}\left( k_{1},k_{2},\bar{k}_{3};\xi _{1},\xi _{2}+%
\bar{k}_{3}-\bar{\xi}_{3},\bar{\xi}_{3};t\right)-\nonumber   \\
&&\hskip 6.1cm  -\widehat{F_{3,3}}\left(
k_{1},k_{2}-\bar{k}_{3}+\bar{\xi}_{3},\bar{k}_{3};\xi _{1},\xi _{2},\bar{\xi}%
_{3};t\right) \Big]
\end{eqnarray*}%
is given by%
\begin{eqnarray*}
&&2\cdot \left( 2\pi \right) ^{\frac{9}{2}}n\left( k_{1}\right) n\left(
k_{2}\right) \int_{\mathbb{R}^{3}}n\left( \bar{k}_{3}\right) d\bar{k}_{3}%
\Big[ \delta \left( k_{1}-\xi _{1}\right) \delta \left( k_{2}-\xi
_{2}\right) +\\
&&+\delta \left( k_{1}+k_{2}-\xi _{1}-\xi _{2}\right) \delta
\left( \bar{k}_{3}-\xi _{2}\right) +\delta \left( k_{1}-\xi _{2}\right)
\delta \left( k_{2}-\xi _{1}\right) \Big] + \\
&&+2\cdot \left( 2\pi \right) ^{\frac{9}{2}}n\left( k_{1}\right) n\left(
k_{2}\right) \int_{\mathbb{R}^{3}}n\left( \bar{k}_{3}\right) d\bar{k}_{3}%
\Big[ \delta \left( k_{1}-\xi _{1}\right) \delta \left( k_{2}-\xi
_{2}\right)+\\
&& +\delta \left( k_{1}-\xi _{2}\right) \delta \left( k_{2}-\xi
_{1}\right) +\delta \left( k_{1}+k_{2}-\xi _{2}-\xi _{1}\right) \delta
\left( \bar{k}_{3}-\xi _{1}\right) \Big] - \\
&&-2\cdot \left( 2\pi \right) ^{\frac{9}{2}}n\left( k_{2}\right) \Bigg[
n\left( k_{1}\right) \int_{\mathbb{R}^{3}}d\bar{\xi}_{3}n\left( \bar{\xi}%
_{3}\right) \delta \left( k_{1}-\xi _{1}\right) \delta \left( k_{2}-\xi
_{2}\right)+\\
&& +n\left( \xi _{1}\right) n\left( \xi _{2}\right) \delta \left(
k_{1}+k_{2}-\xi _{2}-\xi _{1}\right) +n\left( k_{1}\right) \left( \int_{%
\mathbb{R}^{3}}d\bar{\xi}_{3}n\left( \bar{\xi}_{3}\right) \right) \delta
\left( k_{1}-\xi _{2}\right) \delta \left( k_{2}-\xi _{1}\right) \Bigg] \\
&&-2\cdot \left( 2\pi \right) ^{\frac{9}{2}}n\left( k_{1}\right) \Bigg[
n\left( k_{2}\right) \left( \int_{\mathbb{R}^{3}}d\bar{\xi}_{3}n\left( \bar{%
\xi}_{3}\right) \right) \delta \left( k_{1}-\xi _{1}\right) \delta \left(
k_{2}-\xi _{2}\right)+\\
&& +n\left( k_{2}\right) \left( \int_{\mathbb{R}^{3}}d%
\bar{\xi}_{3}n\left( \bar{\xi}_{3}\right) \right) \delta \left( k_{1}-\xi
_{2}\right) \delta \left( k_{2}-\xi _{1}\right) +n\left( \xi _{1}\right)
n\left( \xi _{2}\right) \delta \left( k_{1}+k_{2}-\xi _{1}-\xi _{2}\right) %
\Bigg].
\end{eqnarray*}
Some immediate cancellations that yield
\begin{eqnarray*}
&&2\cdot \left( 2\pi \right) ^{\frac{9}{2}}n\left( k_{1}\right) n\left(
k_{2}\right) \int_{\mathbb{R}^{3}}n\left( \bar{k}_{3}\right) d\bar{k}_{3}%
\Big[ \delta \left( k_{1}+k_{2}-\xi _{1}-\xi _{2}\right) \delta \left( \bar{%
k}_{3}-\xi _{2}\right)+\\
&&\hskip 8cm  +\delta \left( k_{1}-\xi _{2}\right) \delta \left(
k_{2}-\xi _{1}\right) \Big] + \\
&&+2\cdot \left( 2\pi \right) ^{\frac{9}{2}}n\left( k_{1}\right) n\left(
k_{2}\right) \int_{\mathbb{R}^{3}}n\left( \bar{k}_{3}\right) d\bar{k}_{3}%
\Big[ \delta \left( k_{1}-\xi _{2}\right) \delta \left( k_{2}-\xi
_{1}\right) +\\
&&\hskip 7cm +\delta \left( k_{1}+k_{2}-\xi _{2}-\xi _{1}\right) \delta
\left( \bar{k}_{3}-\xi _{1}\right) \Big] - \\
&&-2\cdot \left( 2\pi \right) ^{\frac{9}{2}}n\left( k_{2}\right) \Bigg[
n\left( \xi _{1}\right) n\left( \xi _{2}\right) \delta \left(
k_{1}+k_{2}-\xi _{2}-\xi _{1}\right)+\\
&&\hskip 6cm  +n\left( k_{1}\right) \left( \int_{%
\mathbb{R}^{3}}d\bar{\xi}_{3}n\left( \bar{\xi}_{3}\right) \right) \delta
\left( k_{1}-\xi _{2}\right) \delta \left( k_{2}-\xi _{1}\right) \Bigg] \\
&&-2\cdot \left( 2\pi \right) ^{\frac{9}{2}}n\left( k_{1}\right) \Bigg[
n\left( k_{2}\right) \left( \int_{\mathbb{R}^{3}}d\bar{\xi}_{3}n\left( \bar{%
\xi}_{3}\right) \right) \delta \left( k_{1}-\xi _{2}\right) \delta \left(
k_{2}-\xi _{1}\right)+\\
&&\hskip 8cm +n\left( \xi _{1}\right) n\left( \xi _{2}\right)
\delta \left( k_{1}+k_{2}-\xi _{1}-\xi _{2}\right) \Bigg].
\end{eqnarray*}

We can now combine several terms, to obtain%
\begin{eqnarray*}
&&2\cdot \left( 2\pi \right) ^{\frac{9}{2}}n\left( k_{1}\right) n\left(
k_{2}\right) \int_{\mathbb{R}^{3}}n\left( \bar{k}_{3}\right) d\bar{k}_{3}%
\Big[ \delta \left( k_{1}+k_{2}-\xi _{1}-\xi _{2}\right) \left( \delta
\left( \bar{k}_{3}-\xi _{2}\right) +\delta \left( \bar{k}_{3}-\xi
_{1}\right) \right)+\\
&&\hskip 9cm  +2\delta \left( k_{1}-\xi _{2}\right) \delta \left(
k_{2}-\xi _{1}\right) \Big] - \\
&&-2\cdot \left( 2\pi \right) ^{\frac{9}{2}}\Big[ \left( n\left(
k_{1}\right) +n\left( k_{2}\right) \right) n\left( \xi _{1}\right) n\left(
\xi _{2}\right) \delta \left( k_{1}+k_{2}-\xi _{2}-\xi _{1}\right)+\\
&&\hskip 4cm   +2n\left(
k_{1}\right) n\left( k_{2}\right) \left( \int_{\mathbb{R}^{3}}d\bar{\xi}%
_{3}n\left( \bar{\xi}_{3}\right) \right) \delta \left( k_{1}-\xi _{2}\right)
\delta \left( k_{2}-\xi _{1}\right) \Big].
\end{eqnarray*}

Cancelling the second and fourth we finally obtain%
\begin{eqnarray*}
&&2\cdot \left( 2\pi \right) ^{\frac{9}{2}}n\left( k_{1}\right) n\left(
k_{2}\right) \int_{\mathbb{R}^{3}}n\left( \bar{k}_{3}\right) d\bar{k}_{3}%
\left[ \delta \left( k_{1}+k_{2}-\xi _{1}-\xi _{2}\right) \left( \delta
\left( \bar{k}_{3}-\xi _{2}\right) +\delta \left( \bar{k}_{3}-\xi
_{1}\right) \right) \right] - \\
&&-2\cdot \left( 2\pi \right) ^{\frac{9}{2}}\left[ \left( n\left(
k_{1}\right) +n\left( k_{2}\right) \right) n\left( \xi _{1}\right) n\left(
\xi _{2}\right) \delta \left( k_{1}+k_{2}-\xi _{2}-\xi _{1}\right) \right]
\end{eqnarray*}%
or, equivalently%
\begin{eqnarray*}
&&2\cdot \left( 2\pi \right) ^{\frac{9}{2}}n\left( k_{1}\right) n\left(
k_{2}\right) \left( n\left( \xi _{2}\right) +n\left( \xi _{1}\right) \right)
\delta \left( k_{1}+k_{2}-\xi _{1}-\xi _{2}\right) - \\
&&-2\cdot \left( 2\pi \right) ^{\frac{9}{2}}\left[ \left( n\left(
k_{1}\right) +n\left( k_{2}\right) \right) n\left( \xi _{1}\right) n\left(
\xi _{2}\right) \delta \left( k_{1}+k_{2}-\xi _{2}-\xi _{1}\right) \right]
\end{eqnarray*}

Notice that the form of these terms, in particular the form of the Dirac
masses $\delta \left( k_{1}+k_{2}-\xi _{1}-\xi _{2}\right) ,$ implies the
conservation of the momentum variable. These terms give the expected form of
a kinetic equation for the function $n\left( k\right) .$ More precisely, we
have obtained%
\begin{eqnarray}
&&\int_{\mathbb{R}^{3}}d\bar{k}_{3}\int_{\mathbb{R}^{3}}d\bar{\xi}_{3}\left[ 
\widehat{F_{3,3}}\left( k_{1},k_{2},\bar{k}_{3};\xi _{1}+\bar{k}_{3}-\bar{\xi%
}_{3},\xi _{2},\bar{\xi}_{3};t\right) -\widehat{F_{3,3}}\left( k_{1}-\bar{k}%
_{3}+\bar{\xi}_{3},k_{2},\bar{k}_{3};\xi _{1},\xi _{2},\bar{\xi}%
_{3};t\right) +\right.  \nonumber \\
&&+\left. \widehat{F_{3,3}}\left( k_{1},k_{2},\bar{k}_{3};\xi _{1},\xi _{2}+%
\bar{k}_{3}-\bar{\xi}_{3},\bar{\xi}_{3};t\right) -\widehat{F_{3,3}}\left(
k_{1},k_{2}-\bar{k}_{3}+\bar{\xi}_{3},\bar{k}_{3};\xi _{1},\xi _{2},\bar{\xi}%
_{3};t\right) \right]  \nonumber \\
&=&2\cdot \left( 2\pi \right) ^{\frac{9}{2}}n\left( k_{1}\right) n\left(
k_{2}\right) \left( n\left( \xi _{2}\right) +n\left( \xi _{1}\right) \right)
\delta \left( k_{1}+k_{2}-\xi _{1}-\xi _{2}\right) -  \nonumber \\
&&-2\cdot \left( 2\pi \right) ^{\frac{9}{2}}\left[ \left( n\left(
k_{1}\right) +n\left( k_{2}\right) \right) n\left( \xi _{1}\right) n\left(
\xi _{2}\right) \delta \left( k_{1}+k_{2}-\xi _{2}-\xi _{1}\right) \right]
\label{S4E6}
\end{eqnarray}
\end{proof}

A new approximation of $\widehat{G_{2,2}}\left(
k_{1},k_{2};\xi _{1},\xi _{2};t\right) $ immediately follows from (\ref{S4E6}),  using Duhamel's formula and equation (\ref{S3E6}), \begin{eqnarray*}
&&i\partial _{t}\widehat{G_{2,2}}\left( k_{1},k_{2};\xi _{1},\xi
_{2};t\right)= \\
&&=\frac{1}{2}\left( -\sum_{j=1}^{2}\left\vert k_{j}\right\vert
^{2}+\sum_{j=1}^{2}\left\vert \xi _{j}\right\vert ^{2}\right) \widehat{%
G_{2,2}}\left( k_{1},k_{2},...,k_{L};\xi _{1},\xi _{2},...,\xi _{M};t\right)
+ \\
&&+2\varepsilon \cdot \left( 2\pi \right) ^{\frac{9}{2}}n\left(
k_{1},t\right) n\left( k_{2},t\right) \left( n\left( \xi _{2},t\right)
+n\left( \xi _{1},t\right) \right) \delta \left( k_{1}+k_{2}-\xi _{1}-\xi
_{2}\right) - \\
&&-2\varepsilon \cdot \left( 2\pi \right) ^{\frac{9}{2}}\left( n\left(
k_{1},t\right) +n\left( k_{2},t\right) \right) n\left( \xi _{1},t\right)
n\left( \xi _{2},t\right) \delta \left( k_{1}+k_{2}-\xi _{2}-\xi _{1}\right)
\end{eqnarray*}%
where we write again the dependences of $n\left( k,t\right) $ in $t.$ Using
now that $\widehat{G_{2,2}}\left( k_{1},k_{2};\xi _{1},\xi _{2};0\right) =0$
we obtain%
\begin{eqnarray}
&&\hskip -1cm \widehat{G_{2,2}}\left( k_{1},k_{2};\xi _{1},\xi _{2};t\right) =\frac{%
2\cdot \left( 2\pi \right) ^{\frac{9}{2}}\varepsilon }{i}\delta \left(
k_{1}+k_{2}-\xi _{1}-\xi _{2}\right)\times \nonumber\\
&&\hskip -1cm\times  \int_{0}^{t}\exp \left( \frac{i\left(
t-s\right) }{2}\left( \sum_{j=1}^{2}\left\vert k_{j}\right\vert
^{2}-\sum_{j=1}^{2}\left\vert \xi _{j}\right\vert ^{2}\right) \right) \times 
\nonumber \\
&&\hskip -1cm\times \left[ n\left( k_{1},s\right) n\left( k_{2},s\right) \left( n\left(
\xi _{2},s\right) +n\left( \xi _{1},s\right) \right) -\left( n\left(
k_{1},s\right) +n\left( k_{2},s\right) \right) n\left( \xi _{1},s\right)
n\left( \xi _{2},s\right) \right]ds.  \label{S4E7} 
\end{eqnarray}

\subsection{The non-Markovian equation} 
In this sub Section, using (\ref{S3E2}), (\ref{S4E7}),  the following non-Markovian approximation of the kinetic equation is deduced for the function $n$: 
\begin{eqnarray}
&&\partial _{t}n\left( k_{1},t\right) =4\varepsilon ^{2}\int_{\mathbb{R}%
^{3}}dk_{2}\int_{\mathbb{R}^{3}}d\xi _{1}\int_{\mathbb{R}^{3}}d\xi
_{2}\int_{0}^{t}ds\cos \left( \frac{\left( t-s\right) }{2}\left(
\sum_{j=1}^{2}\left\vert k_{j}\right\vert ^{2}-\sum_{j=1}^{2}\left\vert \xi
_{j}\right\vert ^{2}\right) \right) \times \nonumber  \\
&&\times  \Big[ \left( n\left( k_{1},s\right) +n\left( k_{2},s\right) \right)
n\left( \xi _{1},s\right) n\left( \xi _{2},s\right)-\nonumber \\ 
&&\hskip 3.2cm -n\left( k_{1},s\right)
n\left( k_{2},s\right) \left( n\left( \xi _{2},s\right) +n\left( \xi
_{1},s\right) \right) \Big] \delta \left( \xi _{1}+\xi
_{2}-k_{1}-k_{2}\right).    \label{S5E2}
\end{eqnarray}

To this end, we plug (\ref{S4E7}) into the evolution equation (\ref{S3E2}) for $n\left(k_{1},t\right)$ obtained using the invariance under
translations, that we recall here,%
\begin{eqnarray}
i\left( 2\pi \right) ^{\frac{3}{2}}\delta \left( k_{1}-\xi _{1}\right)
\partial _{t}n\left( k_{1},t\right) =\frac{\varepsilon }{\left( 2\pi \right)
^{3}}\int_{\mathbb{R}^{3}}d\bar{k}_{2}\int_{\mathbb{R}^{3}}d\bar{\xi}_{2}%
\Big[ \widehat{G_{2,2}}\left( k_{1},\bar{k}_{2};\xi _{1}+\bar{k}_{2}-\bar{%
\xi}_{2},\bar{\xi}_{2};t\right)-\nonumber \\
 -\widehat{G_{2,2}}\left( k_{1}-\bar{k}_{2}+%
\bar{\xi}_{2},\bar{k}_{2};\xi _{1},\bar{\xi}_{2};t\right) \Big]
\label{S4E8}
\end{eqnarray}

We remark that if the Dirac masses $\delta \left( k_{1}+k_{2}-\xi _{1}-\xi _{2}\right) $ in the first factor at the right hand side of (\ref{S4E7}) are evaluated in the corresponding arguments of  $\widehat{G_{2,2}}$ at the right hand side of (\ref{S4E8}), they both yield the same Dirac mass at the left hand side of (\ref{S4E8}):
\begin{eqnarray*}
\delta \left( k_{1}+\bar{k}_{2}-\left( \xi _{1}+\bar{k}_{2}-\bar{\xi}%
_{2}\right) -\bar{\xi}_{2}\right) &=&\delta \left( k_{1}-\xi _{1}\right) \\
\delta \left( k_{1}-\bar{k}_{2}+\bar{\xi}_{2}+\bar{k}_{2}-\xi _{1}-\bar{\xi}%
_{2}\right) &=&\delta \left( k_{1}-\xi _{1}\right).
\end{eqnarray*}

However, in order to simplify the computations, it seems convenient to rewrite (\ref{S4E8}), introducing a new integration variables $\eta$  as%
\begin{eqnarray}
&&\delta \left( k_{1}-\xi _{1}\right) \partial _{t}n\left( k_{1},t\right)=
\frac{\varepsilon }{i\left( 2\pi \right) ^{\frac{9}{2}}}\int_{\mathbb{R}%
^{3}}d\bar{k}_{2}\int_{\mathbb{R}^{3}}d\bar{\xi}_{2}\int_{\mathbb{R}%
^{3}}d\eta \cdot  \label{S4E9} \\
&&\cdot \left[ \widehat{G_{2,2}}\left( k_{1},\bar{k}_{2};\eta ,\bar{\xi}%
_{2};t\right) \delta \left( \eta +\bar{\xi}_{2}-\xi _{1}-\bar{k}_{2}\right) -%
\widehat{G_{2,2}}\left( \eta ,\bar{k}_{2};\xi _{1},\bar{\xi}_{2};t\right)
\delta \left( \eta +\bar{k}_{2}-\bar{\xi}_{2}-k_{1}\right) \right].  \nonumber
\end{eqnarray}

We rewrite (\ref{S4E7}) as%
\begin{equation}
\widehat{G_{2,2}}\left( k_{1},k_{2};\xi _{1},\xi _{2};t\right) =\frac{2\cdot
\left( 2\pi \right) ^{\frac{9}{2}}\varepsilon }{i}\delta \left(
k_{1}+k_{2}-\xi _{1}-\xi _{2}\right) \Delta \left( k_{1},k_{2};\xi _{1},\xi
_{2};t\right)  \label{S5E1}
\end{equation}%
where

\begin{eqnarray}
&&\Delta \left( k_{1},k_{2};\xi _{1},\xi _{2};t\right) =\int_{0}^{t}\exp
\left( \frac{i\left( t-s\right) }{2}\left( \sum_{j=1}^{2}\left\vert
k_{j}\right\vert ^{2}-\sum_{j=1}^{2}\left\vert \xi _{j}\right\vert
^{2}\right) \right) \cdot  \nonumber \\
&&\cdot \left[ n\left( k_{1},s\right) n\left( k_{2},s\right) \left( n\left(
\xi _{2},s\right) +n\left( \xi _{1},s\right) \right) -\left( n\left(
k_{1},s\right) +n\left( k_{2},s\right) \right) n\left( \xi _{1},s\right)
n\left( \xi _{2},s\right) \right] ds  \label{S5E1b}
\end{eqnarray}

Plugging (\ref{S5E1}) into (\ref{S4E9}) we obtain%
\begin{eqnarray*}
&&\delta \left( k_{1}-\xi _{1}\right) \partial _{t}n\left( k_{1},t\right) =-%
\frac{2\cdot \left( 2\pi \right) ^{\frac{9}{2}}\varepsilon ^{2}}{\left( 2\pi
\right) ^{\frac{9}{2}}}\int_{\mathbb{R}^{3}}d\bar{k}_{2}\int_{\mathbb{R}%
^{3}}d\bar{\xi}_{2}\int_{\mathbb{R}^{3}}d\eta \\
&& \Big[ \delta \left( k_{1}+%
\bar{k}_{2}-\eta -\bar{\xi}_{2}\right) \Delta \left( k_{1},\bar{k}_{2};\eta ,%
\bar{\xi}_{2};t\right) \delta \left( \eta +\bar{\xi}_{2}-\xi _{1}-\bar{k}%
_{2}\right) - \\
&&  -\delta \left( \eta +\bar{k}_{2}-\xi _{1}-\bar{\xi}_{2}\right)
\Delta \left( \eta ,\bar{k}_{2};\xi _{1},\bar{\xi}_{2};t\right) \delta
\left( \eta +\bar{k}_{2}-\bar{\xi}_{2}-k_{1}\right) \Big]
\end{eqnarray*}
and, combining the Dirac masses%
\begin{eqnarray*}
&&\delta \left( k_{1}-\xi _{1}\right) \partial _{t}n\left( k_{1},t\right) =-%
\frac{2\cdot \left( 2\pi \right) ^{\frac{9}{2}}\varepsilon ^{2}}{\left( 2\pi
\right) ^{\frac{9}{2}}}\delta \left( k_{1}-\xi _{1}\right) \int_{\mathbb{R}%
^{3}}d\bar{k}_{2}\int_{\mathbb{R}^{3}}d\bar{\xi}_{2}\int_{\mathbb{R}%
^{3}}d\eta\\
 &&\left[ \Delta \left( k_{1},\bar{k}_{2};\eta ,\bar{\xi}%
_{2};t\right) \delta \left( \eta +\bar{\xi}_{2}-\xi _{1}-\bar{k}_{2}\right)
-\Delta \left( \eta ,\bar{k}_{2};\xi _{1},\bar{\xi}_{2};t\right) \delta
\left( \eta +\bar{k}_{2}-\bar{\xi}_{2}-k_{1}\right) \right]. 
\end{eqnarray*}

Then%
\begin{eqnarray*}
\partial _{t}n\left( k_{1},t\right) =2\varepsilon ^{2}\int_{\mathbb{R}^{3}}d%
\bar{k}_{2}\int_{\mathbb{R}^{3}}d\bar{\xi}_{2}\int_{\mathbb{R}^{3}}d\eta %
\Big[ \Delta \left( \eta ,\bar{k}_{2};k_{1},\bar{\xi}_{2};t\right) \delta
\left( \eta +\bar{k}_{2}-\bar{\xi}_{2}-k_{1}\right)-\\
 -\Delta \left( k_{1},%
\bar{k}_{2};\eta ,\bar{\xi}_{2};t\right) \delta \left( \eta +\bar{\xi}%
_{2}-k_{1}-\bar{k}_{2}\right) \Big] 
\end{eqnarray*}

We can replace the second Dirac by the first exchanging the variables $\bar{%
\xi}_{2}\longleftrightarrow \bar{k}_{2}.$ Then%
\begin{eqnarray*}
&&\partial _{t}n\left( k_{1},t\right) =2\varepsilon ^{2}\int_{\mathbb{R}^{3}}d%
\bar{k}_{2}\int_{\mathbb{R}^{3}}d\bar{\xi}_{2}\int_{\mathbb{R}^{3}}d\eta %
\Big[ \Delta \left( \eta ,\bar{k}_{2};k_{1},\bar{\xi}_{2};t\right)-\\
&&\hskip 4cm  -\Delta
\left( k_{1},\bar{\xi}_{2};\eta ,\bar{k}_{2};t\right) \Big] \delta \left(
\eta +\bar{k}_{2}-\bar{\xi}_{2}-k_{1}\right) 
\end{eqnarray*}
and relabelling the variables as follows%
\[
\bar{\xi}_{2}\rightarrow k_{2}\ \ ,\ \ \eta \rightarrow \xi _{1}\ \ ,\ \ 
\bar{k}_{2}\rightarrow \xi _{2} 
\]
it follows%
\begin{eqnarray*}
&&\partial _{t}n\left( k_{1},t\right) =2\varepsilon ^{2}\int_{\mathbb{R}%
^{3}}dk_{2}\int_{\mathbb{R}^{3}}d\xi _{1}\int_{\mathbb{R}^{3}}d\xi _{2}\Big[
\Delta \left( \xi _{1},\xi _{2};k_{1},k_{2};t\right)-\\ 
&&\hskip 4cm -\Delta \left(
k_{1},k_{2};\xi _{1},\xi _{2};t\right) \Big] \delta \left( \xi _{1}+\xi
_{2}-k_{1}-k_{2}\right). 
\end{eqnarray*}

We now compute the difference between brackets.%
\begin{eqnarray*}
&&\left[ \Delta \left( \xi _{1},\xi _{2};k_{1},k_{2};t\right) -\Delta \left(
k_{1},k_{2};\xi _{1},\xi _{2};t\right) \right] =\\
&&=\int_{0}^{t}ds\exp \left( \frac{i\left( t-s\right) }{2}\left(
-\sum_{j=1}^{2}\left\vert k_{j}\right\vert ^{2}+\sum_{j=1}^{2}\left\vert \xi
_{j}\right\vert ^{2}\right) \right) \cdot \\
&&\cdot \left[ \left( n\left( k_{1},s\right) +n\left( k_{2},s\right) \right)
n\left( \xi _{1},s\right) n\left( \xi _{2},s\right) -n\left( k_{1},s\right)
n\left( k_{2},s\right) \left( n\left( \xi _{2},s\right) +n\left( \xi
_{1},s\right) \right) \right] - \\
&&-\int_{0}^{t}ds\exp \left( \frac{i\left( t-s\right) }{2}\left(
\sum_{j=1}^{2}\left\vert k_{j}\right\vert ^{2}-\sum_{j=1}^{2}\left\vert \xi
_{j}\right\vert ^{2}\right) \right) \cdot \\
&&\cdot \left[ n\left( k_{1},s\right) n\left( k_{2},s\right) \left( n\left(
\xi _{2},s\right) +n\left( \xi _{1},s\right) \right) -\left( n\left(
k_{1},s\right) +n\left( k_{2},s\right) \right) n\left( \xi _{1},s\right)
n\left( \xi _{2},s\right) \right]
\end{eqnarray*}

\begin{eqnarray}
&&\left[ \Delta \left( \xi _{1},\xi _{2};k_{1},k_{2};t\right) -\Delta \left(
k_{1},k_{2};\xi _{1},\xi _{2};t\right) \right] = \nonumber \\
&&=\int_{0}^{t}ds\Bigg[ \exp \left( \frac{i\left( t-s\right) }{2}\left(
-\sum_{j=1}^{2}\left\vert k_{j}\right\vert ^{2}+\sum_{j=1}^{2}\left\vert \xi
_{j}\right\vert ^{2}\right) \right)+\\
&&\hskip 5.4cm  +\exp \left( \frac{i\left( t-s\right) }{2%
}\left( \sum_{j=1}^{2}\left\vert k_{j}\right\vert
^{2}-\sum_{j=1}^{2}\left\vert \xi _{j}\right\vert ^{2}\right) \right) \Bigg]
\times   \nonumber \\
&&\times  \left[ \left( n\left( k_{1},s\right) +n\left( k_{2},s\right) \right)
n\left( \xi _{1},s\right) n\left( \xi _{2},s\right) -n\left( k_{1},s\right)
n\left( k_{2},s\right) \left( n\left( \xi _{2},s\right) +n\left( \xi
_{1},s\right) \right) \right]  \nonumber \\
&&=2\int_{0}^{t}ds\cos \left( \frac{\left( t-s\right) }{2}\left(
\sum_{j=1}^{2}\left\vert k_{j}\right\vert ^{2}-\sum_{j=1}^{2}\left\vert \xi
_{j}\right\vert ^{2}\right) \right) \cdot  \nonumber \\
&&\cdot \left[ \left( n\left( k_{1},s\right) +n\left( k_{2},s\right) \right)
n\left( \xi _{1},s\right) n\left( \xi _{2},s\right) -n\left( k_{1},s\right)
n\left( k_{2},s\right) \left( n\left( \xi _{2},s\right) +n\left( \xi
_{1},s\right) \right) \right]  \label{S5E1a}
\end{eqnarray}

We then obtain the equation%
\begin{eqnarray*}
&&\partial _{t}n\left( k_{1},t\right) =4\varepsilon ^{2}\int_{\mathbb{R}%
^{3}}dk_{2}\int_{\mathbb{R}^{3}}d\xi _{1}\int_{\mathbb{R}^{3}}d\xi
_{2}\int_{0}^{t}ds\cos \left( \frac{\left( t-s\right) }{2}\left(
\sum_{j=1}^{2}\left\vert k_{j}\right\vert ^{2}-\sum_{j=1}^{2}\left\vert \xi
_{j}\right\vert ^{2}\right) \right) \times \nonumber  \\
&&\times  \Big[ \left( n\left( k_{1},s\right) +n\left( k_{2},s\right) \right)
n\left( \xi _{1},s\right) n\left( \xi _{2},s\right)-\nonumber \\ 
&&\hskip 3.7cm -n\left( k_{1},s\right)
n\left( k_{2},s\right) \left( n\left( \xi _{2},s\right) +n\left( \xi
_{1},s\right) \right) \Big] \delta \left( \xi _{1}+\xi
_{2}-k_{1}-k_{2}\right). 
\end{eqnarray*}

This is the non-Markovian approximation of the kinetic equation, as
expected. The natural time scale is $t\approx \frac{1}{\varepsilon ^{2}}.$

Assuming that $n\left( k,t\right) $ changes in that time scale $t\approx \frac{1}{\varepsilon ^{2}}$, we can
approximate the oscillatory integral in time by a Dirac in the energy. It is
convenient to change the time scale%
\begin{equation}
\tau =t\varepsilon ^{2}\ \ ,\ \ \sigma =s\varepsilon ^{2}  \label{S5E2a}
\end{equation}

Then, assuming that $n\left( k,t\right) $ changes in the time scale $\tau
\approx 1,$ we can replace (\ref{S5E2}) by%
\begin{eqnarray}
&&\hskip -0.7cm \partial _{\tau }n\left( k_{1},\tau \right) =4\int_{\mathbb{R}%
^{3}}dk_{2}\int_{\mathbb{R}^{3}}d\xi _{1}\int_{\mathbb{R}^{3}}d\xi
_{2}\int_{0}^{\tau }\frac{d\sigma }{\varepsilon ^{2}}\cos \left( \frac{%
\left( \tau -\sigma \right) }{2\varepsilon ^{2}}\left(
\sum_{j=1}^{2}\left\vert k_{j}\right\vert ^{2}-\sum_{j=1}^{2}\left\vert \xi
_{j}\right\vert ^{2}\right) \right) \times   \label{S5E3} \\
&&\times  \Big[ \left( n\left( k_{1},\sigma \right) +n\left( k_{2},\sigma
\right) \right) n\left( \xi _{1},\sigma \right) n\left( \xi _{2},\sigma
\right)-\nonumber\\
&& \hskip 3cm -n\left( k_{1},\sigma \right) n\left( k_{2},\sigma \right) \left(
n\left( \xi _{2},\sigma \right) +n\left( \xi _{1},\sigma \right) \right) %
\Big] \delta \left( \xi _{1}+\xi _{2}-k_{1}-k_{2}\right).\nonumber
\end{eqnarray}
We just remark that%
\[
\int_{0}^{\tau }\frac{d\sigma }{\varepsilon ^{2}}\cos \left( \frac{\left(
\tau -\sigma \right) \Omega }{2\varepsilon ^{2}}\right) =\int_{0}^{\tau }%
\frac{d\sigma }{\varepsilon ^{2}}\cos \left( \frac{\sigma \Omega }{%
2\varepsilon ^{2}}\right) 
\]

\[
\int d\Omega \psi \left( \Omega \right) \int_{0}^{\tau }\frac{d\sigma }{%
\varepsilon ^{2}}\cos \left( \frac{\sigma \Omega }{2\varepsilon ^{2}}\right)
=\int d\Omega \psi \left( \Omega \right) \frac{2}{\Omega }\sin \left( \frac{%
\tau \Omega }{2\varepsilon ^{2}}\right) =2\int d\theta \psi \left( \frac{%
2\varepsilon ^{2}\theta }{\tau }\right) \frac{\sin \theta }{\theta } 
\]%
which formally converges to%
\[
2\psi \left( 0\right) \int_{-\infty }^{\infty }\frac{\sin \theta }{\theta }%
d\theta =2\pi \psi \left( 0\right). 
\]

Then (\ref{S5E3}) converges, in the limit $\varepsilon \rightarrow 0$ to%
\begin{eqnarray}
&&\hskip -0.9cm \partial _{\tau }n\left( k_{1},\tau \right)  =8\pi \int_{\mathbb{R}%
^{3}}dk_{2}\int_{\mathbb{R}^{3}}d\xi _{1}\int_{\mathbb{R}^{3}}d\xi
_{2}\delta \left( \sum_{j=1}^{2}\left\vert k_{j}\right\vert
^{2}-\sum_{j=1}^{2}\left\vert \xi _{j}\right\vert ^{2}\right) \delta
\left( \xi _{1}+\xi _{2}-k_{1}-k_{2}\right) \times   \nonumber\\
&&\times  \left[ \left( n\left( k_{1},\tau \right) +n\left( k_{2},\tau \right)
\right) n\left( \xi _{1},\tau \right) n\left( \xi _{2},\tau \right) -n\left(
k_{1},\tau \right) n\left( k_{2},\tau \right) \left( n\left( \xi _{2},\tau
\right) +n\left( \xi _{1},\tau \right) \right) \right]    \label{S5E4} 
\end{eqnarray}%
where $k_{1}\in \mathbb{R}^{3},\ \tau >0.$ This gives the WT kinetic
equation.

\bigskip

\section{Self-similar blow-up profiles for the kinetic equation.}
\setcounter{equation}{0}
\setcounter{theorem}{0}
\bigskip

\subsection{Isotropic solutions of the WT kinetic equation. Reformulation of
the equation in terms of the energy variable.}

\bigskip

In this Subsection we reformulate the equation (\ref{S5E4}) for the
solutions $n\left( k,\tau \right) $ that depend only in the variable $%
\left\vert k\right\vert .$ In addition, it will be convenient to rewrite
these solutions in terms of the energy variable $\epsilon =\frac{\left\vert
k\right\vert ^{2}}{2},$ which has a more relevant physical meaning than $%
\left\vert k\right\vert $. It turns out that the particle density in the
space of energy has the form%
\begin{equation}
f\left( \epsilon ,\tau \right) =n\left( k,\tau \right) \ \ ,\ \ \epsilon =%
\frac{\left\vert k\right\vert ^{2}}{2}  \label{S6E1}
\end{equation}

Notice that the function $f$ is just the original density $n$ written in
terms of the energy variable.

It turns out that, assuming isotropy of the initial data, i.e. $n_{0}\left(
k\right) =n_{0}\left( \left\vert k\right\vert \right) ,$ the corresponding
solution $n\left( k,t\right) $ of (\ref{S5E4}) is also isotropic for $t>0$.
Moreover, the function $f\left( \epsilon ,\tau \right) $ defined in (\ref%
{S6E1}) satisfies%
\begin{equation}
\partial _{\tau }f\left( \epsilon _{1},\tau \right) =\mathbb{C}\left[ f%
\right] \left( \epsilon _{1},\tau \right) \ \ ,\ \ \ \ \epsilon _{1}\geq 0\ 
\label{T2E6}
\end{equation}%
where%
\begin{equation}
\mathbb{C}\left[ f\right] \left( \epsilon _{1},\tau \right) =\Gamma \int
\int d\epsilon _{3}d\epsilon _{4}W\left( \epsilon _{1},\epsilon
_{2},\epsilon _{3},\epsilon _{4}\right) \left[ \left( f_{1}+f_{2}\right)
f_{3}f_{4}-\left( f_{3}+f_{4}\right) f_{1}f_{2}\right]  \label{T2E7}
\end{equation}%
\[
\Gamma =\frac {1} {8 \pi ^6}
\]
\begin{equation}
W\left( \epsilon _{1},\epsilon _{2},\epsilon _{3},\epsilon _{4}\right) =\min
\left\{ 1,\sqrt{\frac{\epsilon _{2}}{\epsilon _{1}}},\sqrt{\frac{\epsilon
_{3}}{\epsilon _{1}}},\sqrt{\frac{\epsilon _{4}}{\epsilon _{1}}}\right\} \
,\ \epsilon _{2}=\epsilon _{3}+\epsilon _{4}-\epsilon _{1}\   \label{T2E8}
\end{equation}%
where from now on we use the notation%
\[
f_{j}=f\left( \epsilon _{j},\tau \right) ,\ \ j=1,2,3,4\ . 
\]

The problem (\ref{T2E6}) must be solved with the initial condition%
\begin{equation}
f\left( \epsilon ,0\right) =f_{0}\left( \epsilon \right) \ \ ,\ \ \epsilon >0
\label{T2E9}
\end{equation}

The region of integration in (\ref{T2E6}) is the set $\left\{ \left(
\epsilon _{3},\epsilon _{4}\right) \in \mathbb{R}_{+}^{2}:\epsilon
_{3}+\epsilon _{4}\geq \epsilon _{1}\right\} .$ Notice that this set can be
also defined by means of the inequalities $\epsilon _{2}\geq 0,$ $\epsilon
_{3}\geq 0,\ \epsilon _{4}\geq 0.$ The equation (\ref{T2E7}) for the
isotropic solutions of the WT equation, as well as the analogous equation
for Nordheim equation has been derived by multiple authors (cf. for instance 
\cite{SK1}, \cite{SK2}, \cite{JPR}, \cite{LLPR}).

The properties of the solutions of the initial value problem (\ref{T2E6}), (%
\ref{T2E9}) have been extensively studied using analytical and numerical
methods (cf. \cite{SK1}, \cite{SK2}, \cite{JPR}, \cite{LLPR}, \cite{EV1}, 
\cite{EV3}, \cite{SGMN}). In particular, it has been shown in \cite{EV1}, 
\cite{EV3} that the solutions of (\ref{T2E6}), (\ref{T2E9}) with bounded
initial data $f_{0}\left( \epsilon \right) $ blow-up in finite time for a
large class of functions $f_{0}.$ There are not available rigorous results
concerning the form of the solutions near the blow-up point. However, the
numerical simulations in \cite{SK1}, \cite{SK2}, \cite{JPR}, \cite{LLPR}
suggest that one stable mechanism of blow-up for the solutions of (\ref{T2E6}%
), (\ref{T2E9}) can be described by a self-similar solution that we describe
in the following Subsection.

\bigskip

\subsection{Self-similar blow-up solutions for the isotropic kinetic
equation.}

\bigskip

In order to determine the form of the self-similar solutions of (\ref{T2E6}%
), (\ref{T2E7}) exhibiting blow-up it is convenient to use some dimensional
analysis arguments. From the dimensional point of view, (\ref{T2E6}), (\ref%
{T2E7}) reads as%
\[
\frac{\left[ f\right] }{\left[ \tau \right] }=\left[ f\right] ^{3}\left[
\epsilon \right] ^{2} 
\]

Then%
\begin{equation}
\left[ f\right] ^{2}\left[ \epsilon \right] ^{2}\left[ \tau \right] =1
\label{A1}
\end{equation}

This formula gives a relation between the scaling parameters associated to $%
f,\ \epsilon $ and $\tau .$ Indeed, if we look for solutions of (\ref{T2E6}%
), (\ref{T2E7}) in the form%
\begin{equation}
f\left( \epsilon ,\tau \right) =\sqrt{2\beta }\left( -\tau \right) ^{-\alpha
}\phi \left( \frac{\epsilon }{\left( -\tau \right) ^{2\beta }}\right) \ \ ,\
\ \omega =\frac{\epsilon }{\left( -\tau \right) ^{2\beta }}  \label{A2}
\end{equation}%
the identity (\ref{A1}) implies

\begin{equation}
\alpha -2\beta =\frac{1}{2}  \label{A4}
\end{equation}
We follow in these computations the notation of \cite{JPR}, \cite{LLPR}. \\
 Notice that in (\ref{A2}) we are assuming that the solution $f$ is defined
for $\tau <0$ and it blows up for $\tau =0.$ This is always possible due to
the invariance of (\ref{T2E6}), (\ref{T2E7}) under translations in time. The
multiplicative (non-dimensional) factor $\sqrt{2\beta }$ on the right-hand
side of the first formula of (\ref{A2}) has been introduced in order to
obtain simpler formulas for the equation satisfied by $\phi $ (cf. \cite{JPR}%
, \cite{LLPR}). The choice of the exponent $2\beta $ in order to scale the
energy is convenient in order to have the scaling with the exponent $\beta $
for $\left\vert k\right\vert .$

The function $\phi $ in (\ref{A2}) satisfies the integro-differential
equation%
\begin{equation}
\alpha \phi \left( \omega \right) +2\beta \omega \phi _{\omega }\left(
\omega \right) =2\beta \mathbb{C}\left[ \phi \right] \left( \omega \right)
\label{S6E2}
\end{equation}%
and dividing this equation by $2\beta $ we finally arrive at the non-linear
eigenvalue problem%
\begin{equation}
\nu \phi \left( \omega \right) +\omega \phi _{\omega }\left( \omega \right) =%
\mathbb{C}\left[ \phi \right] \left( \omega \right)  \label{S6E3}
\end{equation}%
where, using (\ref{A4}) we obtain 
\begin{equation}
\nu =\frac{\alpha }{2\beta }=\frac{2\beta +\frac{1}{2}}{2\beta }>1
\label{A3}
\end{equation}

The parameter $\nu $ is in (\ref{S6E3}) a free parameter which is part of
the solution of the problem to be solved. This type of problems in which the
scaling properties of the solutions cannot be determined by purely
dimensional arguments and require the solution of an eigenvalue problem in
order to determine the scaling of the solutions are usually termed as 
\textit{Self-similarity of the second kind} (cf. \cite{Bar}).\ 

Currently, there is not any rigorous mathematical result concerning the
existence of non-trivial solutions of (\ref{T2E7}), (\ref{S6E2}). A
numerical method that allows to obtain solutions of (\ref{T2E7}), (\ref{S6E2}%
) has been developed in \cite{SGMN}. An alternative approach that it was the
original one used to obtain self-similar behaviours for the solutions of (%
\ref{T2E6}), (\ref{T2E7}) was the one contained in the papers \cite{SK1}, 
\cite{SK2}, \cite{JPR}, \cite{LLPR}, \cite{SGMN}. In these papers numerical
simulations of the time dependent problem (\ref{T2E6}), (\ref{T2E7}) were
performed. The self-similar profiles and the value of $\nu $ obtained in 
\cite{SK1}, \cite{SK2}, \cite{JPR}, \cite{LLPR}, \cite{SGMN} are similar.
The numerical values of $\nu, \alpha , \beta  $ obtained in the papers are the following ones%
\[
\begin{tabular}{lllll}
Paper & $\nu $ & $\alpha $ & $\beta $ &  \\ 
D. V. Semikoz \&al. \cite{SK1}& 1.24 &  2.6& n.a. &  \\ 
R. Lacaze \& al. \cite{LLPR} & 1.234 & 2.639 & 2.139 &  \\ 
B.V. Semisalov \& al. \cite{SGMN} & 1.22 &  n.a. & n.a. & 
\end{tabular}%
\]
\bigskip
(n.a. for  non available).

Notice that the values of $\nu $ obtained in the three papers agree.
Moreover, the fact that these self-similar solutions arise in direct
numerical simulations of the time dependent equations, suggest that these
self-similar solutions are stable. The numerical computations in \cite{SK1}, \cite{LLPR}, \cite{SGMN}  indicate that the solutions of (\ref{T2E7}), (\ref{S6E2}) under
consideration are globally bounded, and they behave for large values of $%
\omega $ as the power law%
\[
\phi \left( \omega \right) \sim \frac{A}{\omega ^{\nu }}\ \ \  as\ \ \
\omega \rightarrow \infty 
\]%
where $A>0.$ This asymptotic behaviour for the solutions of (\ref{T2E7}), (%
\ref{S6E2}) has been justified in \cite{JPR}, \cite{LLPR}, \cite{SGMN}.

\bigskip

We will use in the next Section the self-similar solutions rewritten in the
original variables $n,\ k.$ Using (\ref{S6E1}) we obtain that, assuming that
a solution $\phi $ of (\ref{T2E7}), (\ref{S6E2}) with the properties
indicated above exists, the following function would be a self-similar
solution of (\ref{S5E4}) defined for $\tau <0$ and it blows up at time $\tau
=0.$

\begin{equation}
n\left( k,\tau \right) =\left( -\tau \right) ^{-2\beta -\frac{1}{2}}\Phi
\left( z\right) \ \ ,\ \ z=\frac{k}{\left( -\tau \right) ^{\beta }}\ \ ,\ \
\beta =1.068...  \label{S5E5}
\end{equation}

We value of the parameter $\beta $ has been computed using (\ref{A3}) with
the value of $\nu =1.234...$ obtained in \cite{JPR}, \cite{LLPR}. The
function $\Phi $ in (\ref{S5E5}) is given by%
\begin{equation}
\Phi \left( z\right) =\sqrt{2\beta }\phi \left( \frac{\left( z\right) ^{2}}{2%
}\right)   \label{A5}
\end{equation}%
where $\phi $ is the solution of (\ref{T2E7}), (\ref{S6E2}).

Notice that the function $\Phi $ is expected to have a well defined positive
limit as $\left\vert z\right\vert \rightarrow 0,$ since the function $\phi
\left( \omega \right) $ is seem to converge to a positive constant in the
numerical simulations of \cite{SK1}, \cite{SK2}, \cite{JPR}, \cite{LLPR}, 
\cite{SGMN}.

\section{Loss of Markovianity and onset of correlations near the blow-up
time.}
\setcounter{equation}{0}
\setcounter{theorem}{0}
\bigskip
\subsection{Loss of Markovianity}
We now examine the effect of the blow-up in the formal derivation of the
kinetic equation (\ref{S5E4}). Specifically, we will see that a breakdown in
the arguments yielding the derivation of the kinetic equation (\ref{S5E4})
obtained in Section \ref{Derivation} when the time approaches to the time in which the
solutions of (\ref{S5E4}) develop a singularity.

Notice that the main approximation made in the derivation of (\ref{S5E4})
from (\ref{S5E3}) is to assume that the function $n\left( k,\tau \right) $
has changes of order one if the time $\tau $ changes by amounts of order $%
\left( \tau _{2}-\tau _{1}\right) \left\vert k\right\vert ^{2}\lesssim
\varepsilon ^{2}.$ Given a solution behaving as in (\ref{S5E5}), we obtain
that this approximation fails if $\left( -\tau \right) \left\vert
k\right\vert ^{2}\approx \varepsilon ^{2}.$ Given that $\Phi $ is
concentrated in regions where $\left\vert z\right\vert $ is of order one,
this means that $\left\vert k\right\vert \approx \left( -\tau \right)
^{\beta }.$ Then, the Markovian approximation can be expected to fail for
times in which $\left( -\tau \right) ^{1+2\beta }\approx \varepsilon ^{2}$,
or equivalently%
\begin{equation}
\left( -\tau \right) \approx \varepsilon ^{\frac{2}{1+2\beta }}  \label{S5E6}
\end{equation}

We claim that, when $\left( -\tau \right) $ becomes of order $\varepsilon ^{%
\frac{2}{1+2\beta }},$ the approximation $\left\vert \widehat{G_{2,2}}%
\right\vert \ll n^{2}$  that has been used in all the previous
approximations (\textcolor{blue}{cf. (\ref{S2E8}}) and the comment below), as well as the approximation of $\widehat{F_{3,3}}$ by
products of terms with the form $\widehat{F_{1,1}}$ breaks down. Indeed, we
estimate first the order of magnitude of $\widehat{G_{2,2}}$ if (\ref{S5E6})
holds. Using (\ref{S5E1a}) as well as the fact that the microscopic time
scale is given by $t=\frac{\tau }{\varepsilon ^{2}}$ we obtain the
approximation
\begin{eqnarray*}
&&\Delta \left( k_{1},k_{2};\xi _{1},\xi _{2};\tau \right) =\frac{1}{%
\varepsilon ^{2}}\int_{0}^{\tau }\exp \left( \frac{i\left( \tau -\overline{\tau }
\right) }{2\varepsilon ^{2}}\left( \sum_{j=1}^{2}\left\vert k_{j}\right\vert
^{2}-\sum_{j=1}^{2}\left\vert \xi _{j}\right\vert ^{2}\right) \right)\times 
\\
&&\times  \left[ n\left( k_{1},\overline{\tau } \right) n\left( k_{2},\overline{\tau } \right)
\left( n\left( \xi _{2},\overline{\tau } \right) +n\left( \xi _{1},\overline{\tau } \right)
\right) -\left( n\left( k_{1},\overline{\tau } \right) +n\left( k_{2},\overline{\tau } \right)
\right) n\left( \xi _{1},\overline{\tau } \right) n\left( \xi _{2},\overline{\tau } \right) %
\right] d\overline{\tau }
\end{eqnarray*}

Notice that the quantity inside the exponential factor becomes of order one
if $\left\vert k\right\vert $ is of order $\left( -\tau \right) ^{\beta }$
and (\ref{S5E6}) holds. On the other hand, the factors $n$ yield, using (\ref%
{S5E5}), a term of order $\left( -\tau \right) ^{-6\beta -\frac{3}{2}}.$ The
integration yields another factor of order $\left( -\tau \right) .$ Then,
the order of magnitude of $\left\vert \Delta \right\vert $ is $\frac{1}{%
\varepsilon ^{2}}\left( -\tau \right) ^{-6\beta -\frac{1}{2}}$ and then, the
order of magnitude of $\widehat{G_{2,2}}$ is $\frac{1}{\varepsilon }\left(
-\tau \right) ^{-9\beta -\frac{1}{2}}.$ (Notice that the Dirac mass $\delta
\left( k_{1}+k_{2}-\xi _{1}-\xi _{2}\right) $ scales like $\varepsilon
^{-3\beta }$ since $\left\vert k\right\vert $ scales like $\left( -\tau
\right) ^{\beta }$). We need to compare this order of magnitude of $\widehat{%
G_{2,2}}$ with the order of magnitude of $\left( \widehat{F_{1,1}}\right)
^{2}.$ We recall (cf. (\ref{S3E1})) that $\widehat{F_{1,1}}\left( k_{1};\xi
_{1};t\right) =\left( 2\pi \right) ^{\frac{3}{2}}\delta \left( k_{1}-\xi
_{1}\right) n\left( k_{1},t\right) .$ Then, for the range of values of $k$
and $\tau $ under consideration we have that the order of magnitude of $%
\widehat{F_{1,1}}$ is $\left( -\tau \right) ^{-3\beta }\left( -\tau \right)
^{-2\beta -\frac{1}{2}}=\left( -\tau \right) ^{-5\beta -\frac{1}{2}}.$ Then $%
\left( \widehat{F_{1,1}}\right) ^{2}$ is of order $\left( -\tau \right)
^{-10\beta -1}.$ Then $\widehat{G_{2,2}}$ and $\left( \widehat{F_{1,1}}%
\right) ^{2}$ become of the same order of magnitude if%
\[
\frac{1}{\varepsilon }\left( -\tau \right) ^{-9\beta -\frac{1}{2}}\approx
\left( -\tau \right) ^{-10\beta -1}
\]%
or, equivalently,  $\left( -\tau \right) ^{\beta +\frac{1}{2}}\approx \varepsilon$, i.e.  $\left( -\tau \right) \approx \varepsilon ^{\frac{2}{1+2\beta }}$, that yields (\ref{S5E6}).

Summarizing, the smallness of the correlations (or cumulants) which has been
used to derive the kinetic equation, fails for the range of times defined by
(\ref{S5E6}). Therefore, all the terms must be kept in the variables $%
\widehat{F_{2,2}},\ \widehat{F_{3,3}}$ and the factorization approximation
is not possible anymore for the range of times given by (\ref{S5E6}). 

\subsection{Onset of correlations near the blow-up
time.}

We rewrite (\ref{S2E3}) using the variable $t=\frac{\tau }{\varepsilon ^{2}}.$
Then%
\begin{eqnarray}
&&i\partial _{\tau }\widehat{F_{L,M}}\left( k_{1},k_{2},...,k_{L};\xi _{1},\xi
_{2},...,\xi _{M};\tau \right)=\nonumber  \\
&&=\frac{1}{2\varepsilon ^{2}}\left( -\sum_{j=1}^{L}\left\vert
k_{j}\right\vert ^{2}+\sum_{j=1}^{M}\left\vert \xi _{j}\right\vert
^{2}\right) \widehat{F_{L,M}}\left( k_{1},k_{2},...,k_{L};\xi _{1},\xi
_{2},...,\xi _{M};\tau \right) -  \nonumber \\
&&-\frac{1}{\left( 2\pi \right) ^{3}\varepsilon }\sum_{j=1}^{L}\int_{\mathbb{%
R}^{3}}d\bar{k}_{L+1}\int_{\mathbb{R}^{3}}d\bar{\xi}_{M+1}\nonumber\\
&&\widehat{%
F_{L+1,M+1}}\left( k_{1},k_{2},...k_{j-1},k_{j}-\bar{k}_{L+1}+\bar{\xi}%
_{M+1},k_{j+1},...,k_{L},\bar{k}_{L+1};\xi _{1},\xi _{2},...,\xi _{M},\bar{%
\xi}_{M+1};\tau \right) +  \nonumber \\
&&+\frac{1}{\left( 2\pi \right) ^{3}\varepsilon }\sum_{j=1}^{M}\int_{\mathbb{%
R}^{3}}d\bar{k}_{L+1}\int_{\mathbb{R}^{3}}d\bar{\xi}_{M+1}\widehat{%
F_{L+1,M+1}}\Big( k_{1},k_{2},...,k_{L},\bar{k}_{L+1};\xi _{1},\xi
_{2},...,\nonumber \\
&&,...,\xi _{j-1},\xi _{j}+\bar{k}_{L+1}-\bar{\xi}_{M+1},\xi
_{j+1},...,\xi _{M},\bar{\xi}_{M+1};\tau \Big), \hskip 0.3cm \tau \in  \mathbb{R},\, k_i\in  \mathbb{R}^3,\, \xi _i\in \mathbb{R}^3. \nonumber
\end{eqnarray}

We will assume that the contributions due to the terms with $L\neq M$ are
negligible. (Formally, if these terms are initially zero, they remain equal
to zero for later times). We then consider only the hierarchy of equations
with $L=M,$ i.e.%
\begin{eqnarray}
&&i\partial _{\tau }\widehat{F_{L,L}}\left( k_{1},k_{2},...,k_{L};\xi
_{1},\xi _{2},...,\xi _{L};\tau \right)=\nonumber  \\
&&=\frac{1}{2\varepsilon ^{2}}\left( -\sum_{j=1}^{L}\left\vert
k_{j}\right\vert ^{2}+\sum_{j=1}^{L}\left\vert \xi _{j}\right\vert
^{2}\right) \widehat{F_{L,L}}\left( k_{1},k_{2},...,k_{L};\xi _{1},\xi
_{2},...,\xi _{L};\tau \right) -  \nonumber \\
&&-\frac{1}{\left( 2\pi \right) ^{3}\varepsilon }\sum_{j=1}^{L}\int_{\mathbb{%
R}^{3}}d\bar{k}_{L+1}\int_{\mathbb{R}^{3}}d\bar{\xi}_{L+1}\nonumber \\
&&\widehat{%
F_{L+1,L+1}}\left( k_{1},k_{2},...k_{j-1},k_{j}-\bar{k}_{L+1}+\bar{\xi}%
_{L+1},k_{j+1},...,k_{L},\bar{k}_{L+1};\xi _{1},\xi _{2},...,\xi _{L},\bar{%
\xi}_{L+1};\tau \right) +  \nonumber \\
&&+\frac{1}{\left( 2\pi \right) ^{3}\varepsilon }\sum_{j=1}^{L}\int_{\mathbb{%
R}^{3}}d\bar{k}_{L+1}\int_{\mathbb{R}^{3}}d\bar{\xi}_{L+1}\nonumber \\
&&\widehat{%
F_{L+1,L+1}}\left( k_{1},k_{2},...,k_{L},\bar{k}_{L+1};\xi _{1},\xi
_{2},...,\xi _{j-1},\xi _{j}+\bar{k}_{L+1}-\bar{\xi}_{L+1},\xi
_{j+1},...,\xi _{L},\bar{\xi}_{L+1};\tau \right). 
 \label{S5E7} 
\end{eqnarray}

\bigskip We now examine the scaling of the different terms. Notice that (\ref%
{S3E1a}) suggests the following scaling for $\widehat{F_{L,L}}$%
\[
\left( \frac{1}{\left\vert k\right\vert ^{3}}\right) ^{L}n^{L}\approx \frac{1%
}{\left( -\tau \right) ^{3\beta L}}\frac{1}{\left( -\tau \right) ^{\left(
2\beta +\frac{1}{2}\right) L}}=\frac{1}{\left( -\tau \right) ^{5\beta L+%
\frac{L}{2}}}
\]

We can then see that the three terms in (\ref{S5E7}) become of the same
order of magnitude if (\ref{S5E6}) holds. Indeed, the first term (left-hand
side) is of order%
\[
\frac{1}{\left(-\tau \right) ^{5\beta L+\frac{L}{2}+1}} 
\]

The first term on the right-hand side is of order%
\[
\frac{1}{\varepsilon ^{2}}\left( -\tau \right) ^{2\beta }\frac{1}{\left(
-\tau \right) ^{5\beta L+\frac{L}{2}}}
\]%
and the last term (second on the right-hand side) is of order%
\[
\frac{1}{\varepsilon }\left( -\tau \right) ^{6\beta }\frac{1}{\left( -\tau
\right) ^{5\beta \left( L+1\right) +\frac{L+1}{2}}}
\]

Therefore we need to compare the terms%
\[
\frac{1}{\left( -\tau \right) }\ \ ,\ \ \frac{1}{\varepsilon ^{2}}\left(
-\tau \right) ^{2\beta }\ \ ,\ \frac{1}{\varepsilon }\frac{\left( -\tau
\right) ^{\beta }}{\left( -\tau \right) ^{\frac{1}{2}}}
\]%
and it readily follows that the three terms are of the same order of
magnitude for the range of times defined by (\ref{S5E6}).

It  follows from the previous sub Section
that we need to rescale the time scale using that
\begin{equation}
\label{timefinal}
\left( -\tau \right) \approx \varepsilon ^{\frac{2}{1+2\beta }}
\end{equation}

We then define a set of new variables as follows%
\begin{eqnarray}
\label{newTime}
-\tau =-\varepsilon ^{\frac{2}{1+2\beta }}\bar{\sigma}\ \ ,\ \ k=\varepsilon
^{\frac{2\beta }{1+2\beta }}p\ \ ,\ \ \xi =\varepsilon ^{\frac{2\beta }{%
1+2\beta }}\zeta \ \ ,\ \ \widehat{F_{L,L}}=\frac{\widehat{H_{L,L}}}{%
\varepsilon ^{\frac{10\beta L}{1+2\beta }+\frac{L}{1+2\beta }}}
\end{eqnarray}
where we used

\[
\widehat{F_{L,L}}=\left( \frac{1}{\left( \varepsilon ^{\frac{2\beta }{%
1+2\beta }}\right) ^{3}}\right) ^{L}\left( \frac{1}{\varepsilon ^{\frac{2}{%
1+2\beta }}}\right) ^{\left( 2\beta +\frac{1}{2}\right) L}\widehat{H_{L,L}}=%
\frac{\widehat{H_{L,L}}}{\varepsilon ^{\frac{10\beta L}{1+2\beta }+\frac{L}{%
1+2\beta }}} 
\]

Then, plugging these formulas into (\ref{S5E7}) we obtain%
\begin{eqnarray}
&&i\partial _{\bar{\sigma}}\widehat{H_{L,L}}\left(
p_{1},p_{2},...,p_{L};\zeta _{1},\zeta _{2},...,\zeta _{L};\bar{\sigma}%
\right)=\nonumber   \\
&&=\frac{1}{2}\left( -\sum_{j=1}^{L}\left\vert p_{j}\right\vert
^{2}+\sum_{j=1}^{L}\left\vert \zeta _{j}\right\vert ^{2}\right) \widehat{%
H_{L,L}}\left( p_{1},p_{2},...,p_{L};\zeta _{1},\zeta _{2},...,\zeta _{L};%
\bar{\sigma}\right) -  \nonumber \\
&&-\frac{1}{\left( 2\pi \right) ^{3}}\sum_{j=1}^{L}\int_{\mathbb{R}^{3}}d%
\bar{p}_{L+1}\int_{\mathbb{R}^{3}}d\bar{\zeta}_{L+1}\nonumber \\
&&\widehat{H_{L+1,L+1}}%
\left( p_{1},p_{2},..,p_{j-1},p_{j}-\bar{p}_{L+1}+\bar{\zeta}%
_{L+1},p_{j+1},...,p_{L},\bar{p}_{L+1};\zeta _{1},\zeta _{2},...,\zeta _{L},%
\bar{\zeta}_{L+1};\bar{\sigma}\right) +  \nonumber \\
&&+\frac{1}{\left( 2\pi \right) ^{3}}\sum_{j=1}^{L}\int_{\mathbb{R}^{3}}d%
\bar{p}_{L+1}\int_{\mathbb{R}^{3}}d\bar{\zeta}_{L+1}\widehat{H_{L+1,L+1}}
\Big( p_{1},p_{2},...,p_{L},\bar{p}_{L+1};\zeta _{1},\zeta _{2},...,\nonumber \\
&&...,\zeta
_{j-1},\zeta _{j}+\bar{p}_{L+1}-\bar{\zeta}_{L+1},p_{j+1},...,p_{L},\bar{%
\zeta}_{L+1};\bar{\sigma}\Big),\hskip 0.3cm \overline \sigma \in \mathbb{R},\, p_i\in \mathbb{R}^3,\, \zeta _i\in \mathbb{R}^3\label{S5E8}
\end{eqnarray}%
where we used that the three tems in (\ref{S5E7} )  have the same order of magnitude for the range of times defined in 
(\ref{timefinal}). Indeed, the sizes of the three terms in  (\ref{S5E7} ) , for times given (\ref{timefinal}) are given respectively by

\begin{eqnarray*}
&&\frac{1}{\varepsilon ^{\frac{10\beta L}{1+2\beta }+\frac{L}{1+2\beta }}}%
\frac{1}{\varepsilon ^{\frac{2}{1+2\beta }}},\,\, \frac{1}{\varepsilon
^{2}}\left( \varepsilon ^{\frac{2\beta }{1+2\beta }}\right) ^{2}\frac{1}{%
\varepsilon ^{\frac{10\beta L}{1+2\beta }+\frac{L}{1+2\beta }}},\,\,
\frac{1}{\varepsilon }\left( \varepsilon ^{\frac{6\beta }{1+2\beta }}\right)
^{2}\frac{1}{\varepsilon ^{\frac{10\beta \left( L+1\right) }{1+2\beta }+%
\frac{\left( L+1\right) }{1+2\beta }}}
\end{eqnarray*}
or equivalently, we need to compare
\begin{eqnarray*}
&&\frac{1}{\varepsilon ^{\frac{2}{1+2\beta }}},\ \frac{1}{\varepsilon ^{2}}%
\left( \varepsilon ^{\frac{2\beta }{1+2\beta }}\right) ^{2}\ ,\ \frac{1}{%
\varepsilon }\left( \varepsilon ^{\frac{6\beta }{1+2\beta }}\right) ^{2}%
\frac{1}{\varepsilon ^{\frac{10\beta }{1+2\beta }+\frac{1}{1+2\beta }}},
\end{eqnarray*}%
and all these terms have the same order of magnitude since  $\frac{2}{1+2\beta }=2-\frac{4\beta }{1+2\beta } $.

This equation, that describes the functions $\widehat{H_{L,L}}$ in the range
of times in which the correlations become important, must be solved with the
following matching conditions, which follow from  
 (\ref{S3E1a}), (\ref{S3E1}) and  (\ref{S6E1}), (\ref{A2}),
\begin{equation}
\widehat{H_{1,1}}\left( p;\zeta ;\bar{\sigma}\right) \sim \left( 2\pi
\right) ^{\frac{3}{2}}\delta \left( p_{1}-\zeta _{1}\right) \frac{1}{\left(
-\sigma \right) ^{2\beta +1}}\Phi \left( \frac{p}{\left( -\bar{\sigma}%
\right) ^{\beta }}\right) \ \ as \ \ \bar{\sigma}\rightarrow -\infty
\label{S5E9a}
\end{equation}%
\begin{equation}
\widehat{H_{L,L}}\left( p_{1},p_{2},...,p_{L};\zeta _{1},\zeta
_{2},...,\zeta _{L};\bar{\sigma}\right) =\left( 2\pi \right) ^{\frac{3}{2}%
L}\sum_{\sigma \in S^{L}}\prod_{j=1}^{L}\left[ \widehat{H_{1,1}}\left(
p_{j};\zeta _{\sigma \left( j\right) };\bar{\sigma}\right) \right] \ \ 
as\ \ \bar{\sigma}\rightarrow -\infty  \label{S5E9b}
\end{equation}

\bigskip

The problem (\ref{S5E8}), (\ref{S5E9a}), (\ref{S5E9b}) describes the onset
of correlations between the different variables near the blow-up of the
kinetic equation. 
It has some analogy and some differences  with the problem obtained in 
\cite{EV0} for the hierarchy of equations satisfied by the Wigner functions 
describing the onset of correlations
and the loss of Markovianity for a system of interacting bosons. The resemblance 
of both systems is natural since both equations are essentially the same for large values of the density functions,  
and it shows that the coherent stages follow a similar mechanism. The main difference is that in this paper 
properties of the correlations are described using a double Fourier transformation, while
in \cite{EV0} the classical Wigner transform was used. As a consequence of
this, the equations obtained in \cite{EV0} have the form of transport
equations, while in the equations obtained here we obtain multiplications by
terms with the form $\left( -\sum_{j=1}^{L}\left\vert k_{j}\right\vert
^{2}+\sum_{j=1}^{L}\left\vert \xi _{j}\right\vert ^{2}\right) .$

\begin{remark}
In the derivation of the WT equation used in this paper that
is based in the analysis of the cumulant equations  \cite{BN}, \cite{ZMR}, \cite{DNPZ}  , the breakdown of the
kinetic approximation becomes visible in the loss of the Markovianity
approximation as well as in the onset of correlations of order one in the
probability measures that describe the solution of the nonlinear Schr\"{o}%
dinger equation. It is relevant to ask what would be the fingerprint of the
singularity if the WT equation is derived using the Duhamel series as it has
been made in the rigorous derivation in \cite{DeHa}. The Duhamel series
approach basically provides a series for the solution of the WT kinetic
equation and one might expect, as usually happens in blow-up problems, that
the onset of the blow-up should be detectable in the asymptotic behavior of
the coefficients of the series. A simple example that suggests how this
could happen is the standard ODE yielding blow-up $\dot{x}=x^{2}$ with
initial value $x\left( 0\right) =x_{0}.$ The series power solution of this
problem is given by $x_{0}\sum_{n=0}^{\infty }\left( x_{0}t\right) ^{n}.$
The terms of this series behave like $\left( x_{0}t\right) ^{n}$ as $%
n\rightarrow \infty $. Notice that the blow-up time for the solution $T=%
\frac{1}{x_{0}}$ is visible in the asymptotic behaviour of the coefficients
of the series, which might be written as $\left( \frac{t}{T}\right) ^{n}$ as 
$n\rightarrow \infty .$ One might expect the singularity of the solutions of
the WT equation to become visible in the behavior of the coefficients of the
series that gives the solution of the kinetic equations. Moreover, the
possible regularizing effects, analogous to the onset of correlations that
we described above, should appear in some of the terms of the elements of
the Duhamel series that dissappear as $\varepsilon \rightarrow 0.$
\end{remark}
 \subsection{ Equivalent Gross-Pitaevski equation for the random field $u$}
 Suppose that $u $ is the solution of the initial value problem (\ref{S1E1}), (\ref{S1E2}), with initial data $u_0$ satisfying (\ref{S1E3}), as described in Section \ref{Correlations} and let us  perform the change of time variable given in 
\begin{eqnarray}
\label{S6E10B}
t =\varepsilon ^{\frac{2}{1+2\beta }-2}\bar{\sigma}.
\end{eqnarray} 
The change (\ref{S6E10B}) comes combining  first  the change $\tau =\varepsilon ^2t$ and then the change of time variable in (\ref{newTime}).
 As it was seen in the previous Section, the time variable $\overline{\sigma }$ is the time scale where correlations start to form if $|\overline \sigma |$ becomes of order one. With some abuse of notation we still denote the function $u$ in the new time variable as $u(\overline{\sigma }, x)$. Then, as it is assumed all along this article,
\[
\mathbb{E}\left[ u(x, \overline{\sigma }) \right] =0\ \ ,\ \ \mathbb{E}\left[
u ^{\ast }\left(x, \overline{\sigma }\right) u\left(y, \overline{\sigma }\right) \right] =N\left(
\overline{\sigma }, x-y\right) \ \ ,\ \ x,y\in \mathbb{R}^{3}.  \label{S1E30}
\]
Suppose  that, the Fourier transform, as defined in (\ref{S1E10}),  of $N(t)$ is such that
\begin{eqnarray*}
\widehat{N(\overline{\sigma })}(k, p)&\equiv&\frac{1}{\left( 2\pi \right) ^{\frac{3}{2}\times 2}}\int_{  \mathbb{R}^{3} }dx \int_{  
\mathbb{R}^{3}  }dy \,e^{  -i\left(k x-py\right) }N(x-y, \overline{\sigma })\\
 &\sim&\delta (k-p)f(p, \overline{\sigma }),\,\,\overline \sigma \to -\infty.
\end{eqnarray*}
where moreover
\[
f(p, \overline{\sigma })=\frac {1} {(-\overline{\sigma })^{2\beta +1}}\Phi (\xi ),\,\,\xi =\frac {p} {(-\overline{\sigma })^\beta }.
\]
Then, using  the inverse Fourier transform,
\begin{eqnarray*}
N(x-y, \overline{\sigma })&=&\frac{1}{\left( 2\pi \right) ^{\frac{3}{2}\times 2}}\int_{  \mathbb{R}^{3} }dk \int_{  
\mathbb{R}^{3}  }dp \,e^{i\left(k x-py\right) } \delta (k-p)\frac {1} {(-\overline{\sigma })^{2\beta +1}}\Phi \left( \frac {p} {(-\overline{\sigma })^\beta }\right)\\
&\sim&\frac{1}{\left( 2\pi \right) ^{\frac{3}{2}}}\int_{  \mathbb{R}^{3} }dp e^{i p (x-y ) }\frac {1} {(-\overline{\sigma })^{2\beta +1}}\Phi \left( \frac {p} {(-\overline{\sigma })^\beta }\right),\,\,\overline \sigma \to -\infty\\
\end{eqnarray*}
where a  change of variables in the integral gives
\begin{eqnarray*}
&&\frac{1}{\left( 2\pi \right) ^{\frac{3}{2}}}\int_{  \mathbb{R}^{3} }dp e^{i p (x-y ) }\frac {1} {(-\overline{\sigma })^{2\beta +1}}\Phi \left( \frac {p} {(-\overline{\sigma })^\beta }\right)=\\
&&=\frac {1} {(-\overline{\sigma })^{2\beta +1-3\beta }}\frac{1}{\left( 2\pi \right) ^{\frac{3}{2}}}\int_{  \mathbb{R}^{3} }d\xi  e^{ i(-\overline{\sigma })^\beta \xi } \Phi (\xi )
=(-\overline{\sigma })^{\beta -1}\Psi ((-\overline{\sigma })^{\beta }(x-y))
\end{eqnarray*}
where $\Psi $ is the inverse Fourier transform of $\Phi $. Then, rescaling  the  original variables $x, y$  as
\begin{eqnarray*}
(-\overline{\sigma })^\beta x=z,\,\,(-\overline{\sigma })^\beta y=w,\,\,u(\overline{\sigma }, x)=(-\overline{\sigma })^{\frac {\beta -1} {2}}U(z, \overline{\sigma })
\end{eqnarray*}
we would have,
\[
\mathbb E\left[u^*(x, t\overline{\sigma }), u(y, \overline{\sigma })) \right]\sim (-\overline{\sigma })^{\beta -1}\Psi ((-\overline{\sigma })^{\beta }(x-y))=(-\overline{\sigma })^{\beta -1}\Psi (z-w),\,\overline{\sigma }\to -\infty\\
\]
and $U$ is then  a Gaussian variable, depending on the variables $t\in \mathbb{R}$ and  $z\in  \mathbb{R}^{3} $ such that
\begin{equation}
\label{equationU}
(-\overline{\sigma })^{\beta -1 }\mathbb E\left[U^*(z, \overline{\sigma }), U(w, \overline{\sigma }) \right]\sim \Psi (z-w),\,\,\overline{\sigma }\to -\infty.
\end{equation}

Plugging  $(-\overline{\sigma })^{\frac {\beta -1} {2}}U((-\overline{\sigma })^\beta x, t)$ in equation (\ref{S1E1}),
\begin{eqnarray*}
i(-\overline{\sigma })^{\frac {\beta -1} {2}}U_t(z, \overline{\sigma })-\frac {i (\beta -1)} {2}(-t)^{\frac {\beta -1} {2}-1}U(z, \overline{\sigma })-i\beta (-\overline{\sigma })^{\frac {\beta -1} {2}-1}z\cdot \nabla_z U(z, \overline{\sigma })=\\
-\frac {1} {2}(-\overline{\sigma })^{\frac {\beta -1} {2}+2\beta }\Delta_zU(z, \overline{\sigma })+ (-\overline{\sigma })^{\frac {3} {2}(\beta -1)}|U(z, \overline{\sigma })|^2U(z, \overline{\sigma })
\end{eqnarray*}
and after multiplication  by $(-\overline{\sigma })^{-\frac {\beta -1} {2}-2\beta }$,
\begin{eqnarray*}
i(-\overline{\sigma })^{-2\beta }U_t(z, \overline{\sigma })-i\left(\frac {(\beta -1)} {2}U(z, \overline{\sigma }) +\beta z\cdot \nabla_z U(z, \overline{\sigma })\right)(-\overline{\sigma })^{-(2\beta +1)}=\\
-\frac {1} {2} \Delta_zU(z, \overline{\sigma })+ (-\overline{\sigma })^{-(\beta +1)}|U(z, \overline{\sigma })|^2U(z, \overline{\sigma }).
\end{eqnarray*}
The change of time variable
$$
\frac {d\overline{\sigma } } {d\overline{\tau  } }=(-\overline{\sigma })^{-2\beta },\,\,V(z, \overline{\tau })=U(z, \overline{\sigma })
$$
then $\overline{\sigma }=-(-(2\beta +1)\overline{\tau  } )^{\frac {1} {2\beta +1}}$, 
makes $V _{ \overline{\tau  }  } = (-\overline{\sigma   } )^{-2\beta }U_t(z, \overline{\sigma   } )$ and,
\begin{eqnarray}
&&iV_\tau (z, \overline{\tau  }  )-i\left(\frac {(\beta -1)} {2} V(z, \overline{\tau  } )+\beta z\cdot \nabla_z V(z, \overline{\tau  } )\right)(-(2\beta +1)\overline{\tau  }  )^{- 1}=\nonumber\\
&&\hskip 3cm -\frac {1} {2} \Delta_zV(z, \overline{\tau  } )+ (-(2\beta +1)\overline{\tau  }  )^{- \frac {\beta +1} {2\beta +1}}|V(z, \overline{\tau  } )|^2V(z, \overline{\tau  } ).\label{equationV}
\label{S6E2}
\end{eqnarray}
with
\begin{eqnarray}
&&(-(2\beta +1)\overline{\tau  } )^{\frac {\beta -1} {2\beta +1}}\mathbb E\left[V^*(z, \overline{\tau  }), V(w, \overline{\tau  }) \right]\sim \Psi (z-w),\,\,\overline{\tau  }\to -\infty.
\end{eqnarray}
We see that as $\overline \tau \to -\infty$, the leading term of  equation (\ref{S6E2}) is the linear Schr\"odinger equation as it may be expected,  since in the kinetic regime the solution of the Schr\"odinger equation solves  the linear Schr\"odinger equation to leading order, i.e. 
$iV_\tau (z, \overline{\tau  } )=-\frac {1} {2} \Delta_zV(z, \overline{\tau  } )$.

\section{Non-Markovian problem. Leading order as $\bar{\protect%
\sigma}\rightarrow -\infty .$}
\setcounter{equation}{0}
\setcounter{theorem}{0}

In this Section we discuss an equation that appears in the deduction of the kinetic equation starting from the Schr\"odinger equation, The equation   is non-Markovian and it contains some of the basic ingredients
that yield in the limit the Markovian WT kinetic equation. We discuss it here since it has an independent interest (it appears for example in the  theory of wave  turbulence for  interactions of ocean waves and wind c.f. \cite{J}).  The non Markovian equation (\ref {S5E2}) in its simplest formulation is the following one
\begin{eqnarray*}
\partial _{t}f_{1}\left( k_{1},t\right) &=&\frac{1}{\pi }\int_{\left( 
\mathbb{R}^{3}\right) ^{3}}d\eta _{2}d\eta _{3}d\eta _{4}\delta \left( \eta
_{3}+\eta _{4}-k_{1}-\eta _{2}\right) \cdot \\
&&\int_{-\infty }^{t}ds\cos \left(\left( t-s\right) \left( \left\vert
k_{1}\right\vert ^{2}+\left\vert \eta _{2}\right\vert ^{2}-\left\vert \eta
_{3}\right\vert ^{2}-\left\vert \eta _{4}\right\vert ^{2}\right)\right) \mathbb{K}%
\left[ f\right] \left( k_{1},\eta _{2};\eta _{3},\eta _{4},;s\right)
\end{eqnarray*}%
where 
\begin{eqnarray*}
&&\mathbb{K}\left[ f\right] \left( k_{1},\eta _{2};\eta _{3},\eta
_{4}; s\right) =\Big[ \left( f\left( k_{1},s\right) +f\left( k_{2},s\right) \right)
f\left( \xi _{1},s\right) f\left( \xi _{2},s\right)-\nonumber \\ 
&&\hskip 3.7cm -f\left( k_{1},s\right)
f\left( k_{2},s\right) \left( f\left( \xi _{2},s\right) +f\left( \xi
_{1},s\right) \right) \Big] \delta \left( \xi _{1}+\xi
_{2}-k_{1}-k_{2}\right)
\end{eqnarray*}
Notice that we have  changed the  variables $n$ to $f$ in formula (\ref {S5E2}). After the change of variables $\tau =t\varepsilon ^2$ and letting $\varepsilon \to 0$ we formally deduce the Markovian equation (\ref {S5E4}), replacing $n$ by $f$.\\ \\ \\

\textbf{Acknowledgements} JJLV gratefully acknowledges the support by the Deutsche Forschungsgemeinschaft (DFG) through the collaborative research centre "Analysis of criticality: from complex phenomena to models and estimates" (CRC 1720, Project-ID 539309657).  JJLV is funded also  by the DFG under Germany’s Excellence Strategy
EXC2047/2-390685813. \\
The research of ME is supported by grant PID2023-146872OB-I00 of MINECO. \\
The funders had no role in study design, analysis, decision to publish, or preparation of the manuscript.\\ \\
\textbf{Data availability statement.} 
Data sharing not applicable to this article as no data sets were generated or analysed during the
currents tudy.\\ \\
\textbf{Conflicts of Interest.}
The authors declare no conflicts of interest.

M. Escobedo: Departamento de Matemáticas, Universidad del País Vasco. Bilbao, Apartado 644, 48080, Spain. Email: miguel.escobedo@ehu.es

J. L. L. Velázquez: Institute of Applied Mathematics, University of Bonn,
Endenicher Allee 60, 53115 Bonn, Germany. Email:  jlopezve@uni-bonn.de

\end{document}